\documentclass[onefignum,onetabnum]{siamart190516}

\usepackage{lipsum}
\usepackage{graphicx}
\usepackage{epstopdf}
\usepackage{algorithmic}
\usepackage{mathrsfs}
\usepackage{epsfig, graphicx}
\usepackage{latexsym,amsfonts,amsbsy,amssymb}
\usepackage{comment}

\newtheorem{problem}{Problem}
\newtheorem{example}{Example}
\graphicspath{{include-figure-AA/}}
\ifpdf
  \DeclareGraphicsExtensions{.eps,.pdf,.png,.jpg}
\else
  \DeclareGraphicsExtensions{.eps}
\fi

\newcommand{\ds}{\displaystyle}
\newsiamremark{remark}{Remark}
\newsiamremark{hypothesis}{Hypothesis}
\crefname{hypothesis}{Hypothesis}{Hypotheses}
\newsiamthm{claim}{Claim}

\headers{Anderson acceleration as a Krylov method}{Hans De Sterck, Yunhui He, and Oliver A. Krzysik}

\title{Anderson Acceleration as a Krylov Method with Application to Asymptotic Convergence Analysis\thanks{Submitted to the editors on Sept 29, 2021; this version resubmitted on Nov 22, 2022.
\funding{This work was funded by NSERC of Canada.}}}

\author{Hans De Sterck\thanks{Department of Applied Mathematics
University of Waterloo, 200 University Ave W, Waterloo, ON N2L 3G1, Canada
  (\email{hdesterck@uwaterloo.ca}, \email{yunhui.he@uwaterloo.ca}, \email{okrzysik@uwaterloo.ca}).}
  \and Yunhui He\footnotemark[2] \and Oliver A. Krzysik\footnotemark[2]}

\usepackage{amsopn}
\DeclareMathOperator{\diag}{diag}

\newcommand{\btxt}[1]{#1}
\newcommand{\rtxt}[1]{#1}

\begin{document}

\maketitle

\begin{abstract}
Anderson acceleration (AA) is widely used for accelerating the convergence of nonlinear fixed-point methods $x_{k+1}=q(x_{k})$, $x_k \in \mathbb{R}^n$, \btxt{but little is known about how to quantify the convergence acceleration provided by AA.
As a roadway towards gaining more understanding of convergence acceleration by AA, we study AA($m$), i.e., Anderson acceleration with finite window size $m$, applied to the case of linear fixed-point iterations $x_{k+1}=M x_{k}+b$.
We write AA($m$) as a Krylov method with polynomial residual update formulas, and derive $(m+2)$-term recurrence relations for the AA($m$) polynomials.}
Writing AA($m$) as a Krylov method immediately implies that $k$ iterations of AA($m$) cannot produce a smaller residual than $k$ iterations of GMRES without restart (but without implying anything about the relative convergence speed of (windowed) AA($m$) versus restarted GMRES($m$)).
We find that the AA($m$) residual polynomials observe a periodic memory effect where increasing powers of the error iteration matrix $M$ act on the initial residual as the iteration number increases. We derive several further results based on these polynomial residual update formulas, including orthogonality relations, a lower bound on the AA(1) acceleration coefficient $\beta_k$, and explicit nonlinear recursions for the AA(1) residuals and residual polynomials that do not include the acceleration coefficient $\beta_k$.
\rtxt{Using these recurrence relations we also derive new residual convergence bounds for AA(1) in the linear case, demonstrating how the per-iteration residual reduction $\Vert r_{k+1} \Vert / \Vert r_{k} \Vert$ depends strongly on the residual reduction in the previous iteration and on the angle between the prior residual vectors $r_k$ and $r_{k-1}$.}
We apply these results to study the influence of the initial guess on the asymptotic convergence factor of AA(1), \btxt{and to study AA(1) residual convergence patterns.}
\end{abstract}




\section{Introduction}\label{sec:intro}
%

In this paper we consider the following nonlinear  iteration  with window size $m$,
\begin{equation}\label{eq:AA-iteration}
  x_{k+1}= q(x_k) + \sum_{i=1}^{\min(k,m)}\beta_{i}^{(k)}(q(x_k)-q(x_{k-i})), \qquad k=0,1,2,\ldots,
\end{equation}
which aims to accelerate fixed-point (FP) iterations of the type
\begin{equation}\label{eq:fixed-point}
  x_{k+1}=q(x_{k}), \quad x_k\in\mathbb{R}^n,  k=0,1,2,\ldots,
\end{equation}
with fixed point $x^*=q(x^*)$.
The coefficients $\beta_{i}^{(k)}$ in \cref{eq:AA-iteration} are determined by solving an optimization
problem in each step $k$ that minimizes a linearized residual in the new iterate $x_{k+1}$. Method
\cref{eq:AA-iteration} is known as \emph{Anderson acceleration (AA)} \cite{anderson1965iterative}.
Specifically, AA($m$), with window size $m$, solves in every iteration the optimization problem
\begin{equation}\label{eq:Andersonbetas}
	\min_{\{\beta_i^{(k)} \}} \bigg\| r(x_k) + \sum_{i=1}^{\min(k,m)} \beta_i^{(k)} ( r(x_k) - r(x_{k-i})) \bigg\|^2
\end{equation}
with up to $m$ variables. Here, the residuals $r(x)$ of the fixed-point iteration are defined by
\begin{equation}\label{eq:resid}
	r(x)=x-q(x).
\end{equation}

The $2$-norm is normally used in the optimization problem \cref{eq:Andersonbetas}, and throughout this paper $\| \cdot \|$ will denote the 2-norm.  When $m=0$, \cref{eq:AA-iteration} is reduced to the fixed-point iteration \cref{eq:fixed-point}.
\rtxt{Furthermore, the AA($m$) iteration \cref{eq:AA-iteration} we discuss in this paper is the original Anderson acceleration method as proposed in \cite{anderson1965iterative},
which is also the main focus of theoretical work by Walker and Ni in \cite{walker2011anderson} and Toth and Kelley in \cite{toth2015}. As in \cite{walker2011anderson,toth2015}, we set the so-called mixing or relaxation parameter from \cite{anderson1965iterative} equal to one. Note that the AA($m$) iteration \cref{eq:AA-iteration} is different from the Type-I methods in the so-called ``Anderson family'' of methods proposed by Fang and Saad in \cite{fang2009two}; the original Anderson acceleration method \cref{eq:AA-iteration} we discuss here is a Type-II method in the classification of \cite{fang2009two}.
} 

Anderson acceleration dates back to the 1960s \cite{anderson1965iterative} and is widely used in computational science to accelerate iteration methods which converge slowly or do not converge.  It has gained significant new interest over the past decade  both in terms of theoretical developments and applications. Early on, most  research focused on the application of AA without  theoretical convergence analysis. For example, \cite{lott2012accelerated} examines  the effectiveness of AA applied to modified Picard iteration for nonlinear problems arising in variably saturated flow modeling. Similarly, \cite{ho2017accelerating} applies AA to the Uzawa algorithm for the solution of saddle-point problems, 
\rtxt{\cite{henderson2019damped} applies AA to expectation-maximization,} and \cite{ni2009anderson}  investigates the self-consistent field  method accelerated by  AA for electronic structure computations. For more applications, we refer to \cite{an2017anderson,brune2015composing,fang2009two,lipnikov2013anderson}. Anderson acceleration is also closely related to the nonlinear GMRES method from  \cite{oosterlee2000krylov,sterck2012nonlinear,sterck2013steepest}.

\btxt{
AA is often very effective at accelerating convergence, but little is understood about how to quantify or predict the convergence improvements provided by AA. 
Motivated by a desire to improve our understanding of AA convergence acceleration for finite window size $m$, our aim in this paper is to study AA($m$) applied to the linear case, that is, with the iteration function $q(x)$ in (\ref{eq:fixed-point}) given by
\begin{equation}\label{eq:linear-fixed-point}
  q(x) = Mx + b,
\end{equation}
where the fixed point satisfies $Ax^*=b$ with $A =I-M$. We will assume that $A$ is nonsingular and we exclude the trivial case that $A=I$ and $M=0$.

Our approach will be to write AA($m$) in the linear case as a Krylov method, and we will derive recurrence relations for the AA($m$) residual polynomials that, to the best of our knowledge, have not appeared in the literature before. This will allow us to derive several new properties of AA($m$) iterations in the linear case for general $m$, and, for AA(1), new results on residual convergence bounds and on the influence of the initial guess on the asymptotic convergence factor.
}



\subsection{AA($m$) acceleration coefficients}

\btxt{The AA($m$) acceleration coefficients $\beta_i^{(k)}$ in iteration (\ref{eq:AA-iteration}) satisfy the following relations.}
Assume that $k\ge m$. Define $r_k =x_k-q(x_k)$ and
\begin{equation}\label{eq:beta-vector-form}
  \boldsymbol{\beta}^{(k)} =\begin{bmatrix} \beta_1^{(k)} \\ \vdots \\ \beta_m^{(k)} \end{bmatrix}, \quad
  R_k = \begin{bmatrix} r_k-r_{k-1} &  r_k-r_{k-2} & \ldots & r_k-r_{k-m} \end{bmatrix}.
\end{equation}
Then, using the 2-norm in \cref{eq:Andersonbetas}, the solution of the least-squares problem \cref{eq:Andersonbetas} can be written as
\begin{equation}\label{eq:AAm-beta-form}
 \boldsymbol{\beta}^{(k)}  = -(R_k^TR_k)^{-1} R_k^Tr_k,
\end{equation}
if $ R_{k}^TR_k$ is invertible. Otherwise, we can take
\begin{equation}\label{eq:AAm-beta-form-pseudo}
 \boldsymbol{\beta}^{(k)}  = -R_k^{\dag}r_k,
\end{equation}
where $R_k^{\dag}$ is the pseudo-inverse of $R_k$, and $ \boldsymbol{\beta}^{(k)}$ corresponds to the minimum-norm solution of the least-squares problem.  We note that
%
 $R_k^{\dagger}= \big(R_k^TR_k\big)^{\dagger} R_k^T$.

Particularly, when $m=1$ and $r_k\neq r_{k-1}$,
\begin{equation}\label{eq:AA-1-step-beta}
  \beta_1^{(k)} = \frac{-r_k^T(r_k-r_{k-1})}{\|r_k-r_{k-1}\|^2}=:\beta_k.
\end{equation}
When $r_k= r_{k-1}$, we take  $\beta_k=0$ according to \cref{eq:AAm-beta-form-pseudo}.

\btxt{
As $x_k$ approaches $x^*$, it is important to solve the least-squares problem (\ref{eq:Andersonbetas}) in a numerically stable manner, because $R_k$ in (\ref{eq:beta-vector-form}) may become close to rank-deficient. This topic has been discussed extensively in the AA literature, and several options are available.} \rtxt{The normal equations should not be used because this squares the condition number.} \btxt{Two suitable approaches are discussed in \cite{fang2009two}, including using a rank-revealing $QR$ decomposition or a truncated singular value decomposition. Matlab's $QR$ solver or pseudo-inverse solver use similar approaches \cite{moler2004numerical}.
As an alternative, \cite{walker2011anderson} proposes to use a standard $QR$ factor-updating technique that can save some work, combined with an approach that adaptively changes $m$ to drop columns if $R_k$ becomes too ill-conditioned. In the context of parallel implementations, \cite{lockhart2022performance} discusses further versions of $QR$ factor-updating techniques that are optimized to reduce global communication cost using low-synchronization variants of classical and modified Gram-Schmidt.
}

\subsection{AA($m$) convergence}

It is only recently that the first results have been obtained on the convergence  of AA. In \cite{toth2015} it was shown that AA($m$) is
locally $r$-linearly convergent under the assumptions that $q(x)$ is contractive and the AA coefficients remain bounded. However, this work does not prove that AA actually improves the convergence speed. Further progress was made in \cite{evans2020proof}, showing that, to first order, the convergence gain provided by AA in
step $k$ is quantified by a factor $\theta_k \leq  1$ that is the ratio of the square root of the
optimal value defined in \cref{eq:Andersonbetas} to $\| r(x_k)\|_2$. 
However, it is not clear how this result may be used to quantify the asymptotic gain in convergence speed,
since $\theta_k$ does not appear to observe an upper bound $<1$ as $k$ increases.
\rtxt{In \cite{pollock2021anderson} the authors show refined residual bounds that include higher-order terms
and rely on sufficient linear independence in the least-squares problems, which they ensure using a safeguarding strategy.}
In \cite{desterck2020,wang2020}, the authors consider stationary versions of the AA($m$) iteration where constant iteration coefficients $\beta_i$ are chosen in such a way that they minimize the $r$-linear convergence factor of the stationary AA($m$) iteration, given knowledge of $q'(x^*)$.
This provides insight into how the optimal stationary AA($m$) iteration improves the asymptotic convergence of the FP method by reducing the spectral radius of $q'(x^*)$, where $q(x)$ can be interpreted as a nonlinear preconditioner for AA($m$), see also \cite{brune2015composing}.

\rtxt{In the linear case, it has long been known that, in the case of infinite window size $m=\infty$, AA($m$) and the related nonlinear GMRES method of \cite{oosterlee2000krylov} are equivalent in a certain sense to GMRES \cite{saad1986gmres}. In \cite{walker2011anderson} the following precise equivalence result was obtained: if we apply GMRES to a nonsingular linear system $(I-M)x=b$ with initial guess $x_0$ and assume that GMRES residuals strictly decrease, $\|r^G_{k+1}\|<\|r^G_{k}\|$ for all $k$, then the $k+1$st AA iterate with infinite window size, starting from the same initial guess, can be obtained from the $k$th GMRES iterate as $x^{AA}_{k+1}=q(x^G_k) = M x^G_k + b$, which can also be expressed in terms of the error $e_k=x-x_k$ as $e^{AA}_{k+1}=Me^G_k$. 
Clearly, this so-called \emph{essential equivalence} of GMRES and AA for infinite window size, also implies that all the results of GMRES convergence theory directly apply to AA convergence when the window size is infinite, as long as the GMRES residuals do not stagnate. However, when GMRES does stagnate, i.e., $r^G_{k+1}=r^G_{k}\ne0$ for some $k$, the equivalence between AA and GMRES breaks down. In those cases, GMRES would still converge in at most $n$ steps, but AA stalls at an iterate with a nonzero residual for all further iterations \cite{walker2011anderson}. This stagnation behavior of AA with infinite window size was fully characterized in \cite{potra2013characterization}.

However, much less is known about the convergence of AA($m$) applied to linear problems when the window size is finite, which is the case considered in this paper.
Some limited results are known about the convergence of restarted GMRES($m$), but these results cannot be applied to AA($m$) since AA($m$) uses a windowing approach instead of restarting. Just like restarted GMRES($m$), AA($m$) does not have the finite-iteration convergence property of GMRES.
In terms of stagnation behavior, it is easy to see that AA($m$) with finite window size applied to linear problems can never stagnate when $\|M\|<1$, since \cite{toth2015}
shows that the AA($m$) residuals satisfy $\|r_{k+1}\| \le \|M\| \, \|r_k\|$, which implies $\|r_{k+1}\| < \|r_k\|$ when $\|M\|<1$. We are not aware of results on the stagnation behavior of AA($m$) with finite window size when $\|M\|\ge1$. We give an example in this paper where we apply AA(1) to a linear problem with $\|M\|>1$ and obtain $r_2=r_1\ne0$, but subsequent $r_k$ obtained by AA(1) converge to zero.
}


\subsection{AA($m$) in the linear case}

In this paper, we focus on  AA($m$) applied to the linear case, with iteration function (\ref{eq:linear-fixed-point}).
We are interested in exploring polynomial update formulas for the residual of AA($m$), and deriving recurrence relations for the AA($m$) polynomials. Recall that the order-s Krylov subspace generated by a matrix $T$ and a vector $v$ is the linear subspace spanned by the images of $v$ under the first $s$ powers of $T$, that is
\begin{equation*}
  \mathcal{K}_s (T,v) =\Big\{v, Tv, \ldots, T^{s-1} v\Big\}.
\end{equation*}
The specific case of AA(1) in \cref{eq:AA-iteration} reads, for $k\ge 1$,
\begin{equation}\label{eq:anderson-1-step}
  x_{k+1} = (1+\beta_k) q(x_k) -\beta_k q(x_{k-1}).
\end{equation}
In previous work, \cite{kindermann2021optimal,liu2018parametrized,niu2020momentum} have interpreted Nesterov acceleration, which is similar in form to AA(1) but with a prescribed sequence of acceleration coefficients $\beta_k$, as a Krylov method. Inspired by this, we investigate in this paper  how AA($m$) for linear problems, with  $\boldsymbol{\beta}^{(k)}$ given by \cref{eq:AAm-beta-form}, relates to Krylov methods. Following  \cite{kindermann2021optimal,liu2018parametrized,niu2020momentum}, this is easy to see for AA(1) applied to \cref{eq:linear-fixed-point}, as we now explain.  Given $x_0$, let $ x_1=q(x_0)$.  The residual $r_{k+1}$ generated by AA(1) iteration \cref{eq:anderson-1-step} satisfies
\begin{align}
   r_{k+1}&= x_{k+1}-(Mx_{k+1}+b)= Ax_{k+1}-b,\nonumber\\
           & = A\big((1+\beta_k) (Mx_k+b) -\beta_k(Mx_{k-1}+b)\big)-b,\nonumber\\
          &= (1+  \beta_k)AMx_k-  \beta_kAMx_{k-1}+Ab-b,\nonumber \\
         &=(1+  \beta_k)M(Ax_k-b) - \beta_k M(Ax_{k-1}-b)+Mb+Ab-b,\nonumber\\
         &=(1+  \beta_k)Mr_k -  \beta_k M r_{k-1}.\label{pro:AA1-rk+1=Mg_k}
\end{align}

This gives the following expressions for the first few $r_k$:
\begin{align}
 r_1&=M r_0,\nonumber\\
 r_2&=\big((1+\beta_1)M^2-\beta_1 M\big) r_0,\nonumber\\
 r_3&=\big((1+\beta_2)(1+\beta_1)M^3-((1+\beta_2)\beta_1+\beta_2) M^2\big)r_0,\nonumber\\
 r_4&=\big((1+\beta_3)(1+\beta_2)(1+\beta_1)M^4-((1+\beta_3)(1+\beta_2)\beta_1\nonumber\\
  &\quad  +(1+\beta_3)\beta_2+\beta_3(1+\beta_1)) M^3+\beta_3\beta_1M^2\big)r_0. \label{pro:AA1-rks}
\end{align}
Clearly, AA(1) is a Krylov space method.  We will investigate in this paper how this  extends  to AA($m$) and allows to derive new properties of the AA($m$) iteration in the linear case.
%

\btxt{Finally, it is worth noting that, in practice, AA($m$) is almost always used in the form of \cref{eq:AA-iteration}, with a single initial guess $x_0$ that is used as the starting point for accelerating FP method (\ref{eq:fixed-point}) using a window size that gradually increases from 1 to $m$ over the first $m$ steps of the AA($m$) iteration. Note, also, that $x_1=q(x_0)$ in iteration (\ref{eq:AA-iteration}). However, for our theoretical derivations it will sometimes be useful to consider a version of  AA($m$) iteration \cref{eq:AA-iteration} that starts with a general initial guess $\{x_0,x_1,\ldots, x_m\}$ and uses window size $m$ from the first iteration:
\begin{equation}\label{eq:general-guess-AAm}
   x_{k+1}= q(x_k) + \sum_{i=1}^{m}\beta_{i}^{(k)}(q(x_k)-q(x_{k-i})), \qquad k=m,m+1,\ldots.
\end{equation}
We emphasize, however, that we only use this approach with general initial guess $\{x_0,x_1,\ldots, x_m\}$ as an intermediate step in our theoretical derivations for the regular AA($m$) iteration of (\ref{eq:AA-iteration}) with single initial guess $x_0$, and the version with general initial guess is normally not used in practical computer implementations.
}

The rest of this paper is organized as follows. \btxt{In \Cref{sec:AAm-krylov}, we write AA($m$) as a Krylov method and derive polynomial residual update formulas. We derive recurrence relations for the AA($m$) polynomials, and several further results including orthogonality relations.} In \Cref{sec:AA(1)-Krylov}, we focus specifically on AA(1). We obtain  a lower bound on the AA(1) acceleration coefficient and explicit nonlinear recursions for the AA(1) residuals and residual polynomials. \btxt{Using these recurrence relations, we prove results on the influence of the initial guess on the asymptotic convergence factor of AA(1) in the linear case, and we derive new residual convergence bounds for AA(1).
Numerical results are presented in \Cref{sec:Numerical-result} to illustrate the theoretical findings and how they relate to AA($m$) convergence patterns.} Finally, we draw conclusions in \Cref{sec:con}.

\section{AA($m$) as a Krylov space method}\label{sec:AAm-krylov}
In this section, we consider AA($m$) with \rtxt{finite window size} $m>1$ applied to  linear problems and establish links with Krylov methods. We will discuss some specific aspects of the case that $m=1$ in \Cref{sec:AA(1)-Krylov}. It has been known for a long time that AA($m$) is essentially equivalent to GMRES when the window size $m$ is taken as infinite \cite{oosterlee2000krylov,walker2011anderson}. When this is not the case, AA($m$) is a limited-memory version of AA($\infty$) with a moving window, similar to how
the commonly used restarted GMRES($m$) is a memory-economic version of GMRES -- but GMRES($m$) uses restarts rather than a moving window. 
This section establishes  results on AA($m$) with finite $m$ viewed as a Krylov space method.

Assuming $k\ge m$, AA($m$) iteration (\ref{eq:AA-iteration}) can be rewritten as
\begin{equation}\label{eq:AA(m)-rewrite-form}
  x_{k+1}=\Big(1+ \sum_{i=1}^{m}\beta_i^{(k)}\Big)q(x_k)- \sum_{i=1}^{m}\beta_i^{(k)}q(x_{k-i}),\quad k=m,m+1,\ldots\,\,  .
\end{equation}
In the linear case, we  obtain
\begin{equation}\label{eq:xk-matrix-form-AAm}
  x_{k+1}=\Big(1+ \sum_{i=1}^{m}\beta_i^{(k)}\Big)Mx_k- \sum_{i=1}^{m}\beta_i^{(k)}Mx_{k-i}+b,\quad k=m,m+1,\ldots\,\,  .
\end{equation}
This leads to the following update formula for the AA($m$) residuals:
\begin{proposition}\label{prop:rk-recursive-AAm-random-guess}
The residuals $r_{k+1}$ generated by AA($m$) iteration (\ref{eq:AA-iteration}) applied to linear iteration \cref{eq:linear-fixed-point} with $k\ge m$ satisfy
 \begin{equation}\label{eq:m+1-term-rk-linear-random-guess}
   r_{k+1} = \Big(1+ \sum_{i=1}^{m}\beta_i^{(k)}\Big)M r_k- \sum_{i=1}^{m}\beta_i^{(k)} M r_{k-i},\quad k\geq m.
 \end{equation}
\end{proposition}
\begin{proof}
 From \cref{eq:xk-matrix-form-AAm} and $AM=MA$ we have, for $k\geq m$, 
 \begin{align*}
 Ax_{k+1} &= \Big(1+ \sum_{i=1}^{m}\beta_i^{(k)}\Big)AMx_k- \sum_{i=1}^{m}\beta_i^{(k)}AMx_{k-i}+Ab,\\
         &=\Big(1+ \sum_{i=1}^{m}\beta_i^{(k)}\Big)MAx_k- \sum_{i=1}^{m}\beta_i^{(k)}MAx_{k-i}+Ab,\\
         &=\Big(1+ \sum_{i=1}^{m}\beta_i^{(k)}\Big)M(Ax_k-b) - \sum_{i=1}^{m}\beta_i^{(k)}M(Ax_{k-i}-b)+Mb+Ab,\\
         &=\Big(1+ \sum_{i=1}^{m}\beta_i^{(k)}\Big)Mr_k - \sum_{i=1}^{m}\beta_i^{(k)}M r_{k-i}+ b.
 \end{align*}
 Thus, $r_{k+1}=Ax_{k+1}-b=\displaystyle \Big(1+ \sum_{i=1}^{m}\beta_i^{(k)}\Big)Mr_k - \sum_{i=1}^{m}\beta_i^{(k)}Mr_{k-i}$.
\end{proof}

\btxt{Direct calculations then lead from expression (\ref{eq:m+1-term-rk-linear-random-guess}) to the following result, establishing that AA($m$) is a Krylov method in the linear case, that is, $r_{k+1}\in \mathcal{K}_s (M,r_0)$, and deriving ($m+2$)-term recurrence relations for the AA($m$) residual update polynomials:}
\begin{proposition}\label{pro:polynomial-form-AAm-AAj}
AA($m$) iteration \cref{eq:AA-iteration} applied to linear iteration \cref{eq:linear-fixed-point} is a  Krylov method, with the residuals   given by
\begin{equation}\label{eq:rk-polynomial-form-M-AAm-AAj}
    r_{k+1} =p_{k+1}(M)\,r_0,\quad k\geq 0,
\end{equation}
where $p_{k+1}(\lambda)$ is a polynomial of degree at most $k+1$, satisfying  the  following recurrence relations:
\noindent  When  k=0,
\begin{equation*}
p_1(\lambda) =\lambda.
\end{equation*}
When $1\leq k< m$,
\begin{equation}\label{polynomial-sequence-AAj}
 p_{k+1}(\lambda)=\Big(1+\sum_{i=1}^{k}\beta_i^{(k)}\Big)\lambda p_{k}(\lambda) - \sum_{i=1}^{k}\beta_i^{(k)} \lambda p_{k-i}(\lambda) \,\,\,
\text{where}\,\,\, p_0(\lambda)=1.
\end{equation}
When  $k\geq m$,
\begin{equation*}
 p_{k+1}(\lambda)=\Big(1+\sum_{i=1}^{m}\beta_i^{(k)}\Big)\lambda p_{k}(\lambda) - \sum_{i=1}^{m}\beta_i^{(k)} \lambda p_{k-i}(\lambda).
\end{equation*}
Moreover, $p_k(1)=1$ and $p_k(0)= 0, k=1,2,\ldots \,\, .$
\end{proposition}
\btxt{
\begin{proof}
As explained in \cref{app:AAm-random-guesses}, the result follows from direct calculations using (\ref{eq:m+1-term-rk-linear-random-guess}) as a starting point. To show this, \cref{app:AAm-random-guesses} first establishes a result analogous to \cref{pro:polynomial-form-AAm-AAj} for the more general case of iteration (\ref{eq:general-guess-AAm}) with general initial guess $\{x_0,x_1,\ldots, x_m\}$, in \cref{pro:polynomial-form-AAm-random-guess}. The result of \cref{pro:polynomial-form-AAm-AAj} then easily follows from \cref{pro:polynomial-form-AAm-random-guess}.
\end{proof}
}

Using \cref{pro:polynomial-form-AAm-AAj}, it can be shown easily that there is a periodic pattern with  period $m+1$ in the AA($m$) polynomials:
\begin{proposition}\label{pro:AAm-polynomial-power-initial-AAj}
The residuals of AA($m$) iteration \cref{eq:AA-iteration} applied to linear iteration \cref{eq:linear-fixed-point} satisfy
\begin{equation}\label{eq:period}
r_{s(m+1)+i}=M^{s+1} g_{s(m+1)+i-(s+1)}(M)\,r_0, \quad s=0,1,2,\ldots, \quad i =1,\ldots, m+1,
\end{equation}
where $g_{s(m+1)+i-(s+1)}(\lambda)$ is a polynomial of degree at most $s(m+1)+i-(s+1)$ and $\lambda^{s+1} g_{s(m+1)+i-(s+1)}(\lambda)=p_{s(m+1)+i}(\lambda)$ from \cref{eq:rk-polynomial-form-M-AAm-AAj}.
\end{proposition}
Expression \cref{eq:period} indicates that every $m+1$ iterations, the power of $M$ in the right-hand side of \cref{eq:period} increases by 1. 
We refer to this property as the AA($m$) iterations possessing a periodic \emph{memory effect}. 
Expression \cref{eq:period}, thus, reveals that AA($m$) provides acceleration as a result of two multiplicative effects: damping of error modes by $M^{s}$ as in the FP method, augmented by polynomial acceleration. 
As a result of the windowing in AA($m$), the effect of the FP iteration is, thus, retained in an accumulative fashion as $k$ increases.
This is in contrast to restarted GMRES($m$), where there is no such cumulative damping since the iteration is fully restarted every $m$ steps.

Next, the following result follows directly from \cref{eq:rk-polynomial-form-M-AAm-AAj} in \cref{pro:polynomial-form-AAm-AAj}, since GMRES determines the optimal degree-$k$ polynomial with $p_k(1)=1$.
\begin{proposition}\label{pro:k-AAm-worse-GMRES}
When applied to linear iteration \cref{eq:linear-fixed-point},
$k$ steps of AA($m$) cannot produce a residual that is smaller in the 2-norm than the residual obtained by GMRES($k$) applied to
the corresponding linear system.
\end{proposition}
\begin{remark}
\cref{pro:k-AAm-worse-GMRES} does not imply anything about the relative convergence speed of (windowed) AA($m$) versus restarted GMRES($m$) in the linear case.  To our knowledge, no general theoretical results exist on this topic. Note that it can be expected that AA($m$) would often converge faster than GMRES($m$), since AA($m$) uses $m+1$  previous iterates in each step, while GMRES($m$) uses about half that number per step on average. Some numerical results in the literature confirm this, see, e.g., Figure 9 in \cite{LinearacAA}.
\end{remark}

Finally, we show an orthogonality property of the AA($m$) residual $r_{k+1}$ with respect to $R_k$ of \cref{eq:beta-vector-form}:
\begin{proposition}\label{AAm-rk-M-orthogonal}
 Assume that $M$ is invertible. Then the residuals $r_{k+1}$ generated by AA($m$) iteration \cref{eq:AA-iteration} applied to linear iteration \cref{eq:linear-fixed-point} with $k \ge m$ satisfy
\begin{equation*}
  R_k^T M^{-1}r_{k+1}=0,
\end{equation*}
where $R_k$ is defined in \cref{eq:beta-vector-form}.
\end{proposition}
\begin{proof}
Recall $\boldsymbol{\beta}^{(k)}$ in \cref{eq:AAm-beta-form-pseudo}. We have
\begin{equation*}
 (R_k^TR_k)\boldsymbol{\beta}^{(k)} + R_k^Tr_k=0,
\end{equation*}
which leads to
\begin{equation*}
  R_k^T(R_k\boldsymbol{\beta}^{(k)}+r_k)=0,
\end{equation*}
or, by \cref{eq:beta-vector-form},
\begin{equation}\label{eq:rk-Rk-dot-product}
 R_k^T \left( \Big(1+ \sum_{i=1}^{m}\beta_i^{(k)}\Big)r_k - \sum_{i=1}^{m}\beta_i^{(k)}r_{k-i}\right )=0.
\end{equation}
Using \cref{prop:rk-recursive-AAm-random-guess} this gives
\begin{equation*}
 R_k^T  M^{-1}M\left(\Big(1+ \sum_{i=1}^{m}\beta_i^{(k)}\Big)r_k - \sum_{i=1}^{m}\beta_i^{(k)}r_{k-i}\right)=R_k^T M^{-1} r_{k+1}=0,
\end{equation*}
which completes the proof.
\end{proof}

The previous result is used to show how the residuals $r_{k+1}$ are related to the projector operator $R_kR_k^{\dag}$, as expected in the context of least-squares problem \cref{eq:Andersonbetas}:
\begin{proposition}\label{pros-AAm-residual-A-random}
The residuals $r_{k+1}$ generated by AA($m$) iteration \cref{eq:AA-iteration} applied to linear iteration \cref{eq:linear-fixed-point} with $k \ge m$ satisfy the following:

If $R_k^T R_k$ is invertible, then
\begin{equation*}
  r_{k+1} =M\left(I-R_k(R_k^T R_k)^{-1}R_k^T\right)r_k, \, k\geq m.
\end{equation*}
More generally,
\begin{equation}\label{eq:prop-rk-based-on-projection}
  r_{k+1} =M(I-R_kR_k^{\dag})r_k, \, k\geq m.
\end{equation}
Furthermore,
\begin{equation}\label{eq:prop-orthognal-Rk}
  R_k^T(I-R_kR_k^{\dag})r_k =0, \, k\geq m.
\end{equation}

\end{proposition}
\begin{proof}
From \cref{eq:AAm-beta-form-pseudo} and \cref{eq:m+1-term-rk-linear-random-guess} we have
 \begin{align*}
 r_{k+1}  & = M r_k + \sum_{i=1}^{m}\beta_i^{(k)} M (r_k-r_{k-i}),\\
    &= Mr_k + MR_k \boldsymbol{\beta}^{(k)},  \\
   & =  M(I-R_kR_k^{\dag})r_k.
 \end{align*}
Expression \cref{eq:prop-orthognal-Rk} is then obtained from \cref{eq:prop-rk-based-on-projection} using \cref{AAm-rk-M-orthogonal}.
\end{proof}

\section{Specific results for AA(1)}\label{sec:AA(1)-Krylov}
%
\btxt{In this section, we apply to AA(1) the recurrence relations and orthogonality properties that were derived in \Cref{sec:AAm-krylov} for AA($m$) viewed as a Krylov method, for the linear case with iteration function given by (\ref{eq:linear-fixed-point}).
We first derive explicit nonlinear recurrence relations for the AA(1) residuals and residual update polynomial that no longer depend on the $\beta_k$, and a bound on the acceleration coefficients $\beta_k$. We then derive new convergence bounds on the AA(1) residuals. We finally prove results on the invariance of the asymptotic convergence factor under scaling of the initial condition, and on finite convergence for eigenvector initial conditions. The relevance of these new theoretical properties for understanding AA(1) convergence is briefly discussed here and will be illustrated in numerical tests in \Cref{sec:Numerical-result}.}

\subsection{Nonlinear recurrence relations for AA(1) and bound on $\beta_k$}
\btxt{The result of \cref{pro:polynomial-form-AAm-AAj} shows that AA($m$) is a Krylov method in the linear case, where the coefficients of the residual update polynomials depend on the acceleration coefficients $\beta_i^{(k)}$, which in turn depend on the residuals. To derive further properties for AA(1), it will be useful to formulate recurrence formulas that are nonlinear in the residuals, with the $\beta_k$ eliminated.
The following result establishes an explicit nonlinear three-term recurrence for the AA(1) residuals that does not include $\beta_k$.}
\begin{proposition}\label{AA1-rk-simple-form-randomx0x1}
The residuals $r_{k+1}$ generated by AA($m$) iteration \cref{eq:AA-iteration} with $m=1$ applied to linear iteration \cref{eq:linear-fixed-point} satisfy
\begin{itemize}
\item if $r_k\neq r_{k-1}$,
\begin{equation*}
  r_{k+1} =\frac{1}{(r_k-r_{k-1})^T(r_k-r_{k-1})} M\left(-r_k r_{k-1}^T + (r_k r_{k-1}^T)^T\right)(r_k-r_{k-1}), \, k\geq 1.
\end{equation*}

\item if $r_k=r_{k-1}$, $r_{k+1}=Mr_k$.
\end{itemize}
\end{proposition}
\begin{proof}
We first consider the case that $r_k\neq r_{k-1}$, so
\begin{equation*}
  \beta_k = \frac{-r_k^T(r_k- r_{k-1})}{(r_k-r_{k-1})^T(r_k-r_{k-1})}.
\end{equation*}
From  \cref{eq:m+1-term-rk-linear-random-guess} with $m=1$, we have
\begin{align*}
  r_{k+1} &=(1+\beta_k) Mr_k -\beta_k Mr_{k-1}, \\
  &= \frac{-r_{k-1}^T(r_k- r_{k-1})}{(r_k-r_{k-1})^T(r_k-r_{k-1})} Mr_k+\frac{r_k^T(r_k- r_{k-1})}{(r_k-r_{k-1})^T(r_k-r_{k-1})} Mr_{k-1},\\
  &=\frac{1}{(r_k-r_{k-1})^T(r_k-r_{k-1})}M(-r_k r_{k-1}^T + r_{k-1}r_k^T)(r_k-r_{k-1}),\\
  &=\frac{1}{(r_k-r_{k-1})^T(r_k-r_{k-1})}M\big(-r_k r_{k-1}^T + (r_{k}r_{k-1}^T)^T\big)(r_k-r_{k-1}).
\end{align*}
If $r_{k}=r_{k-1}$, then \rtxt{we choose} $\beta_k=0$. So from \cref{eq:m+1-term-rk-linear-random-guess} we have $r_{k+1}=Mr_k$ \rtxt{(which, in fact, holds for any choice $\beta_k$ as solution of the rank-deficient least squares problem (\ref{eq:Andersonbetas})).}
\end{proof}

\rtxt{It is interesting to note that the condition $r_k=r_{k-1}\neq0$ in \cref{AA1-rk-simple-form-randomx0x1} may actually occur in AA(1) iteration sequences that eventually converge to $r=0$. The following is a simple example to illustrate this point.
%
%
%
\begin{example}\label{prob:r-stall}
Consider solving $Ax=b$ with fixed-point iteration $x_{k+1}=q(x_k)=Mx_k+b$, where $M=I-A$ and the fixed-point iteration is accelerated by AA($m$) iteration (\ref{eq:AA-iteration}) with $m=1$.
For simplicity, we consider an example where $b=0$, and we choose
\begin{equation}\label{eq:q-linear-2x2-stallAM}
  A=\left[ \begin{array}{cc}
  -1/2 & 0\\
  0  &  1/2
  \end{array}
  \right] \qquad \textrm{and} \qquad 
   M=I-A=\left[ \begin{array}{cc}
   3/2 & 0\\
  0  &  1/2
  \end{array}
  \right].    
\end{equation}
We choose initial condition
\begin{equation}\label{eq:q-linear-2x2-stallAM2}
  x_0=\left[ \begin{array}{c}
  -2\\
  2
  \end{array}
  \right] \qquad \textrm{so} \qquad 
   r_0=x_0-q(x_0)=(I-M)x_0= 
   \left[ \begin{array}{c}
   1\\
   1
  \end{array}
  \right].    
\end{equation}
Following \cref{pro:AA1-rks} and \cref{eq:AA-1-step-beta} we get 
\begin{equation}\label{eq:q-linear-2x2-stallAM3}
   r_1=Mr_0= 
   \left[ \begin{array}{c}
   3/2\\
   1/2
  \end{array}
  \right]\qquad \textrm{and} \qquad 
  \beta_1=-1.    
\end{equation}
Since $\beta_1=-1$, we see from \cref{pro:AA1-rks} that
\begin{equation}\label{eq:q-linear-2x2-stallAM4}
   r_2=Mr_0=r_1. 
\end{equation}
So we have $x_1=x_2$ and $r_1=r_2$. 
This renders least-squares problem (\ref{eq:Andersonbetas}) rank-deficient,
but for any choice of $\beta_2$, including the minimum-norm solution $\beta_2=0$,
we obtain from \cref{pro:AA1-rk+1=Mg_k} that
\begin{equation}\label{eq:q-linear-2x2-stallAM5}
   r_3=Mr_2,
\end{equation}
and further numerical calculations show that AA(1) converges to $x=0$ in subsequent iterations.

This simple example shows that the temporary stalling behavior $r_2=r_1$ may occur in the AA(1)
iteration sequence, and that convergence may ensue in subsequent iterations. It is clear from
\cref{pro:AA1-rks} that a nonzero $r_2$ equals $r_1$ if and only if $\beta_1=-1$ (since we
have assumed that $A=I-M$ is nonsingular). From expression \cref{eq:AA-1-step-beta}
we can see that $\beta_1=-1$ if and only if a nonzero initial residual satisfies $r_0^T A r_0=0$. 
The initial residual, $r_0$, in our example was chosen to satisfy this condition for our matrix $A$.
This condition also implies that, if $A$ is such that $r_0^T A r_0>0$ for any nonzero $r_0$ (e.g., if $A$ is symmetric positive definite),
then there are no AA(1) iteration sequences where a nonzero $r_2$ equals $r_1$.
We do not know if $r_k=r_{k-1}$ can occur for $k>2$.
\end{example}
}

\btxt{Next, we will use \cref{AA1-rk-simple-form-randomx0x1} to obtain an explicit nonlinear recurrence relation for the residual update polynomials in \cref{pro:polynomial-form-AAm-AAj} for the case that $m=1$. It is useful for further reference to first specialize two results from \Cref{sec:AAm-krylov} to AA(1), namely \cref{pro:polynomial-form-AAm-AAj} and \cref{pro:AAm-polynomial-power-initial-AAj}:}
\begin{proposition}\label{pro:polynomial-form2}
AA($m$) iteration \cref{eq:AA-iteration} with $m=1$ applied to linear iteration \cref{eq:linear-fixed-point} is a Krylov method, with the residuals given by
\begin{equation}\label{eq:rk-polynomial-form}
    r_{k+1} = p_{k+1}(M)\,r_0,
\end{equation}
where the residual polynomials satisfy the recurrence relation
\begin{equation}\label{eq:AA1-three-term-M}
  p_{k+1}(\lambda)=(1+\beta_k)\lambda p_k(\lambda) - \beta_k \lambda p_{k-1}(\lambda), \quad k\geq 1,
\end{equation}
and $p_0(\lambda) =1, p_1(\lambda) =\lambda$. Moreover, $p_k(\lambda)$ is a polynomial with degree at most $k$ and $p_k(1)=1$ and $p_k(0)= 0$ for $k\geq 1$.
\end{proposition}

\begin{proposition}\label{pro:polynomial-form3}
The residuals of AA($m$) iteration \cref{eq:AA-iteration} with $m=1$ applied to linear iteration \cref{eq:linear-fixed-point} satisfy
\begin{equation*}
 r_{2s+1} =M^{s+1}\widehat{p}_{s}(M)\,r_0, \quad r_{2s+2}=M^{s+1} \widehat{p}_{s+1}(M)\,r_0,
\end{equation*}
where $\widehat{p}_{s+i-1}(\lambda), i=1,2$ is a polynomial with degree at most $s+i-1$.
\end{proposition}
\cref{pro:polynomial-form3} shows that the AA(1) residual polynomial gains a power of $M$ every two iterations.

\btxt{The following property now provides an explicit nonlinear three-term recurrence relation for the AA(1) residual update polynomials, where the polynomials are expressed not in terms of $\beta_k$ but directly using multivariate matrix polynomials $L_k$ in $M$, $M^T$ and $R_0=r_0r_0^T$.}
\begin{proposition} \label{prop:rk-Lk}
Let $r_k$ be the residual of AA($m$) iteration \cref{eq:AA-iteration} with $m=1$ applied to linear iteration \cref{eq:linear-fixed-point}. Assume that $r_k\neq r_{k-1}$ for all $k$. Then
 \begin{equation}\label{eq:rk-Lk-form}
    r_{k} =L_k r_0, \quad  k=0, 1, \ldots
  \end{equation}
with $L_0=I, L_1= M$ and
\begin{equation}\label{eq:Lk-form}
  L_{k+1} = \frac{1}{\|(L_k-L_{k-1})r_0\|^2} M\left(-L_kR_0 L_{k-1}^T+ (L_kR_0 L_{k-1}^T)^T \right) (L_k-L_{k-1}), \quad k\geq 1,
\end{equation}
where $R_0= r_0r_0^T$.
\end{proposition}
\begin{proof}
Assume that when $k\leq n$,  \cref{eq:rk-Lk-form} is true with $L_k$ defined by \cref{eq:Lk-form}. We show that \cref{eq:rk-Lk-form} remains valid for  $k=n+1$ and $L_{k+1}$ given by \cref{eq:Lk-form}. Using the equation for $L_{k+1}$ in \cref{eq:Lk-form}, we have
\begin{align*}
  L_{n+1}r_0 &= \frac{1}{\|(L_n-L_{n-1})r_0\|^2} M\left(-L_nR_0 L_{n-1}^T+ (L_nR_0 L_{n-1}^T)^T \right) (L_n-L_{n-1})r_0, \\
    &=   \frac{1}{\|r_n-r_{n-1}\|^2} M\left(-r_n r_{n-1}^T+ (r_n r_{n-1}^T)^T \right) (r_n-r_{n-1}),\\
    & =r_{n+1},
\end{align*}
by \cref{AA1-rk-simple-form-randomx0x1}. This completes the proof.
\end{proof}
We list the first few  $L_k$ in the following:
\begin{itemize}
\item $L_0 =I$.
\item $L_1= M$.
\item $L_2 = \displaystyle \frac{1}{\|(M-I)r_0\|^2}  M\left(-MR_0+ (MR_0)^T\right)(M-I)$.

\item Let $T_1=L_2 R_0L_1^T=\displaystyle \frac{1}{\|(M-I)r_0\|^2}  M\left(-MR_0+ (MR_0)^T\right)(M-I)R_0 M^T$ and
$T_2 =L_2-L_1=\displaystyle \frac{1}{\|(M-I)r_0\|^2}  M\left(-MR_0+ (MR_0)^T\right)(M-I)-M$.

Then,
\begin{align*}
L_3 & = \frac{1}{\|T_2r_0\|^2}M(-T_1+T_1^T)T_2.
\end{align*}
\end{itemize}
\btxt{
From \eqref{eq:Lk-form}, it is interesting to note that $\textrm{rank}(L_{k+1}) \leq 2$. 
Specifically, ${\textrm{rank}(R_0) = 1}$ (recall $R_0= r_0 r_0^T$), and, so, by the rank product rule, $\textrm{rank}(L_kR_0 L_{k-1}^T) \leq 1$, and then by the subadditivity of the rank, $\textrm{rank}(-L_kR_0 L_{k-1}^T+ (L_kR_0 L_{k-1}^T)^T) \leq 2$. Thus, again by the rank product rule, it follows that $\textrm{rank}(L_{k+1}) \leq 2$.
In \Cref{sec:AA1-bounds} we derive a bound on $\Vert r_{k+1} \Vert$ by exploiting this rank-two nature of the residual update.
%

}

We finally derive a lower bound for the coefficients $\beta_k$ in AA(1) in the linear case.
\rtxt{We noted this simple but conspicuous property in our numerical results
and it appears that this property is not available in the literature, so
we state it here explicitly.}
%

\begin{proposition}\label{thm:lower-boound-betak}
Consider AA($m$) iteration (\ref{eq:general-guess-AAm}) with $m=1$ and general initial guess $\{x_0,x_1\}$, applied to linear fixed-point function $q(x)=Mx+b$ with $\|M\|\leq 1$.
Let $x_0$ and $x_1$ be initial guesses with $||r_1||\leq ||r_0||$.
Then the AA(1) coefficients $\beta_k$  satisfy
\begin{equation*}
\beta_k>-1, \quad  k=1,2,\ldots\,\, .
\end{equation*}
\end{proposition}
\begin{proof}
\rtxt{First, with $||r_1||\leq ||r_0||$ and $\|M\|\leq 1$, we have from Theorem 2.1 in \cite{toth2015} 
that
$$
\|r_k\|\leq \|r_{k-1}\|, \qquad k=1,2,3,\ldots
$$
Next, recall that 
\begin{equation*}
  \beta_k = \frac{-r_k^T(r_k-r_{k-1})}{(r_k-r_{k-1})^T(r_k-r_{k-1})},
\end{equation*}
if  $r_k\neq r_{k-1}$, and $\beta_k=0$ otherwise. 

When $r_k= r_{k-1}$ we have $\beta_k=0 > -1$ and the results holds.

We next consider the case that $r_k\neq r_{k-1}$.
Let $w_k= r_k-r_{k-1}\neq 0$. Then, 
$$
\beta_k ||w_k||^2 =-r_k^Tw_k=(-w_k-r_{k-1})^Tw_k=-\|w_k\|^2-r_{k-1}^Tw_k,
$$
%
%
%
%
which means that
\begin{equation}\label{eq:pf-inner-product}
r_{k-1}^T w_k =(-1-\beta_k)||w_k||^2.
\end{equation}
We also have that
\begin{equation}\label{eq:next}
r_{k-1}^T w_k = r_{k-1}^T (r_k-r_{k-1}) = \|r_{k-1}\| (\|r_{k}\| \cos(\phi_k)-\|r_{k-1}\|) < 0,
\end{equation}
where the inequality follows from $\|r_k\|\leq \|r_{k-1}\|$ and $\cos(\phi_k) \in [-1,1)$ since $r_k\neq r_{k-1}$.
Combining \cref{eq:pf-inner-product} and \cref{eq:next} we conclude
$$
-1-\beta_k < 0,
$$
which proves the result, $\beta_k>-1$.}
\end{proof}
\begin{remark}
It is easy to see that \cref{thm:lower-boound-betak} applies to the traditional  AA(1) iteration \cref{eq:AA-iteration} with single initial guess $x_0$, since then $x_1=Mx_0+b$ and $r_1=Mr_0$, so $\|r_1\|\leq \|r_0\|$ if $\|M\|\leq 1$.
\rtxt{\cref{thm:lower-boound-betak} also shows that when $\|M\|\leq 1$ the temporary stalling behavior from \cref{prob:r-stall} with nonzero $r_2=r_1$ cannot happen for any initial guess $r_0$, since it requires $\beta_1=-1$.
}
\end{remark}

\rtxt{ 
\subsection{Convergence bounds for AA(1) residuals} \label{sec:AA1-bounds}
The current theoretical understanding on quantifying by how much AA($m$) can improve the convergence speed of the underlying FP iteration
is very limited, including for the linear case and $m=1$.
We now present new convergence bounds on the norm of the residuals generated by AA(1) applied to the linear iteration \cref{eq:linear-fixed-point}.
%

\begin{theorem} \label{thm:AA1-bounds}
Let $r_k$ be the residual of AA($m$) iteration \cref{eq:AA-iteration} with $m=1$ applied to linear iteration \cref{eq:linear-fixed-point}. 
Define $y_k > 0$ as the ratio of the norms of the $k$th and $(k-1)$st residual vectors, 
$
y_k := \Vert r_{k} \Vert / \Vert r_{k-1} \Vert,
$
and define $\phi_k \in [0, \pi]$ as the angle between these two vectors,
$
r_{k}^T r_{k-1} = \Vert r_{k-1} \Vert \Vert r_{k} \Vert \cos \phi_k.
$
Let $\sigma_{\min}(M)$ and $\sigma_{\max}(M)$ denote the minimum and maximum singular values of $M$, respectively. 
Then, the norm of the $(k+1)$st residual vector for $k > 1$ may be bounded as
\begin{align} \label{eq:rk+1-lower-upper-bound}
\sqrt{{\cal B}(\phi_k, y_k)} 
\sigma_{\min}(M)
&\leq
\frac{\Vert r_{k+1} \Vert}
{\Vert r_{k} \Vert}
\leq 
\sqrt{{\cal B}(\phi_k, y_k)} 
\sigma_{\max}(M), 
\quad &&r_k \neq r_{k-1},
\\
\label{eq:rk+1-lower-upper-bound-special}
\sigma_{\min}(M)
&\leq
\frac{\Vert r_{k+1} \Vert}
{\Vert r_{k} \Vert}
\leq 
\sigma_{\max}(M), 
\quad &&r_k = r_{k-1},
\end{align}
where the function ${\cal B} \colon \big( [0, \pi] \times (0, \infty) \big) \setminus (0, 1) \to [0, 1]$ is given by\footnote{\rtxt{Note that ${\cal B}(\phi_k, y_k)$ is indeterminate at the point $(\phi_k, y_k) = (0, 1)$ corresponding to the special case when $r_k = r_{k-1}$.} 
%
}
\begin{align} \label{eq:calB_def}
{\cal B}(\phi_k, y_k) = \frac{\sin^2 \phi_k}{y_k^2 - 2 y_k \cos \phi_k + 1}.
\end{align}
\end{theorem}

\begin{proof}

First consider the case of \eqref{eq:rk+1-lower-upper-bound} in which $r_k \neq r_{k-1}$. 
To aid in the readability of this portion of the proof we use the shorthands $a := r_k$ and $b := r_{k-1}$.
Recall from \cref{AA1-rk-simple-form-randomx0x1} that in this case $r_{k+1}$ can be expressed as
\begin{align} \label{eq:S_alpha_form}
r_{k+1} = M \frac{S (a - b)}{\Vert a - b \Vert^2} = M \alpha,
\end{align}
in which $S$ is the rank-2, skew-symmetric matrix given by $S = b a^T - a b^T$, and we have defined the shorthand vector $\alpha := \frac{S (a - b)}{\Vert a - b \Vert^2}$.
Now, recall the following inequalities that hold for any vector $w$: 
\begin{align} \label{eq:sing_vals_def}
\sigma_{\min}^2(M) := \inf_{z \neq 0} \frac{\Vert M z \Vert^2}{\Vert z \Vert^2}
\leq 
\frac{\Vert M w \Vert^2}{\Vert w \Vert^2} 
\leq 
\sup_{z \neq 0} \frac{\Vert M z \Vert^2}{\Vert z \Vert^2} =: \sigma_{\max}^2(M).
\end{align}
Taking the squared norm of both sides of \eqref{eq:S_alpha_form} and then applying the inequalities in \eqref{eq:sing_vals_def} to the vector $w = \alpha$ gives
\begin{align} \label{eq:rk+1_alpha_bounds}
\sigma_{\min}^2(M) \Vert \alpha \Vert^2
\leq 
\Vert r_{k+1} \Vert^2
\leq 
\sigma_{\max}^2(M) \Vert \alpha \Vert^2.
\end{align}

We now proceed by showing that $\Vert \alpha \Vert^2 = \Vert a \Vert^2 {\cal B}(\phi_k, y_k) $, and, thus, that \eqref{eq:rk+1_alpha_bounds} is equivalent to the claim \eqref{eq:rk+1-lower-upper-bound}.
By the skew-symmetry of $S$, the squared norm of $\alpha$ can be written as
\begin{align} \label{eq:alpha-norm^2}
\Vert \alpha \Vert^2 
= 
\frac{\Vert S (a-b) \Vert^2}{\Vert a - b \Vert^4} 
= 
\frac{-(a - b)^T S^2 (a-b)}{\Vert a - b \Vert^4}.
\end{align}
To evaluate $\Vert \alpha \Vert^2$, it is instructive to first consider the vector $S^2 (a - b)$:
\begin{align}
S^2 (a - b)
&=
\big[
\big( 
b a^T - a b^T \big) \big( b a^T - a b^T 
\big)
\big] 
(a - b), 
\\
&=
\big[
- \Vert b \Vert^2 a a^T - \Vert a \Vert^2 b b^T 
+ \big(a^T b \big) \big( a b^T + b a^T \big)
\big]
(a - b), 
\\
\begin{split}
&=
-\Vert b \Vert^2 \big( \Vert a \Vert^2 - a^T b \big) a
- \Vert a \Vert^2 \big( a^T b - \Vert b \Vert^2 \big) b
\\ 
& \hspace{4ex} 
+ \big( a^T b \big)
\big[ 
\big( a^T b - \Vert b \Vert^2 \big) a + \big( \Vert a \Vert^2 - a^T b \big) b
\big],
\end{split}
\\
\label{eq:S^2(b-a)}
&=
\Big[ \big( a^T b \big)^2 - \big(\Vert a \Vert \Vert b \Vert \big)^2 \Big] (a - b).
\end{align}
That is, $a - b$ is an eigenvector of $S^2$ with eigenvalue $\big( a^T b \big)^2 - \big(\Vert a \Vert \Vert b \Vert \big)^2 \leq 0$. Plugging \eqref{eq:S^2(b-a)} into \eqref{eq:alpha-norm^2} and simplifying, we find 
\begin{align}
\Vert \alpha \Vert^2 
= 
\frac{ \big(\Vert a \Vert \Vert b \Vert \big)^2 - \big( a^T b \big)^2 }{\Vert a - b \Vert^2} 
&=
\frac{
\big(\Vert a \Vert \Vert b \Vert \big)^2 (1 - \cos^2 \phi_k)
}{\Vert a \Vert^2 - 2 \Vert a \Vert \Vert b \Vert \cos \phi_k + \Vert b \Vert^2},
\\
&=
\Vert a \Vert^2
\frac{\sin^2 \phi_k}
{
\frac{\Vert a \Vert^2}{\Vert b \Vert^2}  
- 2 \frac{\Vert a \Vert}{\Vert b \Vert} \cos \phi_k
+
1 
}.
\end{align}
Recalling that $a = r_k$, and $y_k = \Vert a \Vert / \Vert b \Vert$, we indeed see that $\Vert \alpha \Vert^2 = \Vert r_k \Vert^2 {\cal B}(\phi_k, y_k) $ with ${\cal B}$ defined in \eqref{eq:calB_def}. 
Plugging this into \eqref{eq:rk+1_alpha_bounds}, dividing through by $\Vert r_k \Vert^2$, and then taking the square root yields the claim \eqref{eq:rk+1-lower-upper-bound}.

We conclude the proof by considering the special case in which $r_k = r_{k-1}$.
Recall from \cref{AA1-rk-simple-form-randomx0x1} that when $r_k = r_{k-1}$ the next AA(1) residual is given by $r_{k+1} = M r_k$.
Invoking the inequalities \eqref{eq:sing_vals_def} for the vector $w = r_k$ leads immediately to the claim \eqref{eq:rk+1-lower-upper-bound-special}. 
\end{proof}

\begin{corollary} \label{cor:AA1-equality}
Suppose that the two-norm condition number of $M$ is one, such that ${\sigma_{\min}(M) = \sigma_{\max}(M) = \Vert M \Vert}$. 
Then,
\begin{align}
\frac{\Vert r_{k+1} \Vert}
{\Vert r_{k} \Vert}
&=
\Vert M \Vert
\sqrt{{\cal B}(\phi_k, y_k)}, 
&&r_k \neq r_{k-1},
\\
\frac{\Vert r_{k+1} \Vert}
{\Vert r_{k} \Vert}
&=
\Vert M \Vert, 
&&r_k = r_{k-1}.
\end{align}
\end{corollary}

\begin{proof}
These two equalities follow from \eqref{eq:rk+1-lower-upper-bound} and \eqref{eq:rk+1-lower-upper-bound-special}, respectively, because the left- and right-hand sides are equal.
\end{proof}

\Cref{thm:AA1-bounds} and \cref{cor:AA1-equality} describe local convergence behaviour for AA(1) in the sense that they hold for any iteration $k$, but they do not say anything about the effective or average convergence behaviour in the asymptotic regime as ${k \to \infty}$.
Presently, it is not clear whether these results can in fact be used to derive asymptotic convergence results (see the numerical results in \Cref{subsec:linear-bounds} for further discussion on this).

The residual vectors from the FP iteration \eqref{eq:linear-fixed-point}, i.e., the underlying iteration that AA(1) is accelerating, satisfy $r_{k+1} = M r_k$. Clearly, the FP residuals satisfy the local bounds $\sigma_{\min}(M) \leq \Vert r_{k+1} \Vert / \Vert r_k \Vert \leq \sigma_{\max}(M)$. It is these local bounds, rather than an asymptotic convergence result, that serve as the relevant comparison to the local AA(1) bounds from \Cref{thm:AA1-bounds} and \cref{cor:AA1-equality}.\footnote{\rtxt{Note that these local FP bounds are the same as the AA(1) bounds \eqref{eq:rk+1-lower-upper-bound-special} when $r_{k} = r_{k-1}$, because, in this case, AA(1) just applies the basic FP iteration, $r_{k+1} = M r_k$. 
For this reason, in the following comparison between AA(1) and the FP iteration, we suppose that AA(1) residuals satisfy $r_k \neq r_{k-1}$ so that AA(1) is in fact distinct from the FP iteration.}}
%

\begin{figure}[b!]
\centering
\includegraphics[width=.485\textwidth]{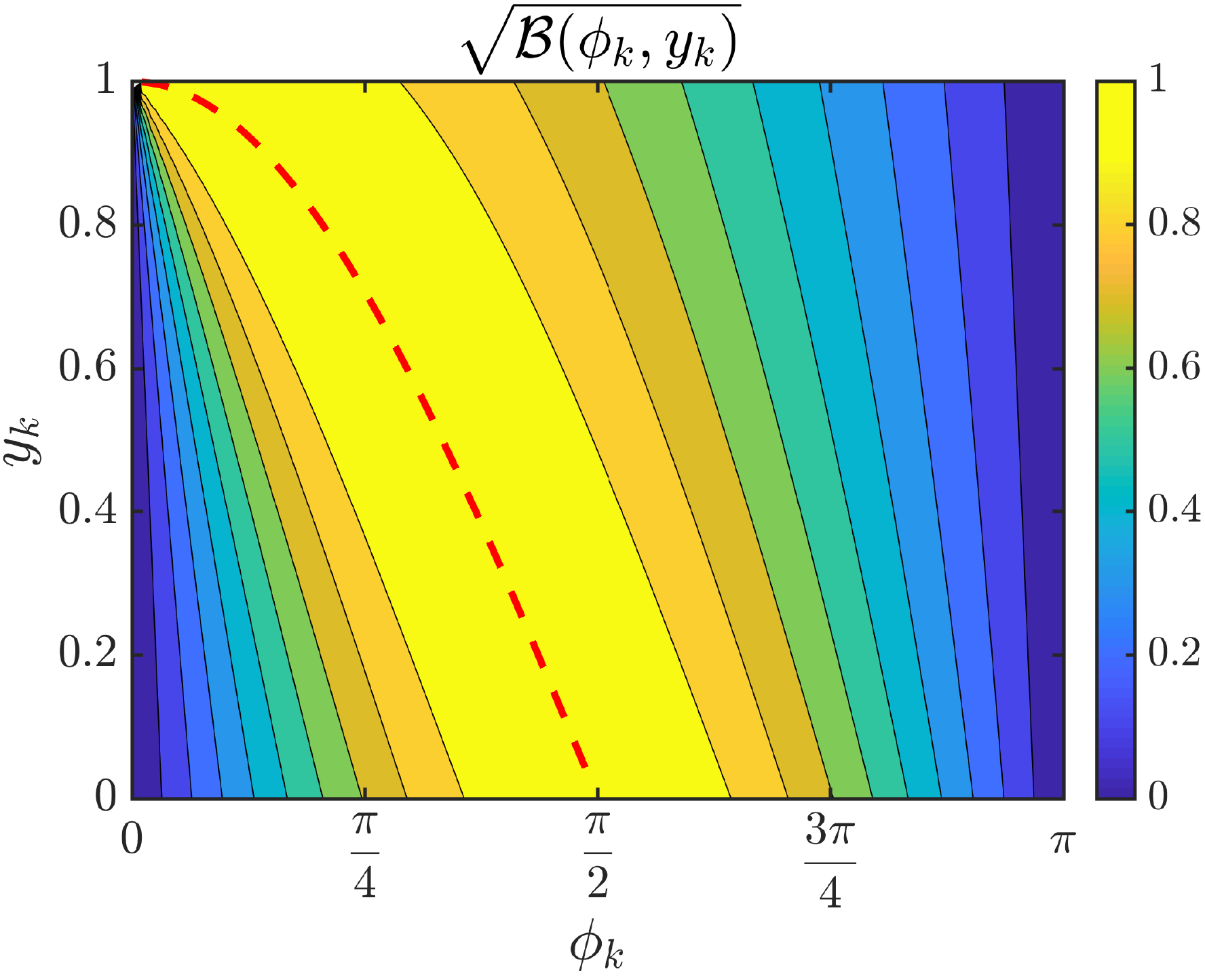}
\quad
\includegraphics[width=.46\textwidth]{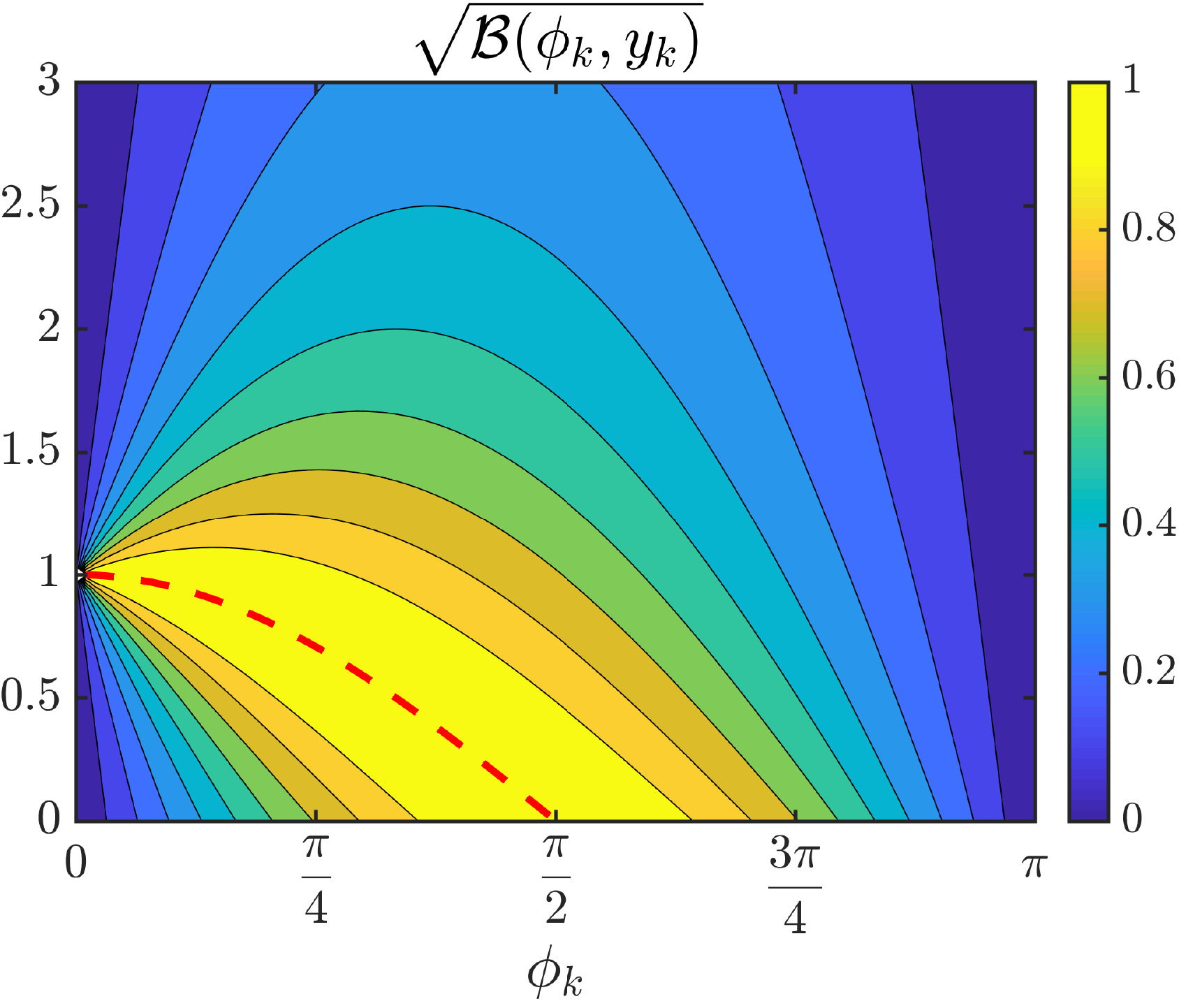}
\caption{Contours of the function $\sqrt{{\cal B}(\phi_k, y_k)}$ that appears in the AA(1) residual bounds and expressions given in \cref{thm:AA1-bounds} and \cref{cor:AA1-equality}.
At left, the function is shown on the interval $y_k \in (0, 1]$, while at right it is shown on the interval $y_k \in (0, 3]$.
The dashed red line is the critical line $y_k = \cos \phi_k \in  (0, 1)$ along which the function ${\cal B}(\phi_k, y_k)$ reaches its maximum value of unity, and also corresponds to the acceleration coefficient \eqref{eq:AA-1-step-beta} being $\beta_k = 0$.
\label{fig:calB_plot}
}
\end{figure}

To compare the local FP iteration bounds with our new AA(1) bounds it is instructive to consider in more detail the function ${\cal B}$ given by \eqref{eq:calB_def}.
See \cref{fig:calB_plot} for contour plots of $\sqrt{{\cal B}(\phi_k, y_k)}$ in $\phi_k$--$y_k$ space. 
The function ${\cal B}$ has several interesting properties, including that ${\cal B}(\phi_k, y_k) \in [0, 1]$, with its extremum of unity being reached only along the critical line $y_k = \cos \phi_k \in (0, 1)$.
Geometrically, $y_k = \cos \phi_k$ corresponds to $r_k$ being orthogonal to $r_k - r_{k-1}$, which, as may be seen from \eqref{eq:AA-1-step-beta}, results in the acceleration coefficient $\beta_k = 0$, and, thus, the new AA(1) residual being nothing but the FP residual, $r_{k+1} = M r_k$.
From the upper bound in \eqref{eq:rk+1-lower-upper-bound} this means that AA(1) residuals satisfy $\Vert r_{k+1} \Vert / \Vert r_k \Vert < \sigma_{\max}(M)$ whenever $y_k \neq \cos \phi_k$ and $\Vert r_{k+1} \Vert / \Vert r_k \Vert \leq \sigma_{\max}(M)$ whenever $y_k = \cos \phi_k$.

From the plots shown in \cref{fig:calB_plot}, it is also clear that, depending on the values of $\phi_k$ and $y_k$, AA(1) may produce residuals $r_{k+1}$ satisfying $\Vert r_{k+1} \Vert / \Vert r_k \Vert \ll \sigma_{\max}(M)$.
For example, $\Vert r_{k+1} \Vert / \Vert r_k \Vert$ becomes arbitrarily small as $\phi_k \to 0$ (provided $y_k \not \to 1$ also) or as $\phi_k \to \pi$. 
In fact, ${\cal B}(0, y) = {\cal B}(\pi, y) = 0$, which means that if $r_k$ and $r_{k-1}$ point in the same or opposite directions to one another, and $r_k \neq r_{k-1}$, then AA(1) converges exactly on the next iteration, $r_{k+1} = 0$.

It is already known in the literature that AA(1) residuals satisfy the upper bound $\Vert r_{k+1} \Vert / \Vert r_k \Vert \leq \sigma_{\max}(M)$. 
For example, this follows as the $m=1$ special case of the AA($m$) result given by \cite[Thm. 2.1]{toth2015}, and it may also be derived as the special linear case of the AA(1) result given as \cite[Thm. 4.1]{evans2020proof}.
A key distinction between the results of \cref{thm:AA1-bounds} and \cref{cor:AA1-equality} and those from \cite{toth2015,evans2020proof} is that our new bounds are parametrized in terms of $y_k$ and $\phi_k$, revealing that the AA(1) residual norm $\Vert r_{k+1} \Vert$ is very strongly influenced by the relationship between $r_k$ and $r_{k-1}$---specifically, the angle between them and their relative magnitudes.
Furthermore, having both lower and upper bounds allows for a more complete understanding of the residual convergence of AA(1)---the upper bound in \eqref{eq:rk+1-lower-upper-bound} is larger than the lower bound by a factor of the condition number of $M$, so we know $\Vert r_{k+1} \Vert / \Vert r_k \Vert$ to within this factor.
Note that from \cite[Thm. 4.1]{evans2020proof} it may also be possible to develop insight into how the relationship between $r_k$ and $r_{k-1}$ influences $\Vert r_{k+1} \Vert$ by further examining the so-called gain $\theta_k$ (see \cite[Sec. 3.3]{evans2020proof}) that appears in that result; however, no lower bound is provided in that work, so it is not clear how tight the resulting upper bound may be.

Interestingly, similar to \cref{thm:AA1-bounds} and \cref{cor:AA1-equality}, some known convergence results for restarted GMRES also involve angles between residual vectors.
For example, for certain matrices, the norm of the GMRES($m$) residual at the end of the $(k+1)$st restart cycle equals the product of the norm of the residual at the end of the $k$th restart cycle and the cosine of the angle between the two residuals, as shown in \cite[Thm. 4]{Baker_etal_2005}.
Unfortunately, it is not immediately clear that existing theoretical insights on the convergence of restarted GMRES($m$), such as those described in \cite{Baker_etal_2005}, can be extended to describe the convergence of windowed AA($m$).

\Cref{subsec:linear-bounds} will illustrate numerically how the residual bounds derived in this section are relevant for understanding AA(1) convergence patterns.

} 

\subsection{AA(1) asymptotic convergence factor: scaling invariance of initial guess and finite convergence property}
\label{subsec:scaling_theory}
\btxt{As discussed in the next section, numerical results have shown that the asymptotic root-linear convergence factor of AA($m$) iteration sequences strongly depends on the initial guess, see \cite{LinearacAA}. The theoretical understanding of this is still very limited. We now apply some of the AA(1) results of this section to prove results on the influence of the AA(1) initial guess on the asymptotic convergence factor.

We first discuss a scaling invariance property of the initial guess $x_0$ for the AA(1) method in the linear case. Since solving $Ax=b$ is equivalent to solving $Ay=0$ with $y=x-A^{-1}b$, we formulate these scaling properties for the case of a homogenous system.}
\begin{proposition}\label{scaling-initial-guess-polynomial-invariant}
Consider solving $Ax=0$ using AA($m$) iteration \cref{eq:AA-iteration} with $m=1$ applied to linear iteration \cref{eq:linear-fixed-point} with $M=I-A$ and nonzero initial guess $x_0$. Consider the AA(1) polynomials in \cref{eq:AA1-three-term-M}, which depend on $x_0$ through the $\beta_k$. We have the following properties: for any nonzero scalar $\alpha$,

\begin{equation}\label{beta-invariant-scaling-guess}
\beta_k(x_0) =\beta_k(\alpha x_0),
\end{equation}
and, therefore,
\begin{equation}\label{pk-invariant-scaling-guess}
  p_{k}(\lambda,x_0) =p_{k}(\lambda,\alpha x_0),
\end{equation}
where we explicitly indicate the dependence  of $\beta_k$ and $p_k(\lambda)$ on the initial condition.
\end{proposition}
\begin{proof}
We use induction. Note that
\begin{align*}
  r_0 &= Ax_0, \\
  r_1&= Mr_0=MAx_0, \\
  \beta_1(x_0)&= \frac{-r_1^T(r_1-r_0)}{(r_1-r_0)^T(r_1-r_0)}=\frac{-(MAx_0)^T(MA-A)x_0}{\left((MA-A)x_0\right)^T (MA-A)x_0}.
\end{align*}
Thus, when $k=1$, we have $\beta_1(x_0) =\beta_1(\alpha x_0)$. Since $p_1(\lambda)=\lambda$, it is obvious that $p_{1}(\lambda,x_0) =p_{1}(\lambda,\alpha x_0)$.

Assume that for $k\leq n$,  $p_{k}(\lambda,x_0) =p_{k}(\lambda,\alpha x_0)$. Note that for $k\leq n$,
\begin{align*}
  \beta_k(\alpha x_0) &=\frac{-r_k^T(r_k-r_{k-1})}{\|r_k-r_{k-1}\|^2},\\
   &=\frac{- \left(p_k(M,\alpha x_0) A\alpha x_0\right)^T \left(p_k(M,\alpha x_0)-p_{k-1}(M,\alpha x_0)\right)A\alpha x_0 }{\|\big(p_k(M,\alpha x_0)-p_{k-1}(M,\alpha x_0)\big)A\alpha x_0\|^2}, \\
   &=\beta_k(x_0).
\end{align*}
For $k=n+1$, we then have
\begin{align*}
  p_{n+1}(\lambda,\alpha x_0)&=(1+\beta_n(\alpha x_0))\lambda p_n(\lambda,\alpha x_0) - \beta_n(\alpha x_0) \lambda p_{n-1}(\lambda,\alpha x_0),\\
  & = (1+\beta_n(x_0))\lambda p_n(\lambda, x_0) - \beta_n( x_0) \lambda p_{n-1}(\lambda, x_0),\\
  &= p_{n+1}(\lambda, x_0),
\end{align*}
which completes the proof.
\end{proof}
The following result shows the invariance of the root-linear asymptotic convergence factor under scaling of the initial condition.
\begin{proposition}\label{pro:scaling-invariant-limit}
Consider solving $Ax=0$ using AA($m$) iteration \cref{eq:AA-iteration} with $m=1$ applied to linear iteration \cref{eq:linear-fixed-point} with $M=I-A$ and nonzero initial guess $x_0$. Then the AA(1) residuals defined in \cref{eq:rk-polynomial-form}  have the following property:
\begin{equation}
  \lim_{k\rightarrow \infty }\|r_k(x_0)\|^{\frac{1}{k}} = \lim_{k\rightarrow \infty } \|r_k( \alpha x_0)\|^{\frac{1}{k}},
\end{equation}
where  we explicitly indicate the dependence of $r_k$ on the initial condition.
\end{proposition}
\begin{proof}
For initial guess $\alpha x_0$, we have $r_0(\alpha x_0)= A\alpha x_0=\alpha Ax_0=\alpha r_0(x_0)$. From \cref{scaling-initial-guess-polynomial-invariant}, we know that $p_k(\lambda,x_0)=p_k(\lambda,\alpha x_0)$. Furthermore,
\begin{align*}
 \lim_{k\rightarrow \infty } \|r_k( \alpha x_0)\|^{\frac{1}{k}} &=\lim_{k\rightarrow \infty } \|p_k(M,\alpha x_0)r_0(\alpha x_0))\|^{\frac{1}{k}},\\
 &= \lim_{k\rightarrow \infty } (\alpha)^{\frac{1}{k}}\|p_k(M,x_0)r_0(x_0))\|^{\frac{1}{k}},\\
 &=\lim_{k\rightarrow \infty }\|p_k(M,x_0)r_0(x_0))\|^{\frac{1}{k}},\\
 &= \lim_{k\rightarrow \infty } \|r_k( x_0)\|^{\frac{1}{k}}.
\end{align*}
\end{proof}

Next, we give a result on the number of iterations  in which  AA(1) converges exactly for a special choice of initial guess.
\begin{proposition}\label{AA1-2-iterations-eigevector}
Consider solving $Ax=b$ using AA($m$) iteration \cref{eq:AA-iteration} with $m=1$ applied to linear iteration \cref{eq:linear-fixed-point} with $M=I-A$. For initial guess $x_0=x^* + v$, where $Ax^*=b$ and $v$ is any eigenvector of $A$,  AA(1) converges to the true solution in at most two iterations, i.e., $x_2=x^*$.
\end{proposition}
\begin{proof}
Assume that $A v =\mu v$ with $v\neq 0$. Note that $r_0 = Ax_0-b=A(v+x^*)-b=Av=\mu v$ and  $r_1=Mr_0= \mu Mv$.  Since
\begin{equation*}
r_1r_0^T = \mu M v (\mu v)^T = \mu^2(I-A)v v^T =\mu^2(1-\mu)v v^T,
\end{equation*}
we have
\begin{equation*}
-r_1r_0^T + (r_1r_0^T)^T =0.
\end{equation*}
From  \cref{AA1-rk-simple-form-randomx0x1} we then have $r_2=0$, which means that $x_2=x^*$.
\end{proof}

\rtxt{
Note that \cref{AA1-2-iterations-eigevector} can also be derived from \eqref{eq:rk+1-lower-upper-bound} in \cref{thm:AA1-bounds}.
That is, as described in the proof of \cref{AA1-2-iterations-eigevector}, the first two AA(1) residuals are $r_0 = \mu v$, and $r_1 = M r_0 = (1 - \mu) r_0$, and, are thus parallel but not equal (so long as $\mu \neq 1$). It follows immediately from \eqref{eq:rk+1-lower-upper-bound} that $r_2 = 0$, since ${\cal B}(0, y) = {\cal B}(\pi, y) = 0$. 
}

\btxt{The numerical results in the next section will illustrate these properties.}

\section{Numerical Results}\label{sec:Numerical-result}
\btxt{
In this section, we illustrate numerically how the theoretical results of this paper help in understanding AA(1) convergence. We first discuss how the theoretical results from \Cref{subsec:scaling_theory}, on scaling invariance of the initial guess for asymptotic convergence speed and on finite convergence properties, lead to improved understanding of some intricate AA(1) convergence patterns that are observed in numerical results. We also give numerical illustrations of the AA(1) residual polynomials that were presented in \Cref{sec:AAm-krylov}, compared to the fixed-point polynomials. 
\rtxt{
We then perform a numerical study of the new AA(1) residual bounds that were presented in \Cref{sec:AA1-bounds}, showing how the per-iteration residual reduction is strongly influenced by the relationship between the two previous residual vectors, as is implied by \cref{thm:AA1-bounds}. 
}
This is followed by a brief discussion of some numerical stability considerations when solving the AA least-squares problems, related to regularization schemes that have been proposed in the recent literature.
Finally, we broaden our view and investigate how the asymptotic convergence behavior we observe for AA(1) in the case of linear problems extends to a nonlinear problem. In our numerical tests we use Matlab's $QR$ factorization code to solve least-squares problem (\ref{eq:Andersonbetas}) in a robust manner.}

We first introduce some convergence terminology for our discussion, see, e.g., \cite{LinearacAA}.
\begin{definition}[$r$-linear convergence]
Let $\{x_k\}$  be any sequence that converges to $x^*$. Define
\begin{equation*}
  \rho_{\{x_k\}} = \limsup\limits_{k\rightarrow \infty}\|x_k-x^*\|^{\frac{1}{k}}.
\end{equation*}
We say $\{x_k\}$ converges $r$-linearly with $r$-linear convergence factor $\rho_{\{x_k\}}$ if $\rho_{\{x_k\}}\in(0,1)$ and $r$-superlinearly if $\rho_{\{x_k\}}=0$. The ``r-'' prefix stands for ``root''.
\end{definition}

\begin{definition}[r-linear convergence of a fixed-point iteration]\label{def:rho-method}
Consider fixed-point iteration $x_{k+1}=q(x_k)$. We define the set of iteration sequences that converge to a given fixed point $x^*$ as
\begin{equation*}
C(q, x^*)=\Big\{  \{x_k\}_{k=0}^{\infty} | \quad x_{k+1} = q(x_k) \textrm{ for } k=0,1,\ldots, \textrm{ and } \lim_{k \rightarrow \infty}x_k=x^*\Big\},
\end{equation*}
and the worst-case r-linear convergence factor over $C(q, x^*)$ is defined as
\begin{equation}\label{eq:r-factor-defi}
  \rho_{q,x^*} = \sup \Big\{ \rho_{\{x_k\}} |\quad  \{x_k\}\in C(q, x^*) \Big\}.
\end{equation}
We say that the FP method converges r-linearly to $x^*$ with r-linear convergence factor $\rho_{q,x^*}$ if $\rho_{q,x^*} \in(0,1)$.
\end{definition}

We define the root-averaged error sequence of $\{x_k\}$ converging to $x^*$ as
\begin{equation}\label{eq:def_numer-r-convergence}
  \sigma_k = \|x_k-x^*\|_2^{\frac{1}{k}}.
\end{equation}

Two of the numerical tests we consider in this section were previously discussed in \cite{LinearacAA}, which is a companion paper to this paper that discusses continuity and differentiability of the iteration function of AA($m$) viewed as a fixed-point method. Paper \cite{LinearacAA} identifies and sheds interesting light on AA($m$) convergence patterns  for these problems, including oscillating behavior of $\beta_k$ as $k\rightarrow \infty$, and strong dependence of the root-linear asymptotic convergence factors of AA($m$) sequences $\{x_k\}$ on the initial guess. Here we further explain some of these observations using the theoretical results from this paper.
\subsection{AA(1) for a linear system}\label{subsec:linear-problem}

In our numerical tests, we first consider a linear example:
\begin{problem}\label{prob:linear2x2-upper-tri}
Consider linear FP iteration function
\begin{equation}\label{eq:q-linear-2x2}
  q(x)  =Mx=\begin{bmatrix}
  \frac{2}{3} & \frac{1}{4}\\
  0  & \frac{1}{3}
  \end{bmatrix}
  x,
\end{equation}
with fixed point  $x^*  = (0,0)^T$ (see also \cite{LinearacAA}).
\end{problem}

\cref{AA-Linear-plot} shows $\sigma_k$ for FP and AA(1), where we use three different initial guesses, $(0.2,0.3)^T, \ (0.2,1)^T$, and $(0.2,-0.2)^T$  for both FP and AA(1). Note that for FP, $\lim_{k \rightarrow\infty }\sigma_k =\rho_{q,x^*}=\rho(q'(x^*))=\frac{2}{3}$. However, for AA(1), $\lim_{k \rightarrow \infty }\sigma_k =\rho_{\{x_k\}}$ depends on the initial condition.

\begin{figure}[h]
\centering
\includegraphics[width=.6\textwidth]{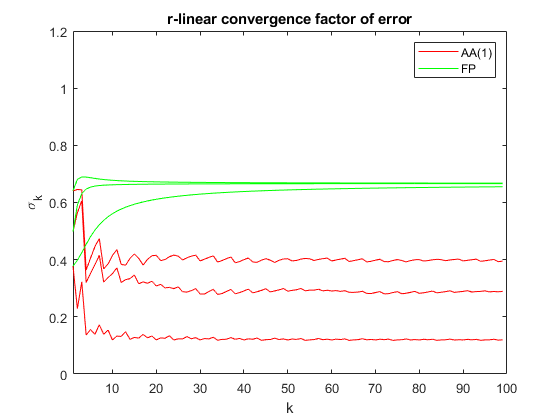}
\caption{\cref{prob:linear2x2-upper-tri} (linear). Root-averaged error $\sigma_k$ as a function of iteration number $k$ for different initial guesses.} \label{AA-Linear-plot}
\end{figure}

\begin{figure}[h]
\centering
\includegraphics[width=.49\textwidth]{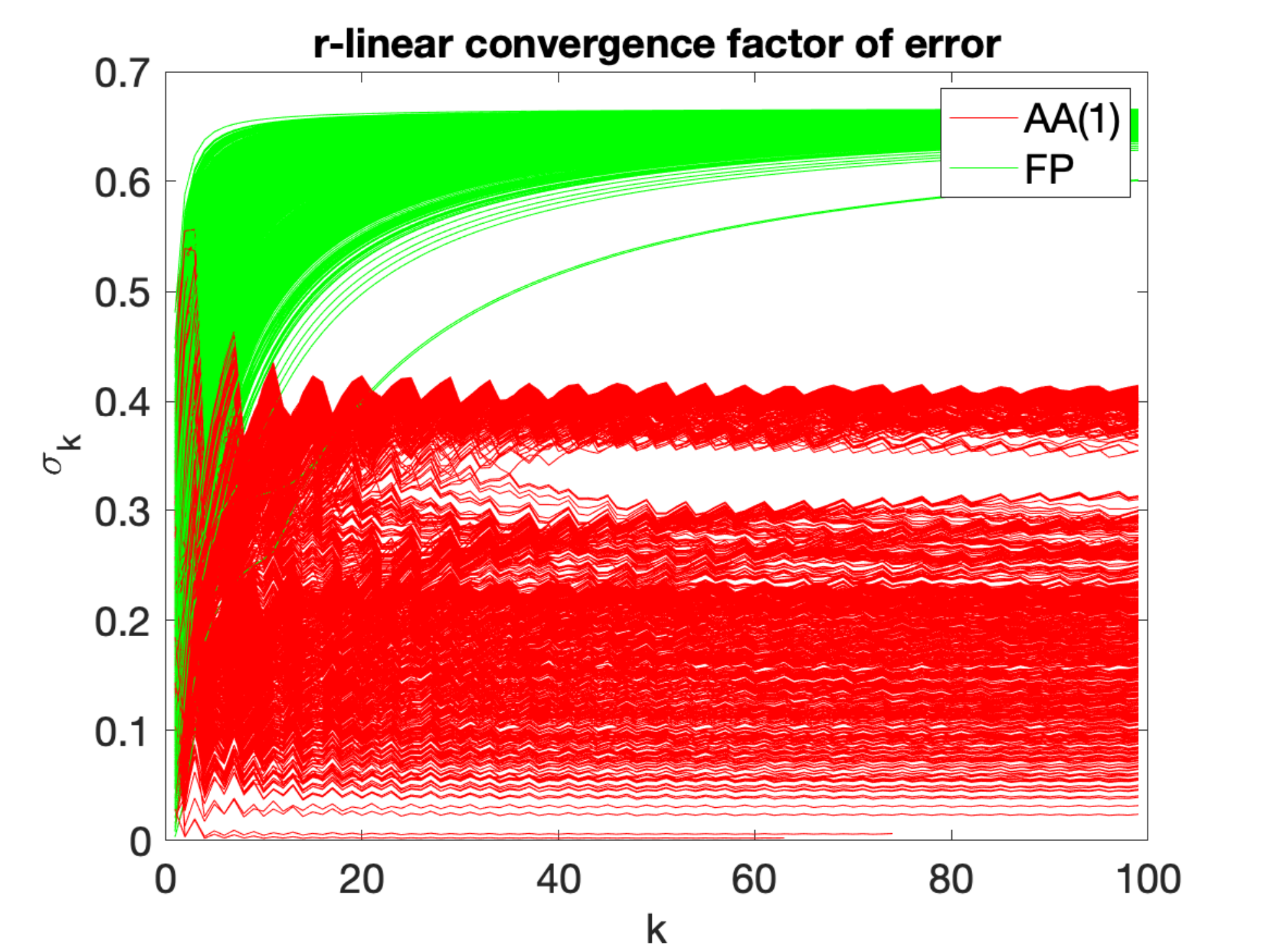}
\includegraphics[width=.49\textwidth]{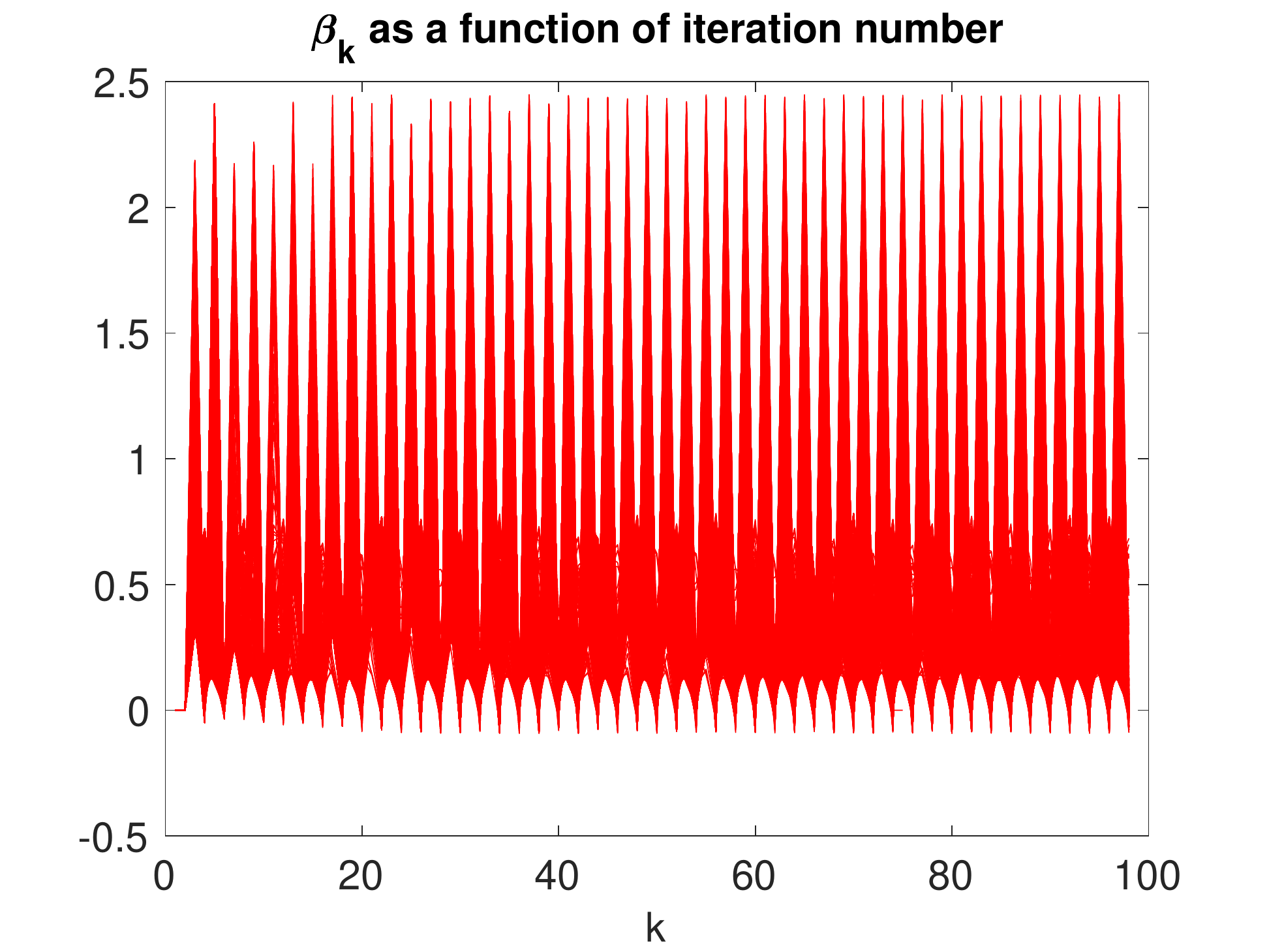}
\caption{Monte Carlo tests for \cref{prob:linear2x2-upper-tri} (linear). (Results from \cite{LinearacAA}.)} \label{prob2-mc}
\end{figure}

Monte Carlo results with a large number of random initial guesses in $[-1,1]^2$ for \cref{prob:linear2x2-upper-tri} with  FP and AA(1) iterations  are shown in  \cref{prob2-mc} (see also \cite{LinearacAA}). \cref{prob2-mc} indicates that AA(1) sequences $\{x_k\}$ converge $r$-linearly. However, the $r$-linear convergence factors $\rho_{\{x_k\}}$ strongly depend on the initial guess on a set of nonzero measure (see \cite{LinearacAA} for more discussion). There seems to be a least upper bound $\rho_{AA(1),x^*}$ for $\rho_{\{x_k\}}$  for the AA(1) iteration \cref{eq:anderson-1-step} that is smaller than the $r$-linear convergence factor $\rho_{q,x^*}=2/3$ of fixed-point iteration \cref{eq:fixed-point} by itself, but as far as we know, there are no theoretical results that allow us to compute $\rho_{AA(1),x^*}$ and to show it is smaller than $\rho_{q,x^*}$. Furthermore, we note that the $\beta_k$ sequences oscillate for this linear problem when the AA(1) iteration approaches $x^*$, and $\beta_k>-1$, which is consistent with our \cref{thm:lower-boound-betak}.

\begin{figure}[h]
\centering
\includegraphics[width=.7\textwidth]{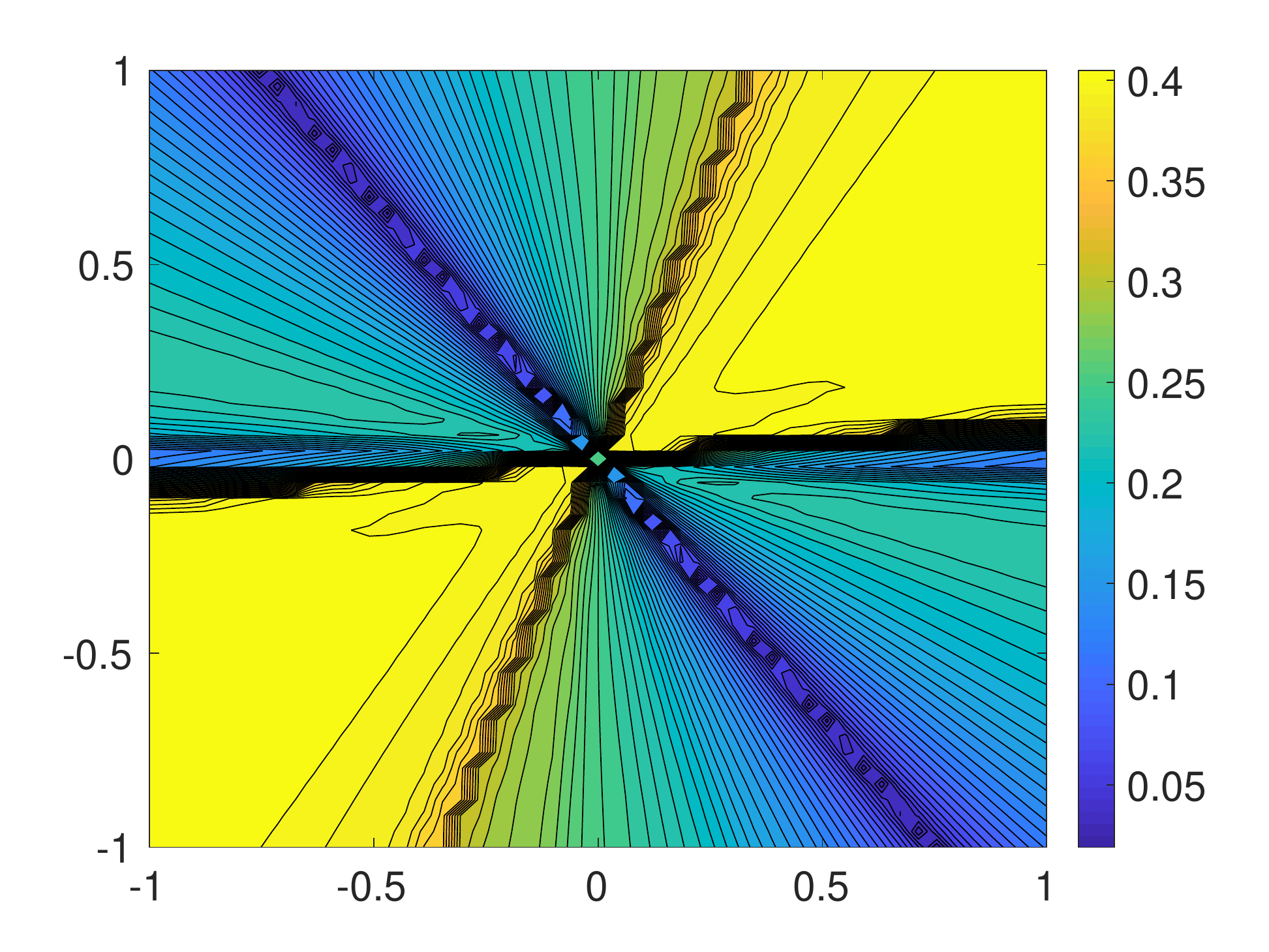}
\includegraphics[width=.7\textwidth]{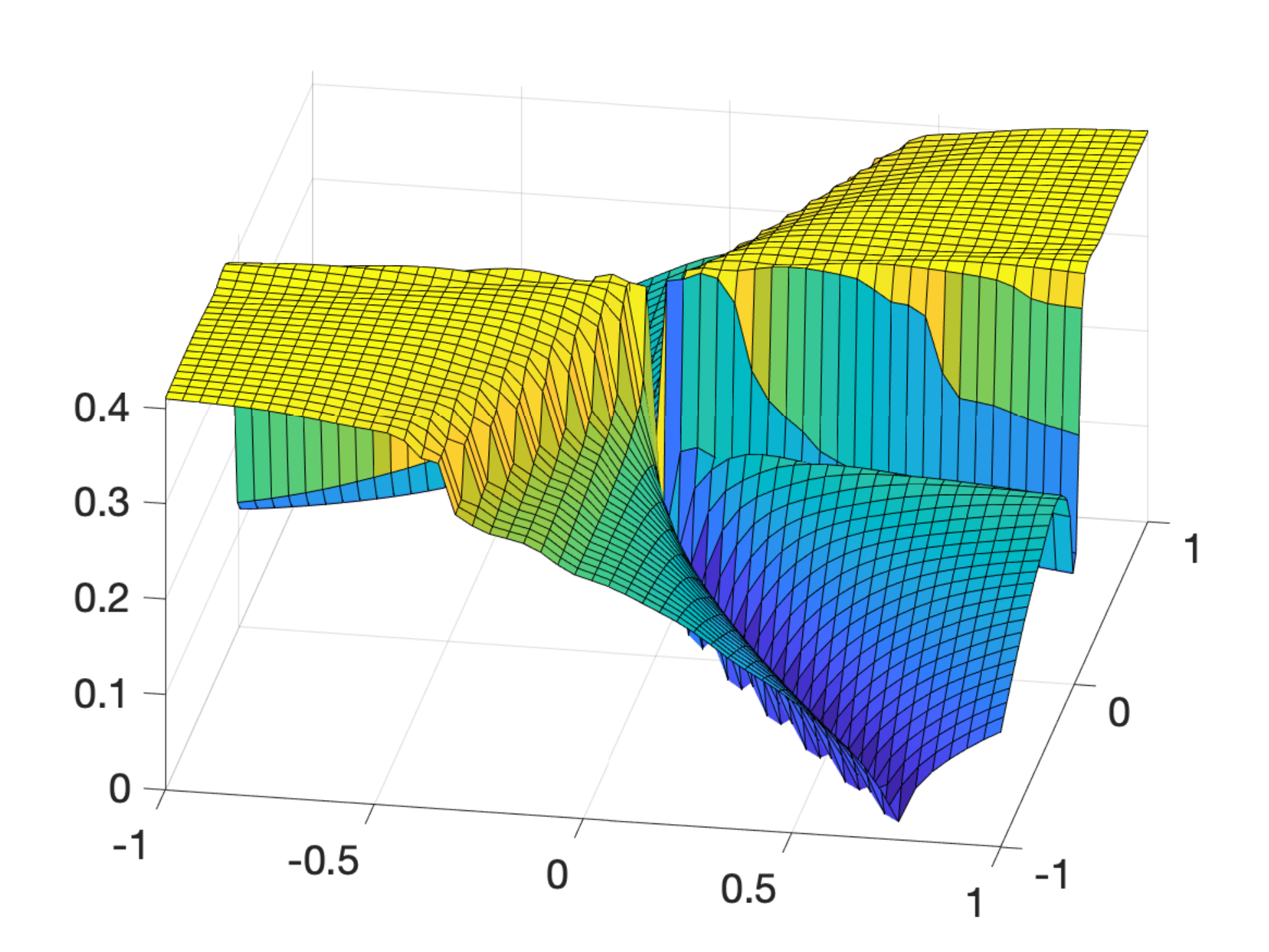}
\caption{Asymptotic convergence factor ($\sigma_k$ for large $k$) as a function of initial condition for \cref{prob:linear2x2-upper-tri} (linear).} \label{prob2-grid}
\end{figure}

In \cref{prob2-grid} we show the convergence factor of AA(1) applied to the linear \cref{prob:linear2x2-upper-tri} for different initial guesses $x_0$ on a regular grid with 50 by 50 points and $x_1=q(x_0)$. Again, we see  that  a least upper bound $\rho_{AA(1),x^*}$ for $\rho_{\{x_k\}}$ exists for the AA(1) iteration \cref{eq:anderson-1-step}. We see, however, a clear pattern in the numerically determined asymptotic convergence factors for the AA(1) sequences $\{x_k\}$, with radial invariance that we investigate further below. We also see in \cref{prob2-grid} that the apparent gap  in the convergence factor spectrum in \cref{prob2-mc} corresponds to a steep, possibly discontinuous jump in the convergence factor surface in the initial condition plane.

The radial invariance in  \cref{prob2-grid} validates our \cref{scaling-initial-guess-polynomial-invariant,pro:scaling-invariant-limit}.
To investigate the radial  dependence in \cref{prob2-grid} further, we take $x_0=(\cos(\theta),\sin(\theta))^T$ and for AA(1), and plot in \cref{prob2-theta} the numerical convergence factor for large $k$ as a function of $\theta$.

\begin{figure}[h]
\centering
\includegraphics[width=.6\textwidth]{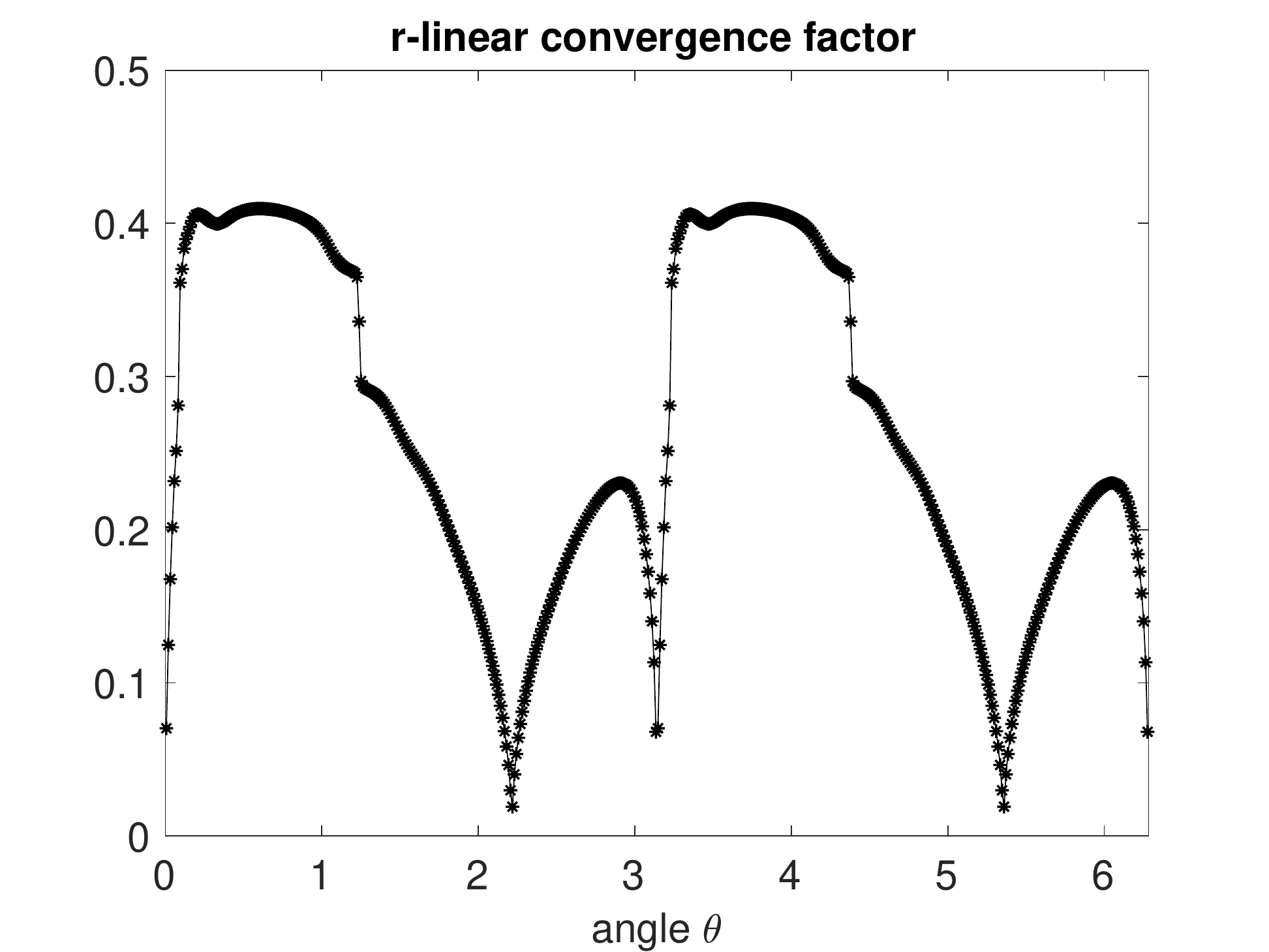}
\caption{Asymptotic convergence factor as a function of initial condition angle $\theta$ for \cref{prob:linear2x2-upper-tri} (linear).} \label{prob2-theta}
\end{figure}
%
%

Note that, for \cref{prob:linear2x2-upper-tri}, $A=I-M$ is given by
\begin{equation*}
  A = \begin{bmatrix}
  \frac{1}{3} & -\frac{1}{4}\\
  0 & \frac{2}{3}
  \end{bmatrix}.
\end{equation*}
%
The largest eigenvalue of $A$ is $\frac{2}{3}$ with eigenvector $v_{\max}=(1,-4/3)^T$, and $v_{\min}=(1,0)^T$ corresponds to the smallest eigenvalue $\frac{1}{3}$. \cref{prob2-grid,prob2-theta} also validate \cref{AA1-2-iterations-eigevector}: the convergence factors decay to 0 for initial conditions in the direction of the eigenvectors of $A$, corresponding to exact convergence in two steps.

\btxt{Next, we consider the standard AA(1) scheme of iteration \cref{eq:AA-iteration} for $m=1$ with single initial guess $x_0$ and $x_1=q(x_0)$, and compare it with the non-standard version (\ref{eq:general-guess-AAm}) with $m=1$ that uses two initial guesses $x_0$ and $x_1$. We want to explore whether there can be an advantage of the latter approach in terms of the asymptotic convergence factor. To this end, we we consider the non-standard AA(1) scheme with $x_0=(\cos \theta_1, \sin \theta_1)^T$ and $x_1= \alpha \, (\cos \theta_2, \sin \theta_2)^T$, and compute the asymptotic convergence factor in a grid search over $50^3$ values for $\theta_1$, $\theta_2$ and $\alpha$ for \cref{prob:linear2x2-upper-tri}, see the left panel of \cref{prob2-x0-x1}. This is compared with a grid search for the standard AA(1) scheme with $x_0=(\cos \theta_1, \sin \theta_1)^T$ and $x_1=q(x_0)$, as for \cref{prob2-grid,prob2-theta}, see the right panel of \cref{prob2-x0-x1}.}

\begin{figure}[h]
\centering
\includegraphics[width=.49\textwidth]{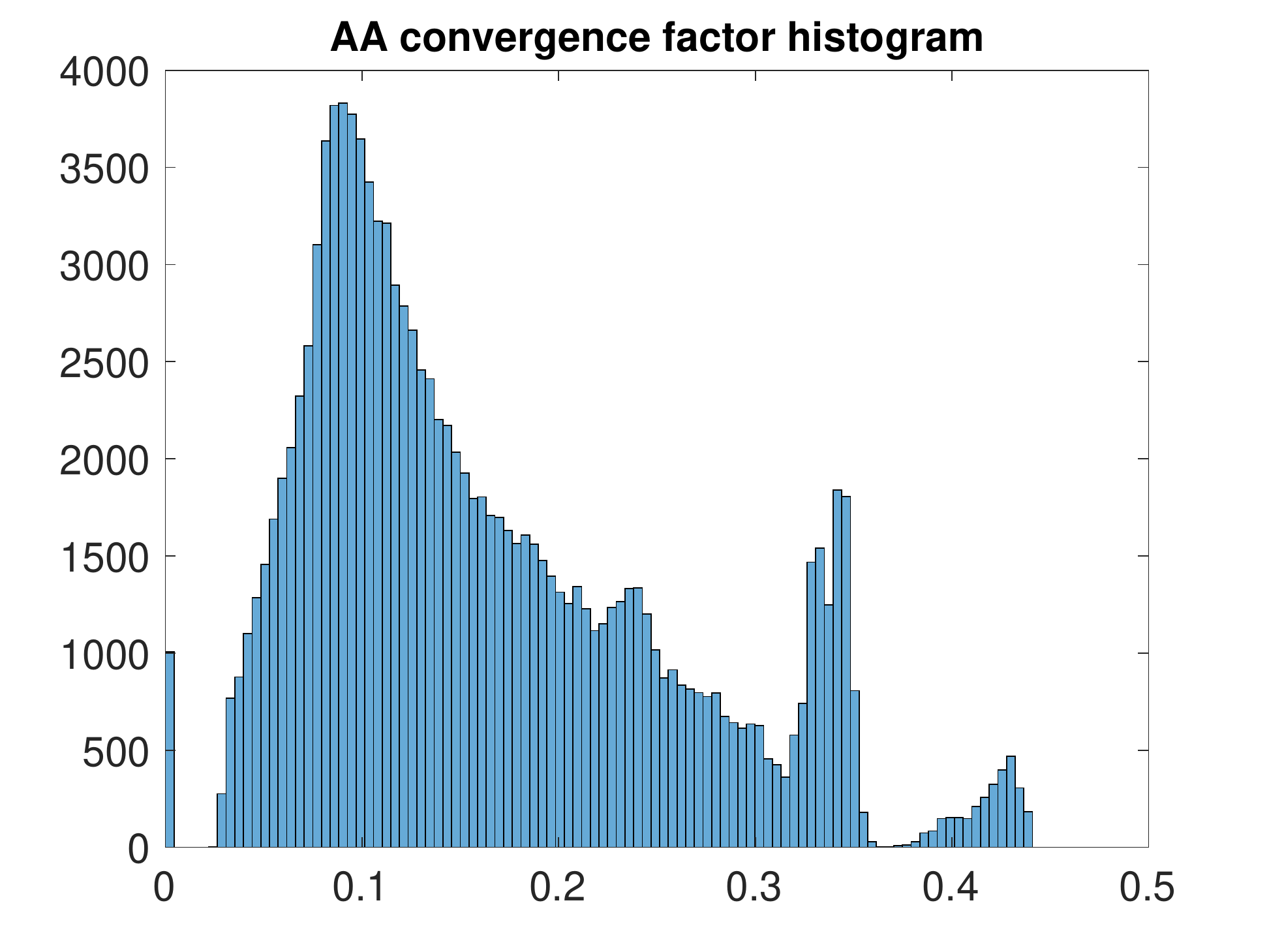}
\includegraphics[width=.49\textwidth]{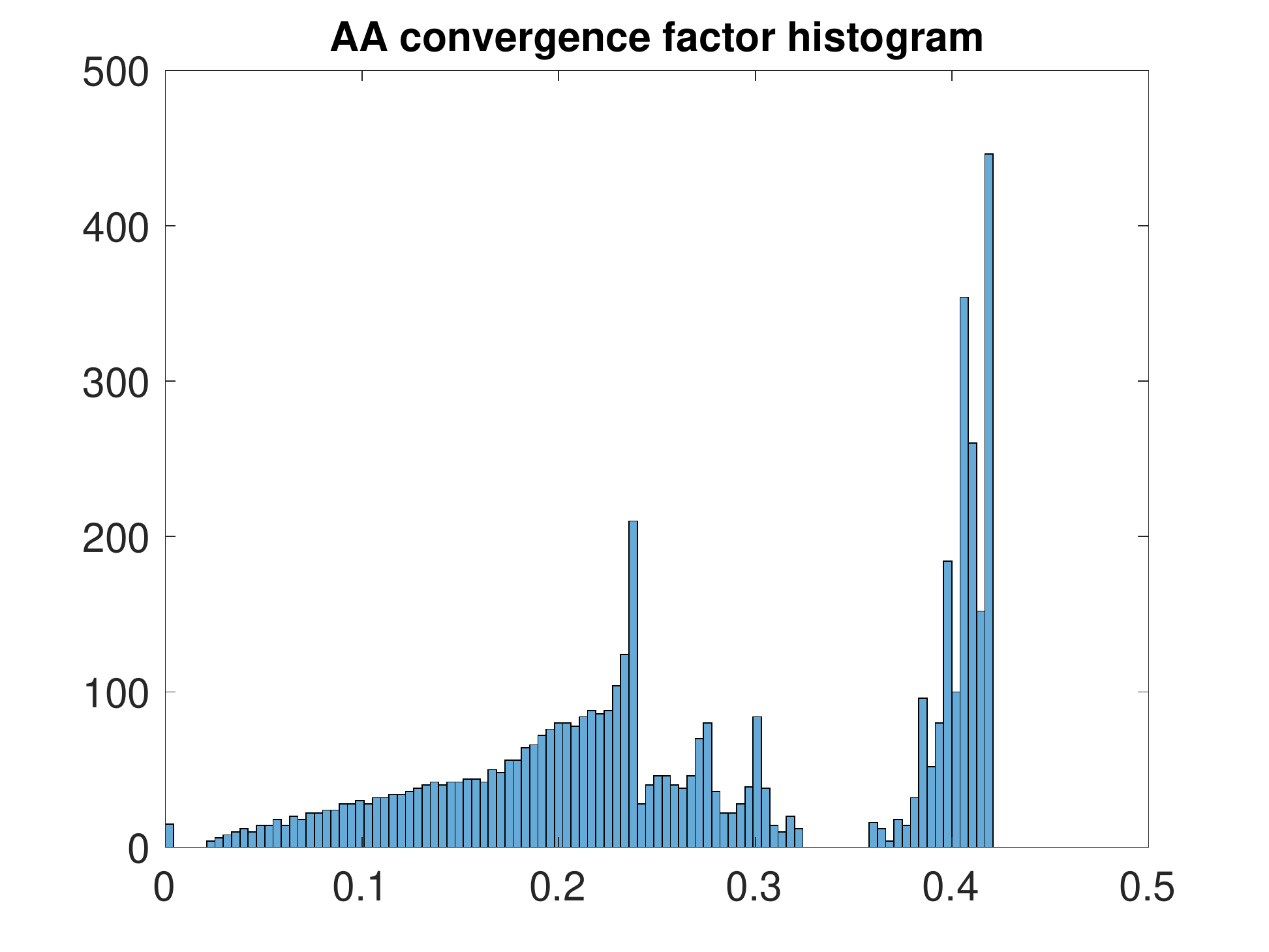}
\caption{\cref{prob:linear2x2-upper-tri} (linear). Left: Histogram of asymptotic convergence factors for $50^3=125,000$ initial condition pairs
$x_0=(\cos \theta_1, \sin \theta_1)^T$ and $x_1= \alpha \, (\cos \theta_2, \sin \theta_2)^T$, where $\theta_1$ and $\theta_2$ take on 50 equally spaced values $2k\pi/50$, $k=0, 1,\ldots,49$, and $\alpha$ takes on 50 equally spaced values
$10k/50$, $k=1, \ldots,50$.
The largest convergence factor value is $\approx$ 0.440, which is slightly greater than the largest value, 0.410, in Fig.\ \ref{prob2-theta} where $x_0=(\cos \theta, \sin \theta)$ and $x_1=Mx_0$. Right:   Histogram of asymptotic convergence factors for 5,000 initial condition
$x_0=(\cos \theta_1, \sin \theta_1)$ and $x_1=q(x_0)$, as in \cref{prob2-mc,prob2-grid,prob2-theta}.} \label{prob2-x0-x1}
\end{figure}
For the choice where  $x_0=(\cos \theta_1, \sin \theta_1)^T$ and $x_1= \alpha \, (\cos \theta_2, \sin \theta_2)^T$ with 125,000 initial condition pairs, the average convergence factor is  0.164131496689178.
For the choice where $x_0=(\cos \theta_1, \sin \theta_1)^T$ and $x_1= q(x_0)$ with 5,000   initial guesses, the average  convergence factor is 0.273090075503153. Note that increasing the number of initial condition points does not have  a  big influence on these results.  It is clear that the first choice gives a better average convergence factor. So while $x_1=q(x_0)$ appears to give a slightly  better worst-case performance, the average convergence factor appears to be better when $x_1$ is chosen randomly as in the experiment, independent of $x_0$.

\begin{figure}[h]
\centering
\includegraphics[width=.49\textwidth]{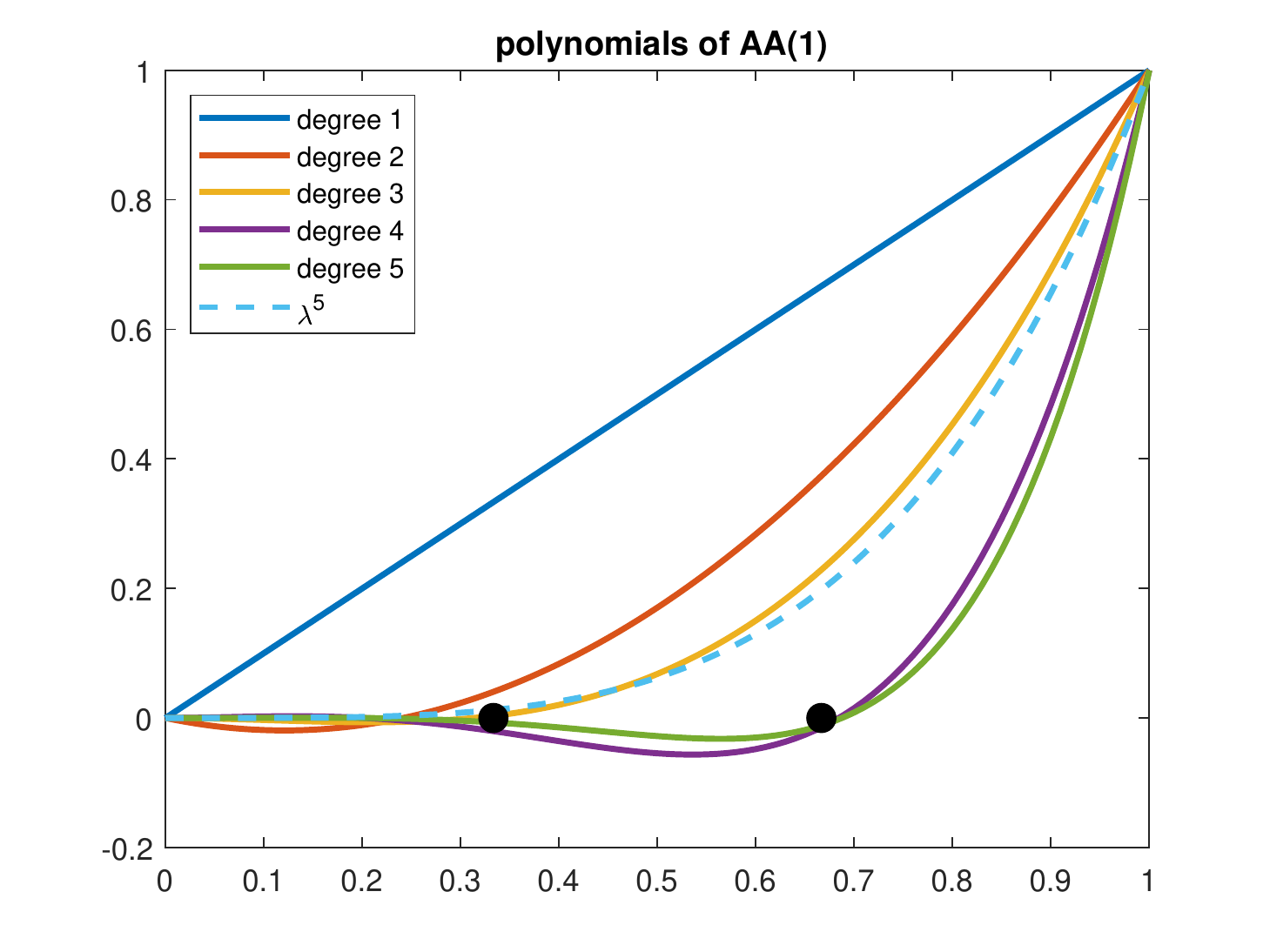}
\includegraphics[width=.49\textwidth]{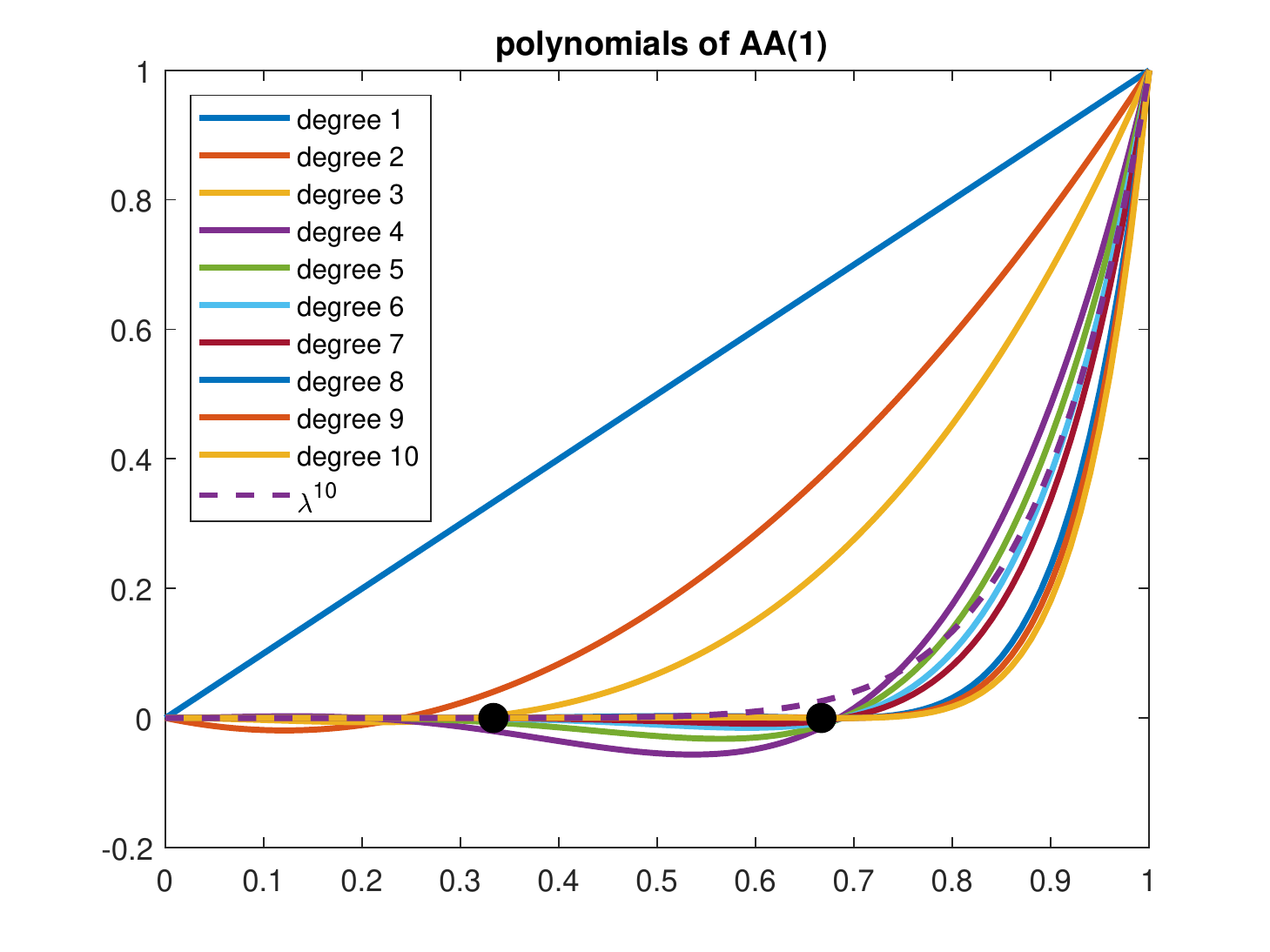}
\caption{Residual polynomials of AA(1) for \cref{prob:linear2x2-upper-tri} (linear).} \label{prob2-poly}
\end{figure}

%
%
Finally, we investigate the polynomials \cref{eq:AA1-three-term-M} for the residual of AA(1) with $x_1=q(x_0)$ applied to  \cref{prob:linear2x2-upper-tri}.   It is easy to show that the residual $r_k$ for fixed-point iteration $x_{k+1}=q(x_k)$ in the linear case is given by $r_k=M^k r_0$. We take initial guess $x_0=(0.2,0.3)^T$. \cref{prob2-poly} presents the AA(1) residual polynomials from \cref{eq:rk-polynomial-form}, and $\lambda^k$  for the fixed-point method.  The filled circles are the two eigenvalues $\lambda=\frac{1}{3}, \frac{2}{3}$ of matrix $M$.  We can see that $p_5(\lambda)$ of AA(1) is a better polynomial than $\lambda^5$ at the eigenvalues of $M$, and $p_{10}(\lambda)$ of AA(1) is a better  polynomial than $\lambda^{10}$ at the eigenvalues of $M$, illustrating the mechanism by which AA(1) accelerates the convergence of the FP method.

\rtxt{
\subsection{AA(1) residual bounds for linear systems}\label{subsec:linear-bounds}

To investigate the AA(1) residual bounds from \Cref{sec:AA1-bounds}, we now consider a second linear example:
\begin{problem}\label{prob:linear2x2-diag}
Consider linear FP iteration function
\begin{equation}\label{eq:q-linear-2x2-diag}
  q(x)  =Mx=\begin{bmatrix}
  0.5784 & 0\\
  0  & 0.999
  \end{bmatrix}
  x,
\end{equation}
with fixed point  $x^*  = (0, 0)^T$. 
\end{problem}
%
We consider the initial guess $x_0  = (0.0001,0.3023)^T$ (this was chosen arbitrarily), and show in \cref{fig:rnorm_FP_vs_AA1} the residual histories for both the basic FP iteration \eqref{eq:fixed-point} and the AA($m$) iteration \eqref{eq:AA-iteration} with $m = 1$ applied to this problem.
Evidently, the FP iteration converges excruciatingly slowly, which is to be expected because the asymptotic convergence factor of the method $\rho(M) =0.999$ is only very slightly less than unity.
In stark contrast, the averaged root-linear convergence speed of AA(1) is much faster, even though locally AA(1) can exhibit convergence rates, in a quasi-periodic fashion, which are seemingly no faster than those of the FP iteration. 
This quasi-periodic behaviour is somewhat reminiscent of quasi-periodic convergence patterns that can be observed for restarted GMRES($m$) applied to symmetric matrices, but that are only partially understood \cite{Baker_etal_2005}.

\begin{figure}[h]
\centerline{
\includegraphics[scale=0.36]{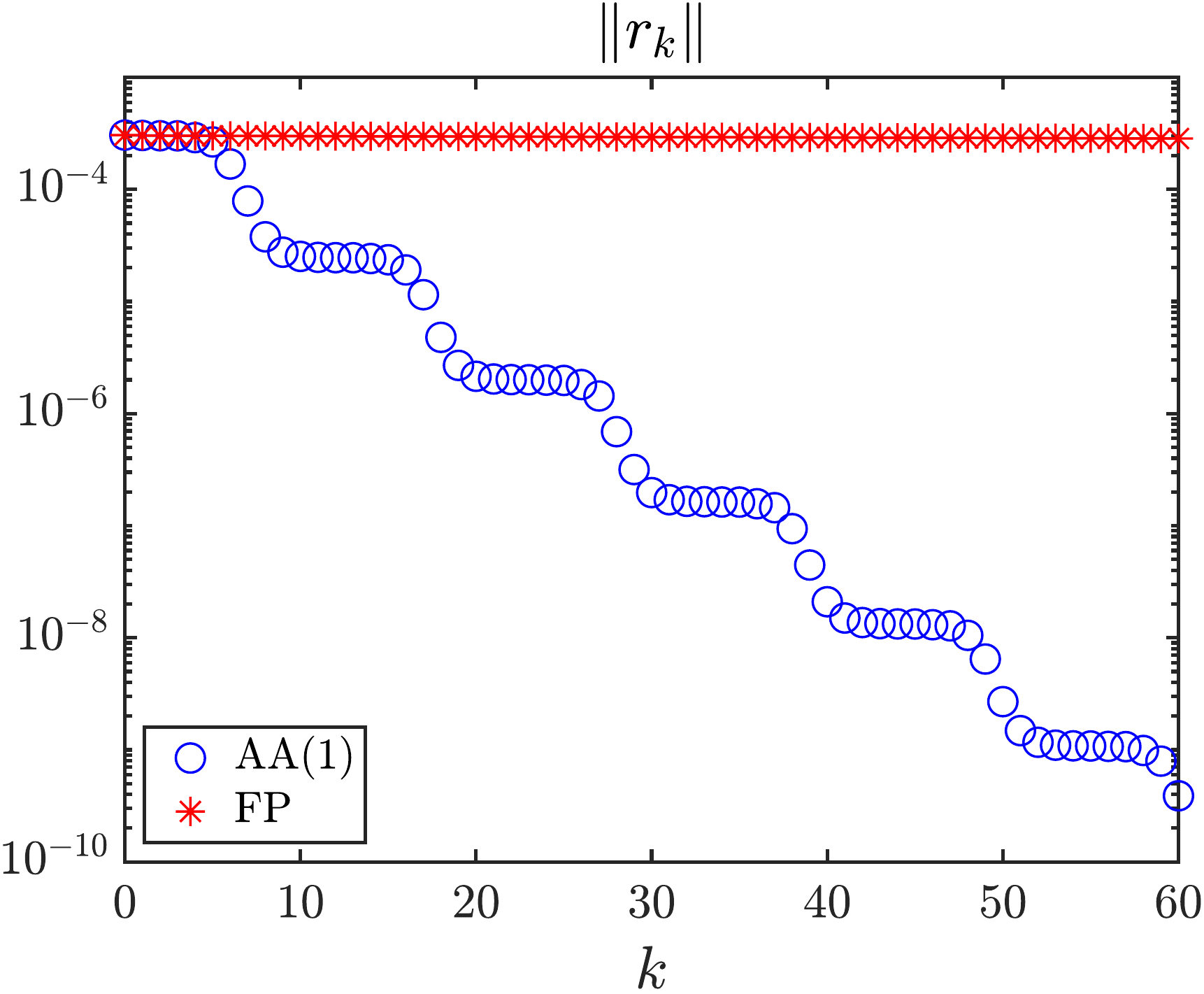}
\quad
\includegraphics[scale=0.36]{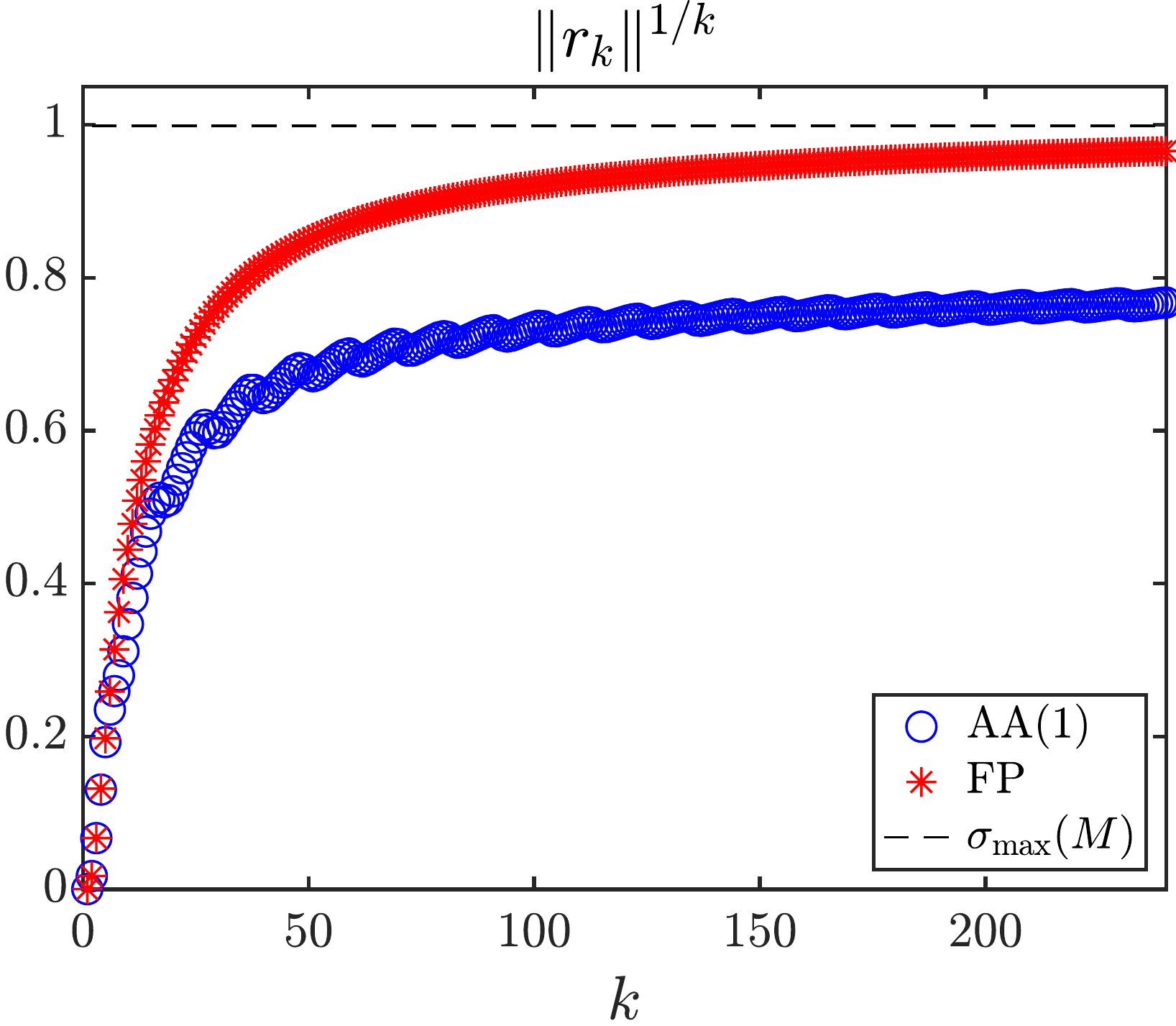}
}
\caption{Residual history for the basic FP method \eqref{eq:fixed-point} and the AA(1) method applied to \cref{prob:linear2x2-diag}. 
At left, the norm of the residual as a function of iteration index $k \in [0,60]$.
At right, the averaged root-linear residual convergence factor for iteration index $k \in [1, 240]$; note that $\sigma_{\max}(M) = 0.9990$ is shown for reference.
%
%
\label{fig:rnorm_FP_vs_AA1}
}
\end{figure}

\begin{figure}[h]
\centerline{
\includegraphics[scale=0.35]{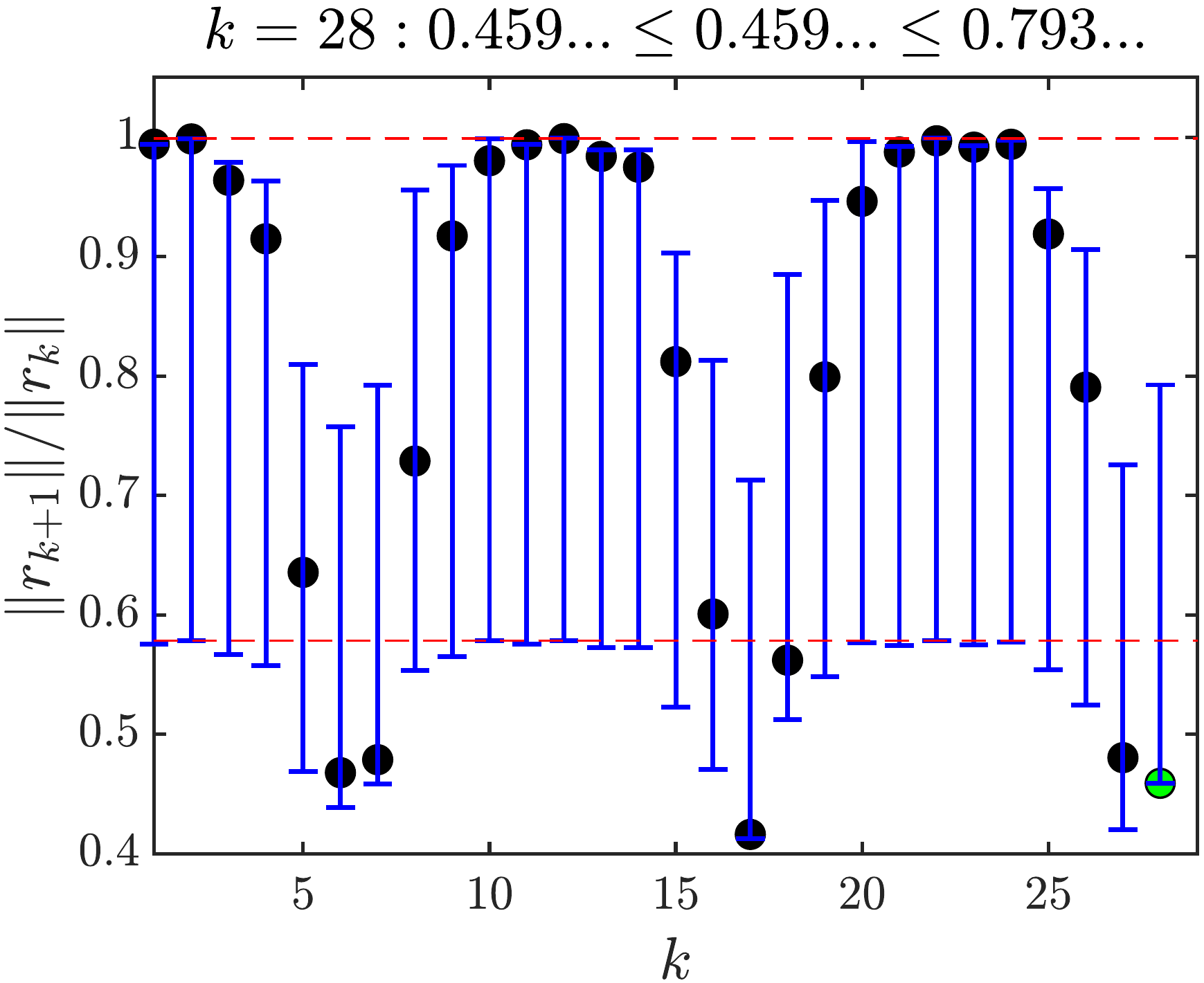}
\quad
\includegraphics[scale=0.35]{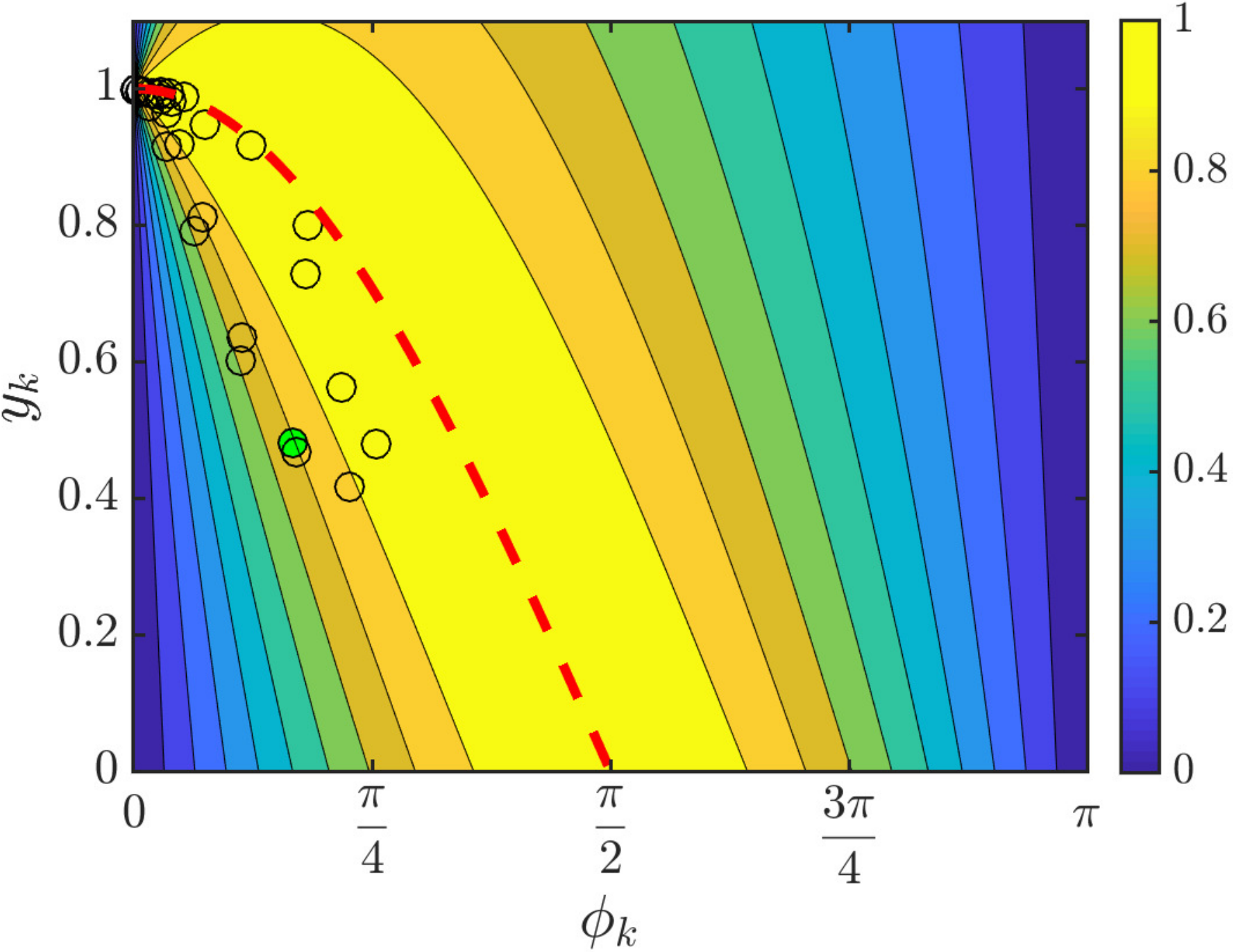}
}
\vspace{2ex}
\centerline{
\includegraphics[scale=0.35]{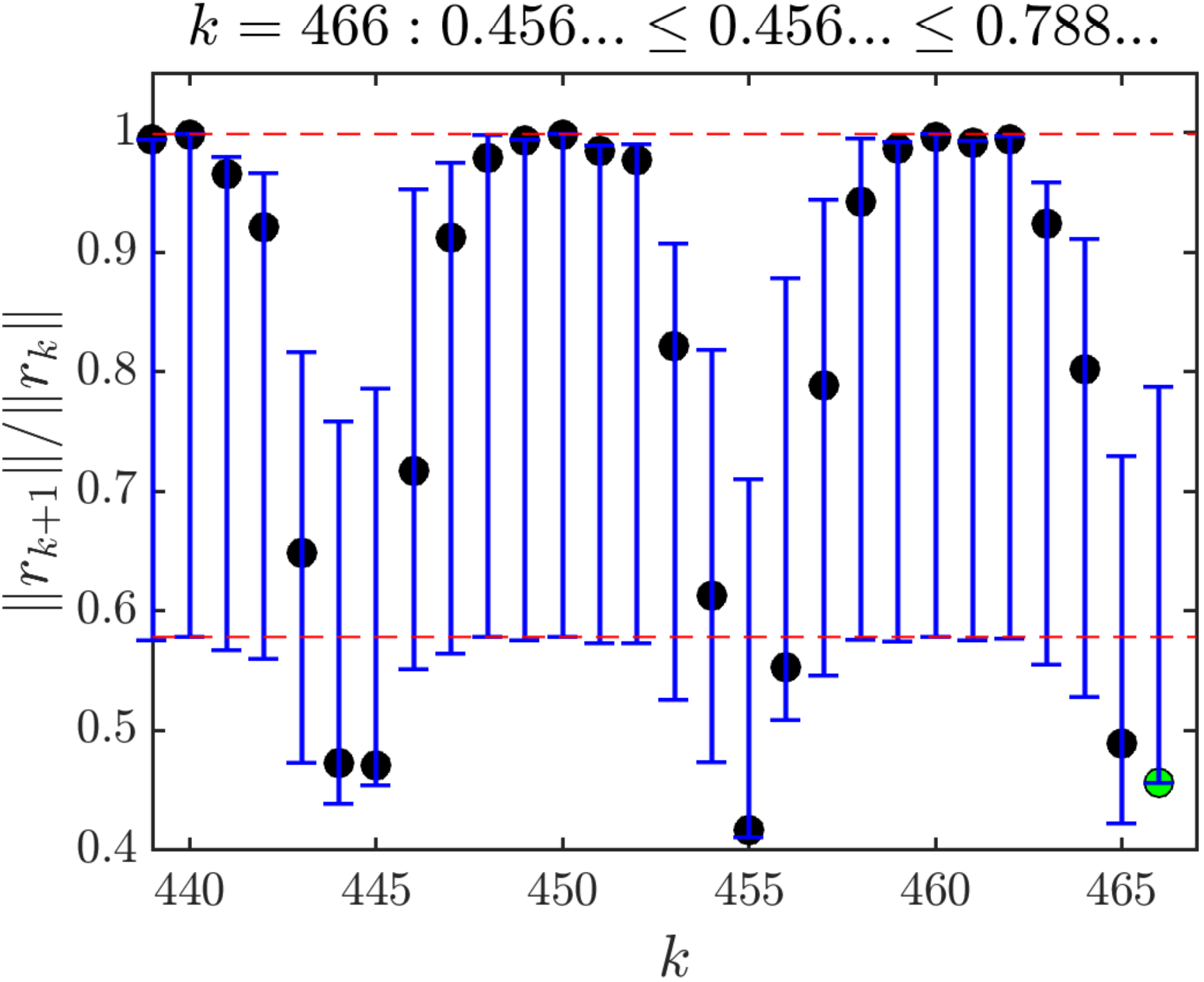}
\quad
\includegraphics[scale=0.35]{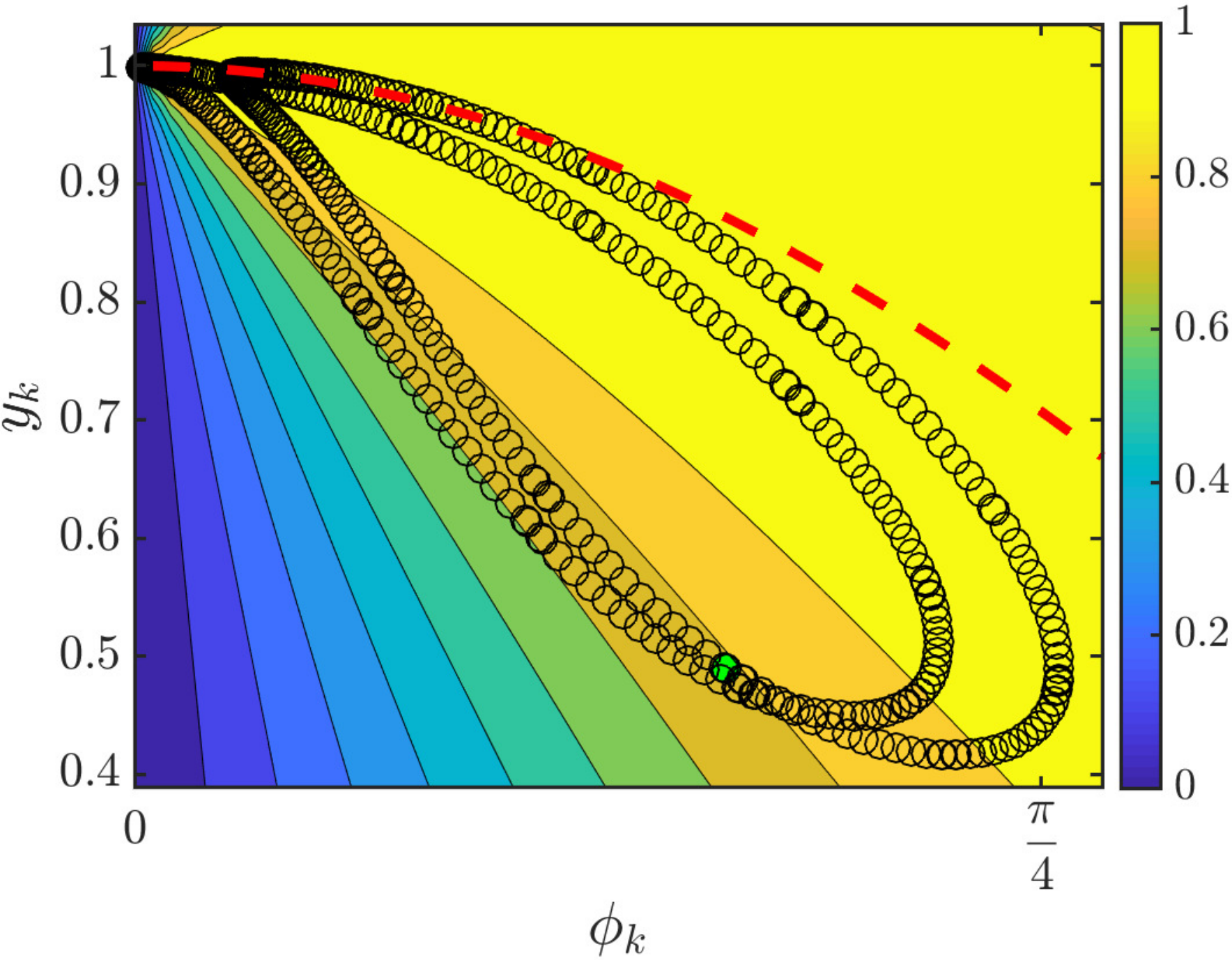}
}
\caption{Illustration of AA(1) residual bounds from \Cref{sec:AA1-bounds} for applying AA(1) to \cref{prob:linear2x2-diag}. 
At left, the AA(1) per-iteration residual reduction $\Vert r_{k+1} \Vert / \Vert r_{k} \Vert = y_{k+1}$ is shown for 28 consecutive iterations, with $k \in [1, 28]$ (top), and $k \in [439, 466]$ (bottom).
For each $k$, the value of $\Vert r_{k+1} \Vert / \Vert r_{k} \Vert$ is shown as a filled circle marker, and the lower and upper bounds of it given in \eqref{eq:rk+1-lower-upper-bound} are shown as blue error bars.
Note that the upper limit of each error bar is a factor of $\sigma_{\max}(M) / \sigma_{\min}(M) \approx 1.73$ times larger than the lower limit.
Dashed red lines in the left panel are the values of the minimum and maximum singular values of $M$.
Also shown in the titles of the plots on the left are, for iterations $k = 28$ (top) and $k = 466$ (bottom), the value of $\Vert r_{k+1} \Vert / \Vert r_k \Vert$ and the bounds from \eqref{eq:rk+1-lower-upper-bound}.
Note that the corresponding values of $\Vert r_{k+1} \Vert / \Vert r_{k} \Vert$ in the plots are represented with green-filled circles.
At right, empty circle markers represent the AA(1) residuals in $\phi_k$--$y_k$ space for iterations $k \in [0, 28]$ (top), and $k \in [0, 466]$ (bottom).
The residual markers corresponding to $k = 28$ (top) and $k = 466$ (bottom) are again colored green.
Note that the bottom right plot is shown over a zoomed in region of $\phi_k$--$y_k$ space so as to better highlight the details.
}
\label{fig:AA1_bounds_exposition} 
\end{figure}

To better understand the relevance of the convergence theory from \Cref{sec:AA1-bounds}, we further consider in \Cref{fig:AA1_bounds_exposition} the convergence behavior of AA(1) for \cref{prob:linear2x2-diag}.
The caption of the figure gives an overview of what is represented.
There are many interesting points to discuss here. 
First, we see in greater detail the almost-periodic behaviour in the per-iteration residual reduction factor $y_{k+1}=\Vert r_{k+1} / \Vert / \Vert r_{k} \Vert$.
For example, from \cref{fig:rnorm_FP_vs_AA1} it appears as though $\Vert r_{k+1} \Vert / \Vert r_{k} \Vert$ will be periodic with period $\approx 10$ (because $\Vert r_k \Vert$ seems to pass through almost six periods in the first $60$ iterations); however, from the top left plot in \Cref{fig:AA1_bounds_exposition}, we see that $\Vert r_{k+1} \Vert / \Vert r_{k} \Vert$ is not quite periodic over the first $\approx 30$ iterations.
Instead, it appears as though $\Vert r_{k+1} \Vert / \Vert r_{k} \Vert$ may exhibit some almost-periodic behavior with period much larger than $\approx 10$, because the bottom left plot looks essentially identical to the top left plot, despite it corresponding to 438 iterations later.
Note that residual reduction factors in the two left plots are slightly different, however, as can be seen by comparing their titles: $\Vert r_{29} \Vert / \Vert r_{28} \Vert = 0.459 \ldots \neq 0.456 \ldots = \Vert r_{467} \Vert / \Vert r_{466} \Vert$, so the behavior is not quite periodic even on the longer timescale. 
From the bottom right panel in \Cref{fig:AA1_bounds_exposition} observe that, as the AA(1) iteration proceeds, the residuals appear confined to a specific region of $\phi_k$--$y_k$ space, and moreover they appear to be filling in two distinct continuous curves, which are reminiscent of limit cycles. We note that 
between consecutive iterations the residual oscillates between these two curves.
Notice also that if the residual was genuinely periodic in these first 466 iterations rather than almost-periodic, then markers would overlap more often and these two curves would not continue being traced out as the iteration proceeds. 

Considering again the left panel in \Cref{fig:AA1_bounds_exposition}, we see that the upper red dashed line that represents the per-iteration residual reduction upper bound $\Vert r_{k+1} \Vert / \Vert r_{k} \Vert \leq \sigma_{\max} (M)$, which may be derived from results in \cite{toth2015} and \cite{evans2020proof}, is pessimistic compared to the new upper residual bound (top blue horizontal line segments).
Interestingly, we also see that both the lower and upper bounds on $\Vert r_{k+1} \Vert / \Vert r_k \Vert$ given in \eqref{eq:rk+1-lower-upper-bound} are relevant, with each being approximately reached at certain iterations. 
On one hand, it is perhaps somewhat surprising that the upper bound can be so pessimistic/the lower bound can be tight, but on the other hand, we know from \Cref{cor:AA1-equality} that both of these bounds are tight in the limit that the condition number of $M$ approaches unity.

Recall from the proof of \cref{thm:AA1-bounds} that the AA(1) residual can be written as $r_{k+1} = M \alpha$, in which $\alpha$ is a vector depending upon $r_k$ and $r_{k-1}$---specifically, see \eqref{eq:S_alpha_form}---,
and that the particular bounds given in \eqref{eq:rk+1-lower-upper-bound} take the form $\sigma_{\min}(M) \Vert \alpha \Vert \leq \Vert r_{k+1} \Vert \leq \sigma_{\max}(M) \Vert \alpha \Vert$.
Clearly, the lower and upper bounds can achieve equality if and only if $\alpha$ is the smallest singular vector or largest singular vector of $M$, respectively (note that the singular vectors and singular values of $M$ in  \Cref{prob:linear2x2-diag} are equal to its eigenvectors and eigenvalues).
Thus, on iterations in the left panel of \Cref{fig:AA1_bounds_exposition} where the lower bound from \eqref{eq:rk+1-lower-upper-bound} is approximately reached, it must be the case that $\alpha$ is approximately the smallest singular vector of $M$, and conversely, on iterations where the upper bound from \eqref{eq:rk+1-lower-upper-bound} is approximately reached, it must be that $\alpha$ is approximately the largest singular vector of $M$. 

The dynamic behavior described above for the AA(1) residual vectors is very different from the behavior of the simple FP iteration, $r_{k+1} = M r_k$. For \Cref{prob:linear2x2-diag} this simple FP iteration is in essence just the power method, with it driving $r_{k}$ to the largest singular vector of $M$ as $k$ increases. 
Thus, while the residuals of the FP iteration do satisfy the local bounds $\sigma_{\min}(M) \leq \Vert r_{k+1} \Vert / \Vert r_k \Vert \leq \sigma_{\max}(M)$, within a very small number of iterations $\Vert r_{k+1} \Vert / \Vert r_k \Vert$ quickly approaches $\sigma_{\max}(M)$.
We believe that the dynamic convergence behavior that AA(1) can exhibit, as seen in \Cref{fig:rnorm_FP_vs_AA1,fig:AA1_bounds_exposition}, for example, is in essence what makes rigorously quantifying its asymptotic convergence rate so challenging.

}

\btxt{
\subsection{Stability considerations when solving the AA least-squares problems}\label{subsec:stability}
We now discuss a practical point that merits some discussion in the context of this paper, namely,
the possible use of regularization in solving least-squares problem
\cref{eq:Andersonbetas} as $x_k$ approaches $x^*$.
Recently, it has been argued in the machine learning literature on Anderson acceleration \cite{scieur2016regularized}
that least-squares problem \cref{eq:Andersonbetas} needs to be regularized by adding a diagonal regularization matrix to the
normal equation matrix $R_k^T R_k$ in (\ref{eq:AAm-beta-form}), when $x_k$ approaches $x^*$,
due to the matrix $R_k^T R_k$ in (\ref{eq:AAm-beta-form}) becoming increasingly singular, see also \cite{washio1997krylov,fu2020anderson,henderson2019damped}.
Here we will refute this for AA.
As already indicated in the discussion on stable methods to solve nearly rank-deficient least-squares problems
in \Cref{sec:intro}, such regularization is not necessary when using appropriate well-established methods to solve the
potentially rank-deficient least-squares problem numerically, see 
 \cite{fang2009two,moler2004numerical,walker2011anderson,lockhart2022performance}.
Moreover, we will show numerically that adding regularization as in \cite{scieur2016regularized} to AA($m$) may be detrimental in that it may destroy the asymptotic convergence advantages
of AA($m$) when accurate solutions are sought.


In \cref{sweep} we compare the $QR$-based approach we use in our implementation with
solving a regularized version of  least-squares problem \cref{eq:Andersonbetas}, for a
variation of \cref{prob:linear2x2-upper-tri} where
\begin{equation}\label{eq:q-linear-2x2-simple-RHS}
  x_{k+1} = M (x_k-b) + b, \qquad
   M=\begin{bmatrix}
  8/9 & 1/4\\
  0  & 2/3
  \end{bmatrix}, \quad
  b=\begin{bmatrix}
  1\\
  1
  \end{bmatrix},
\end{equation}
and $x^*=b$. \rtxt{The initial guess is $x_0 = (1.2, 1.3)^T$.} I think this is the initial guess from the code I saw.

The tests with regularization use
\begin{equation}\label{eq:AAm-beta-form-reg}
\boldsymbol{\beta}^{(k)}  = -(R_k^TR_k + \lambda I)^{-1} R_k^Tr_k,
\end{equation}
where $\lambda>0$ is a small regularization parameter, see, e.g., \cite{washio1997krylov,scieur2016regularized}.
Note that the choice of a proper $\lambda$ is difficult a priori but can be made adaptively, for example, \cite{washio1997krylov}
proposes $\lambda=\epsilon \max(\diag(R_k^T R_k))$, with $\epsilon=10^{-16}$.

\begin{figure}[h]
\centering
\includegraphics[width=.49\textwidth]{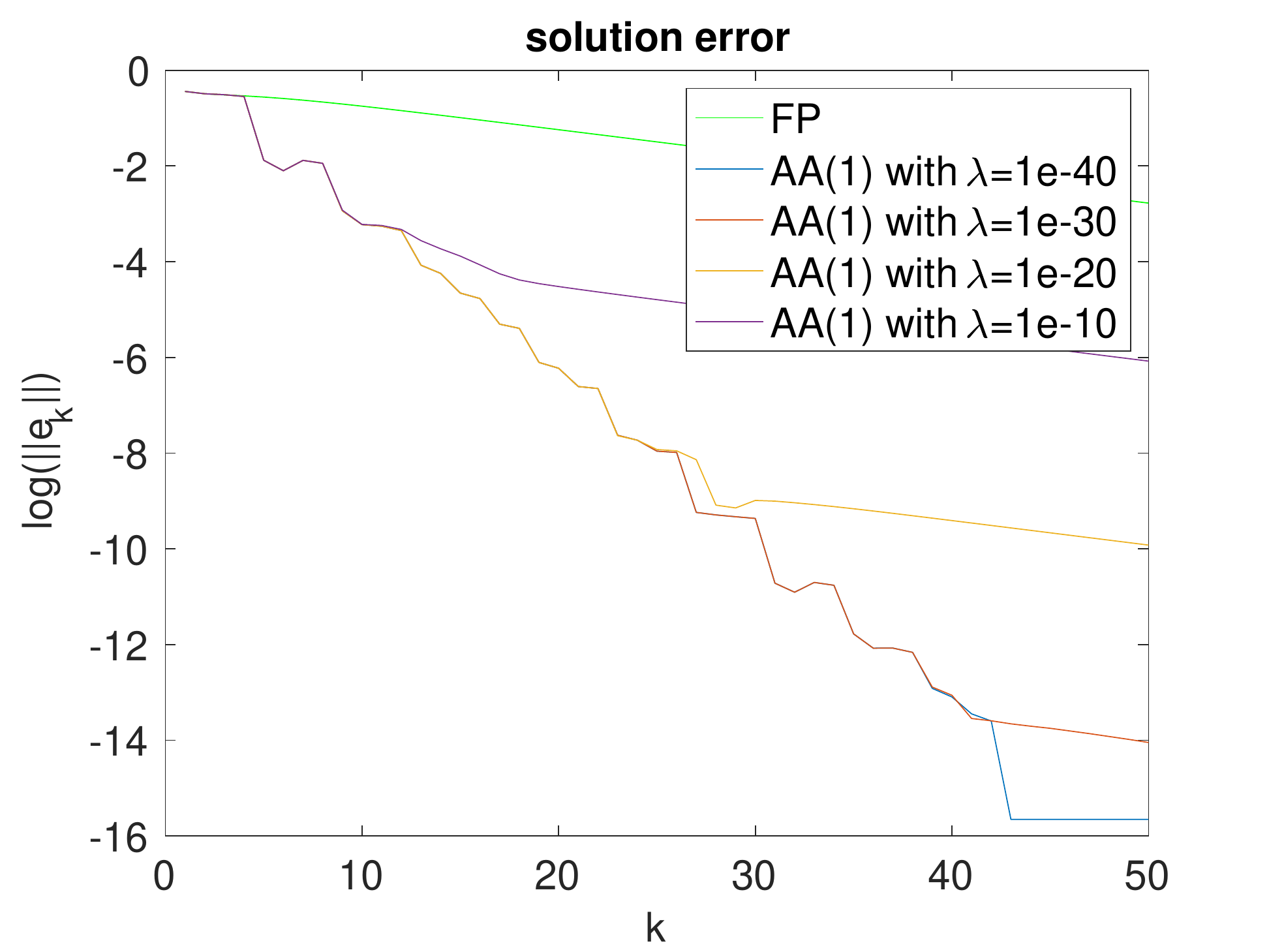}
\includegraphics[width=.49\textwidth]{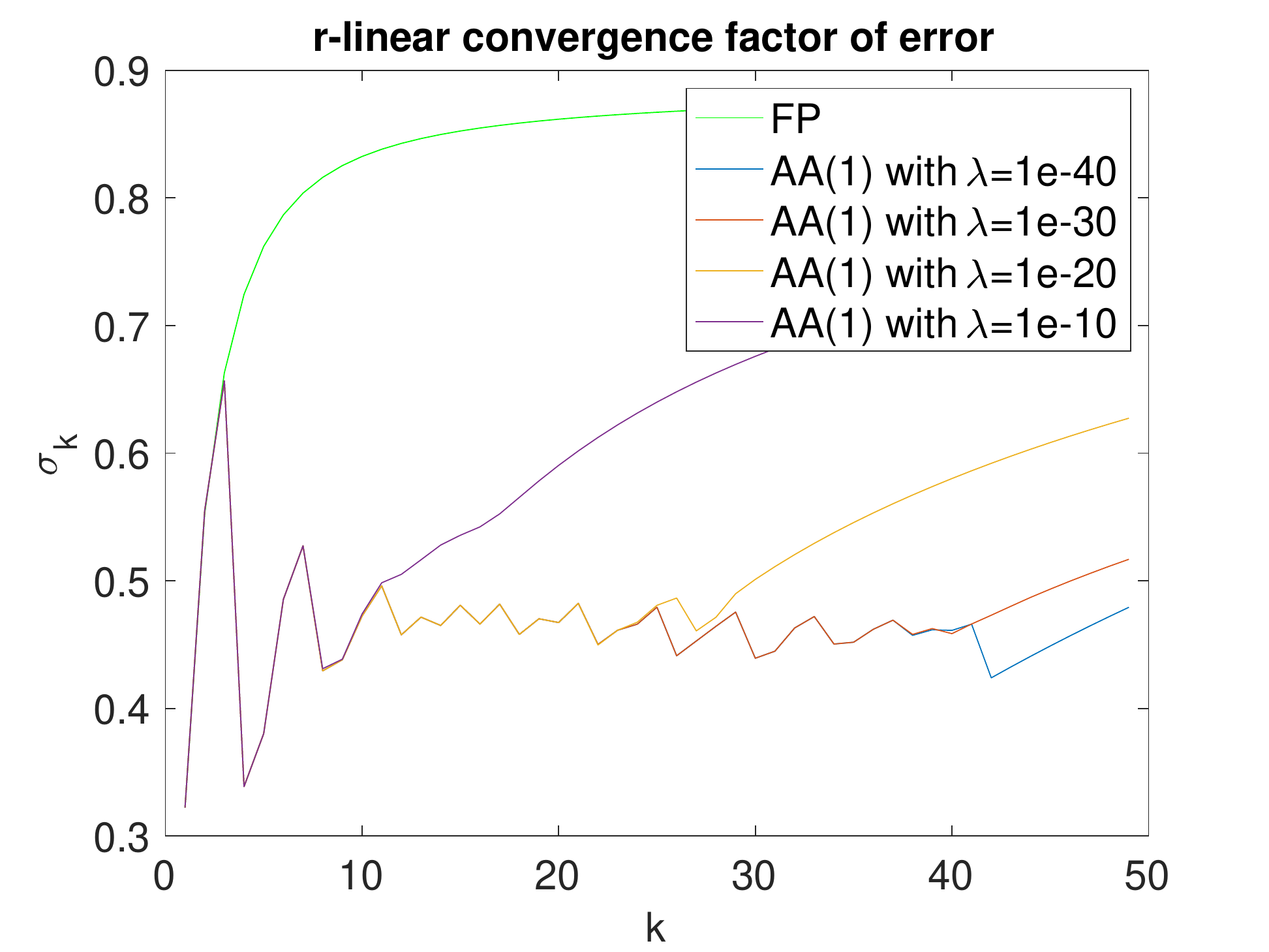}
\caption{AA(1) applied to linear problem \cref{eq:q-linear-2x2-simple-RHS} with least-squares problem \cref{eq:Andersonbetas} approximately
solved using regularized normal equation system \cref{eq:AAm-beta-form-reg}, for various values of the regularization parameter $\lambda$. (left panel) Decimal logarithm of the 2-norm of the solution error. (right panel) r-linear convergence factor of the error.
}
\label{sweep}
\end{figure}

\cref{sweep} shows that, when using regularization, it is crucial to choose $\lambda$ sufficiently small, because choosing
$\lambda$ too large makes the asymptotic convergence factor of AA($m$) revert to the FP factor $\rho(q'(x^*))$, which means
that AA($m$) completely loses its acceleration effect asymptotically as soon as the error is sufficiently small.
Unless the regularization parameter is chosen carefully, regularization, thus, may defeat the purpose of Anderson acceleration
in the asymptotic limit, which is an important case in practice when accurate solutions are required.

It is also worthwhile to point out that, in exact arithmetic, the case of rank-deficient $R_k^T R_k$ is handled properly by the pseudo-inverse formula of \cref{eq:AAm-beta-form-pseudo} which computes the minimum-norm solution when the system is singular. The analysis 
in \cite{LinearacAA} also sheds further light on AA($m$) convergence as $\{x_k\}$ approaches $x^*$. It writes
AA($m$) as the fixed-point method 
\begin{equation}\label{eq:Psi}
\boldsymbol{z}_{k+1}=\Psi(\boldsymbol{z}_k),
\end{equation}
where $\boldsymbol{z} \in\mathbb{R}^{n(m+1)}$ is an augmented state vector composed of the $m+1$ most recent iterates.
The work in \cite{LinearacAA} then analyzes the smoothness properties of the AA($m$) iteration function $\Psi(\boldsymbol{z})$ and the acceleration coefficients $\boldsymbol{\beta}(\boldsymbol{z})$, showing that $\boldsymbol{\beta}(\boldsymbol{z})$ is not continuous at $\boldsymbol{z}^*$, which essentially leads to oscillatory $\boldsymbol{\beta}^{(k)}$ as shown in \cref{AA-Linear-plot}. However, $\Psi(\boldsymbol{z})$, which contains products of the $\beta_{i}^{(k)}$ with terms of the form $q(x_k)-q(x_{k-i})$, is
continuous and Gateaux-differentiable at $\boldsymbol{z}^*$, so the discontinuity of $\boldsymbol{\beta}(\boldsymbol{z})$ does
not preclude convergence of $\{x_k\}$ to $x^*$.

In summary, we find in our numerical experiments that, when the least-squares problem is solved using robust techniques, no regularization of type (\ref{eq:AAm-beta-form-reg}) is needed in practice and AA($m$) sequences $\{x_k\}$ properly converge in a finite number of steps to an approximation of $x^*$ that is numerically exactly a fixed point of \cref{eq:fixed-point}, maintaining the accelerated asymptotic convergence speed up to the point where machine accuracy is reached.
}

\rtxt{
Next we want to verify that AA(1) implemented with Matlab's $QR$ solver leads to backward-stable numerical results that can be
expected to faithfully reflect the theoretical mathematical properties of AA($m$) that were derived in this paper.
Specifically, we explore the normwise relative backward error (NRBE) \cite{Paigebackward,higham1992backward,golub2013matrix} for AA(1).
For a linear system $Ax=b$ and a given approximation $x_k$,  the NRBE is defined as 
\begin{equation}\label{eq:NRBE-definition}
\chi(x_k) = \frac{\|r_k\|}{\|b\| + \|A\| \|x_k\|},
\end{equation}
where $r_k=b-Ax_k$.
When solving $Ax=b$ iteratively, one desires $\chi(x_k)$ to decrease with $k$ until it eventually approaches machine accuracy.

We consider AA(1) applied to the linear system $Ax=Ab$ of \eqref{eq:q-linear-2x2-simple-RHS} with $A=I-M$, solving the AA(1) least-squares problems by Matlab's $QR$ factorization code. The initial guess is $x_0 = (1.2, 1.3)^T$.  The NRBE \cref{eq:NRBE-definition} for solving $Ax=Ab$ is given by
\begin{equation*}
\chi(x_k) = \frac{\| r_k\|}{\|Ab\| + \|A\| \|x_k\|},
\end{equation*}
where $r_k=A x_k -Ab$.
Similarly, for the least-squares problem \cref{eq:Andersonbetas} of AA(1) in the $k$-th step, the NRBE is given by
 \begin{equation*}
\chi(\beta_k) = \frac{\|R_k^T R_k \beta_k+ R_k^T r_k\|}{\|R_k^T r_k\| + \|R_k^T R_k\| \|\beta_k\|},
\end{equation*}
where $r_k=A x_k -Ab$.
 
\cref{fig-NRBE} shows the NRBE for the consecutive least-squares problems solved in each AA(1) iteration $k$, and the NRBE for the linear system as a function of the iteration number $k$. We see that the NRBE for solving the least-squares problem in each AA(1) iteration is of the order of $10^{-15}$, and
the NRBE for solving the linear system decreases steadily to about $10^{-15}$ as the AA(1) iteration count increases.
This confirms that the Matlab $QR$ code solves the least-squares problems in a robust manner, and the
AA(1) iteration obtains an accurate result with small backward error.

\begin{figure}[h]
\centering
\includegraphics[width=.49\textwidth]{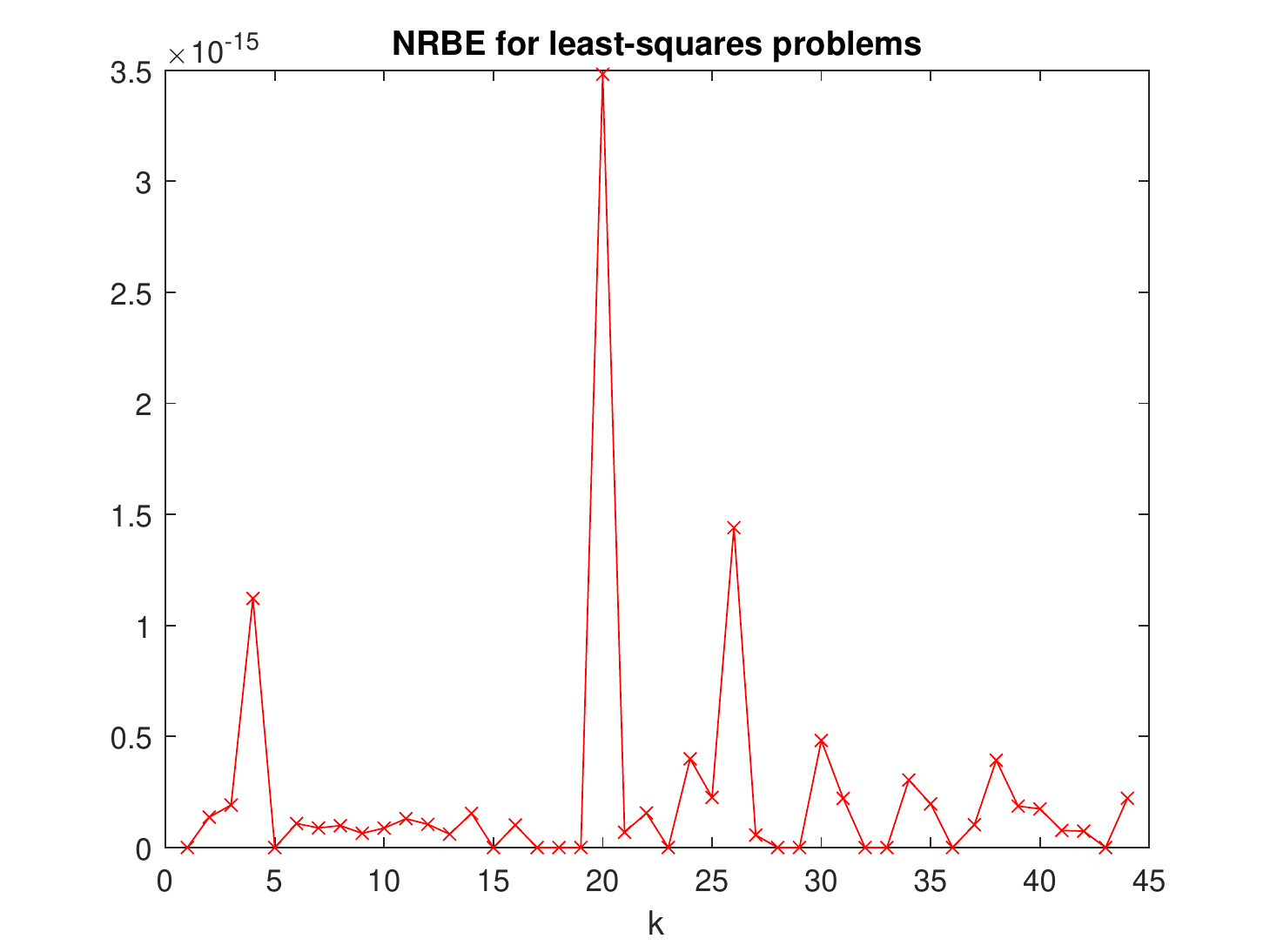}
\includegraphics[width=.49\textwidth]{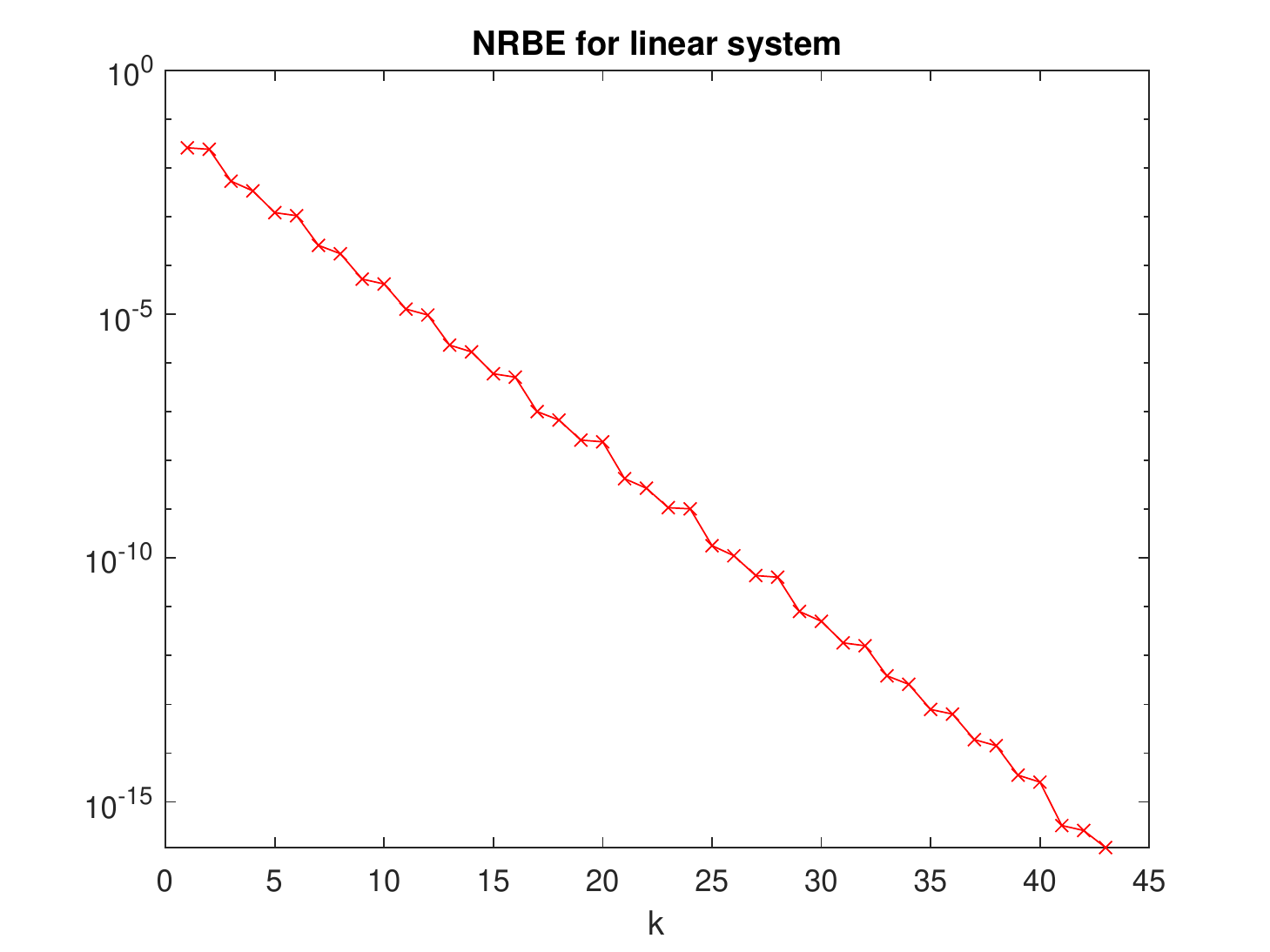}
\caption{NRBE for AA(1) applied to linear problem \eqref{eq:q-linear-2x2-simple-RHS} with the least-squares problems 
solved by Matlab's $QR$ factorization code. (left panel) NRBE for the consecutive least-squares problems solved in each AA(1) iteration $k$. (right panel) NRBE for the linear system as a function of the iteration number $k$.
}
\label{fig-NRBE}
\end{figure}

 }

\subsection{AA(1) for a nonlinear system}
Although our analysis in this paper is for AA($m$) applied to linear iteration \cref{eq:linear-fixed-point}, it is interesting to explore, finally, how the AA(1) convergence patterns we observe and the understanding provided by our theoretical results may extend to a nonlinear system.
\begin{problem}\label{prob:nonlinear2x2}
Consider the nonlinear system
\begin{align}
x_2= x_1^2 \label{eq:exm1-1}\\
x_1+(x_1-1)^2 +x_2^2 =1 \label{eq:exm1-2}
\end{align}
with solution $(x_1^*,x_2^*) = (0,0)$.
Let $x=[x_1 \ x_2]^T$ and define the FP iteration function
\begin{equation*}
  q(x) = \begin{bmatrix} \ds \frac{1}{2}(x_1+x_1^2+x_2^2) \\ \\ \ds \frac{1}{2}(x_2+x_1^2) \end{bmatrix},
\end{equation*}
with Jacobian matrix
\begin{equation*}
q'(x)=
\begin{bmatrix}
x_1+\ds \frac{1}{2}& x_2\\
x_1 & \ds \frac{1}{2}
\end{bmatrix}.
\end{equation*}
We have
\begin{equation*}
  q'(x^*) = \begin{bmatrix}
  \ds \frac{1}{2}& 0 \\
0 & \ds \frac{1}{2}
  \end{bmatrix}, \ \textrm{and} \quad
\rho(q'(x^*)) =\ds \frac{1}{2}<1.
\end{equation*}
\end{problem}

Monte Carlo results with a large number of random initial guesses for the nonlinear \cref{prob:nonlinear2x2}, with FP and AA(1) with $x_1=q(x_0)$, are shown in  \cref{prob3-mc}. From \cref{prob3-mc}, we see that the nonlinear convergence behavior is qualitatively similar to that of the linear case shown in \cref{prob2-mc}, that is, the AA(1) sequences $\{x_k\}$ converge $r$-linearly, while the $r$-linear convergence factors $\rho_{\{x_k\}}$ depend strongly on the initial guess. An upper bound $\rho_{AA(1),x^*}$ seems to exist for $\rho_{\{x_k\}}$ for the AA(1) iteration which is smaller than $\rho_{q,x^*}=1/2$ of fixed-point iteration \cref{eq:fixed-point} by itself.  It is clear that the $\beta_k$ sequences oscillate for this nonlinear problem as the AA(1) iteration approaches $x^*$, and  $\beta_k>-1$ even though we do not have theoretical  results on this for the nonlinear case.

In \cref{prob3-grid}, to be compared with \cref{AA-Linear-plot}, we show the convergence factor of AA(1) applied to the nonlinear \cref{prob:nonlinear2x2} for different initial guesses, where we take $x_0$ on a uniform grid with 101 by 101 points.
It is interesting to see, for the nonlinear problem of \cref{prob3-grid}, that there are also preferred directions with fast convergence for the initial condition near the solution $(0,0)^T$, as in the linear case. Future work will investigate this further, possibly using eigenvectors obtained after linearization.

\begin{figure}[h]
\centering
\includegraphics[width=.49\textwidth]{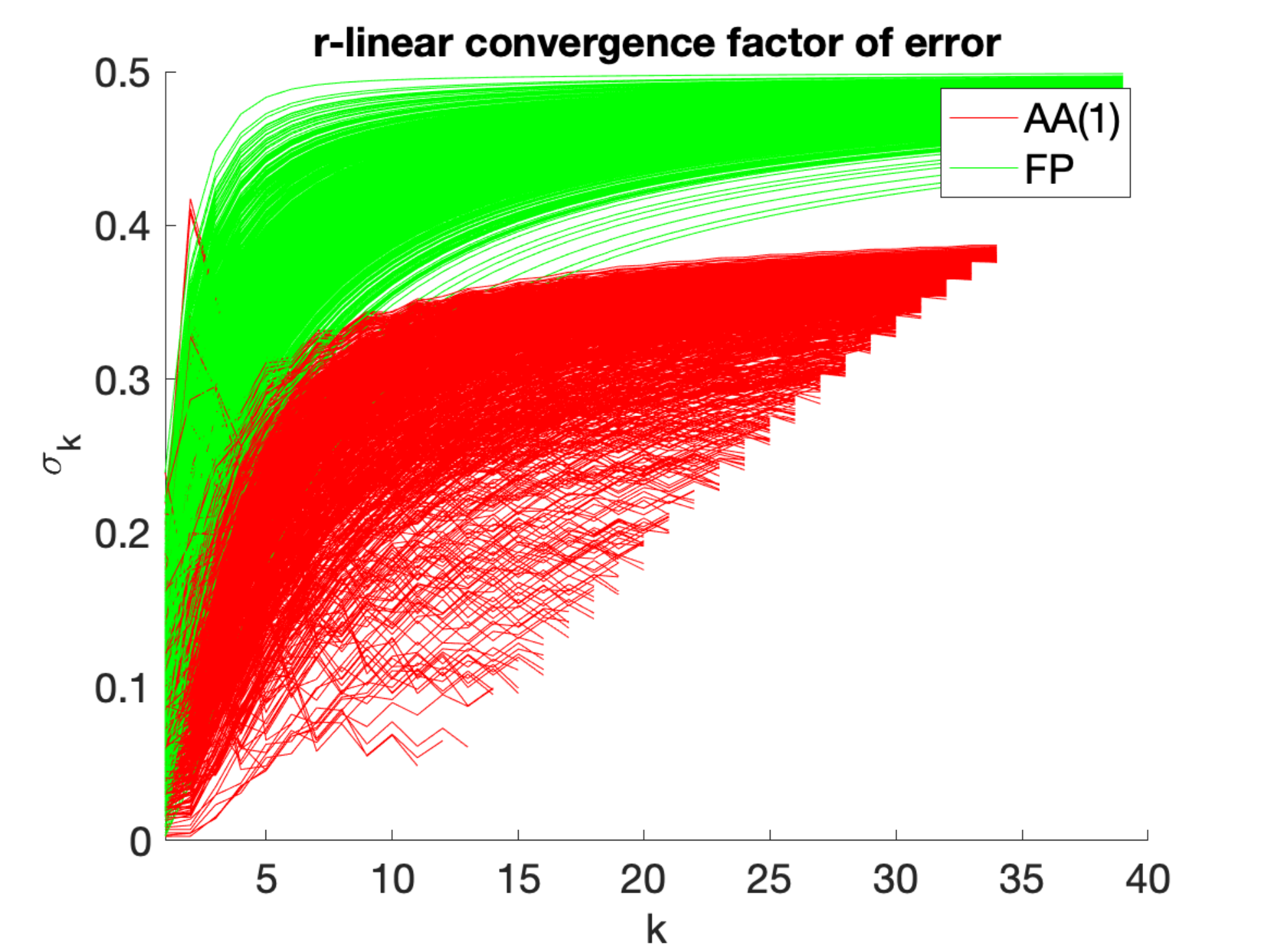}
\includegraphics[width=.49\textwidth]{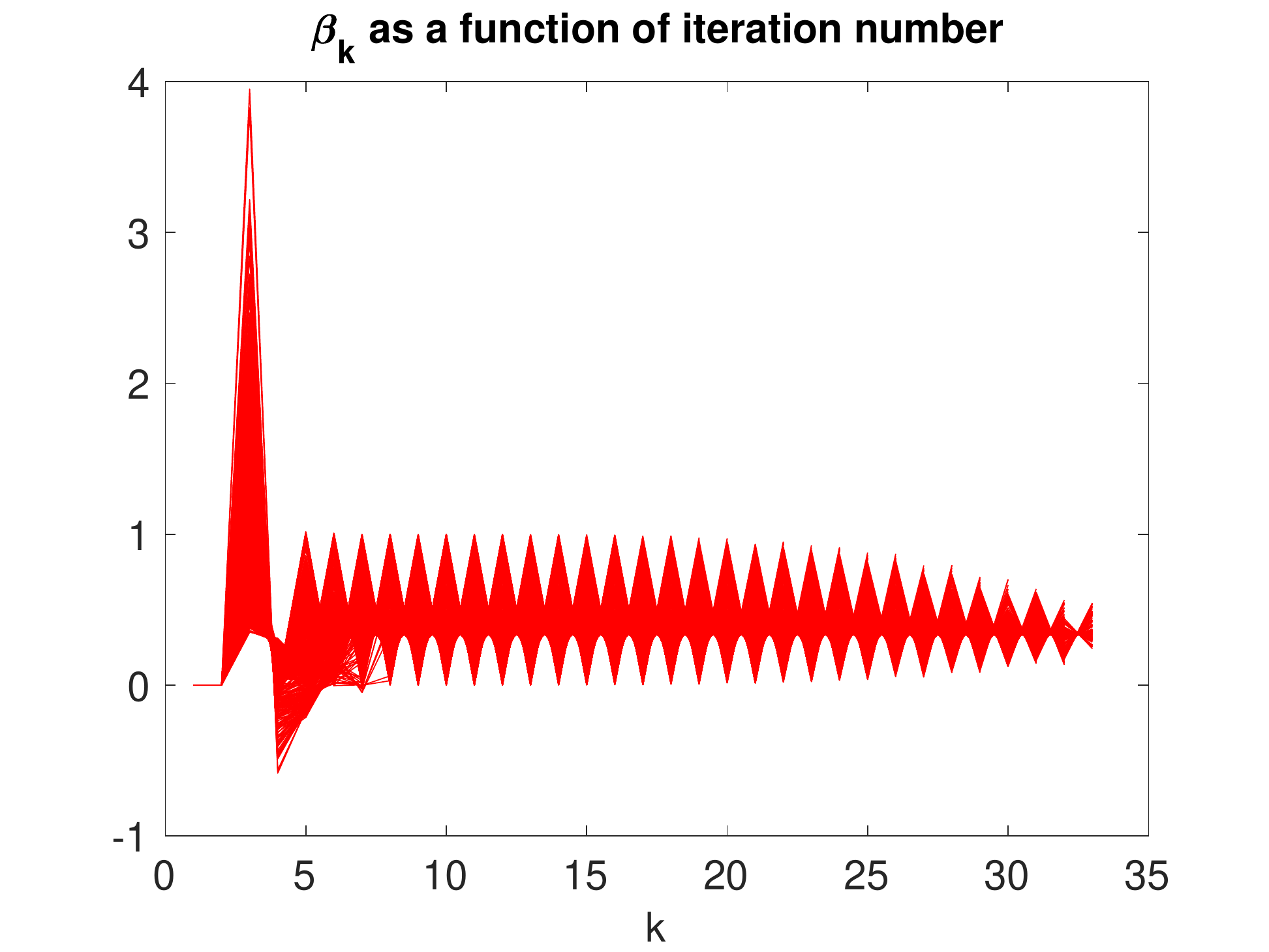}
\caption{Monte Carlo tests for \cref{prob:nonlinear2x2} (nonlinear). (Results from \cite{LinearacAA}.)} \label{prob3-mc}
\end{figure}

\begin{figure}[h]
\centering
\includegraphics[width=.7\textwidth]{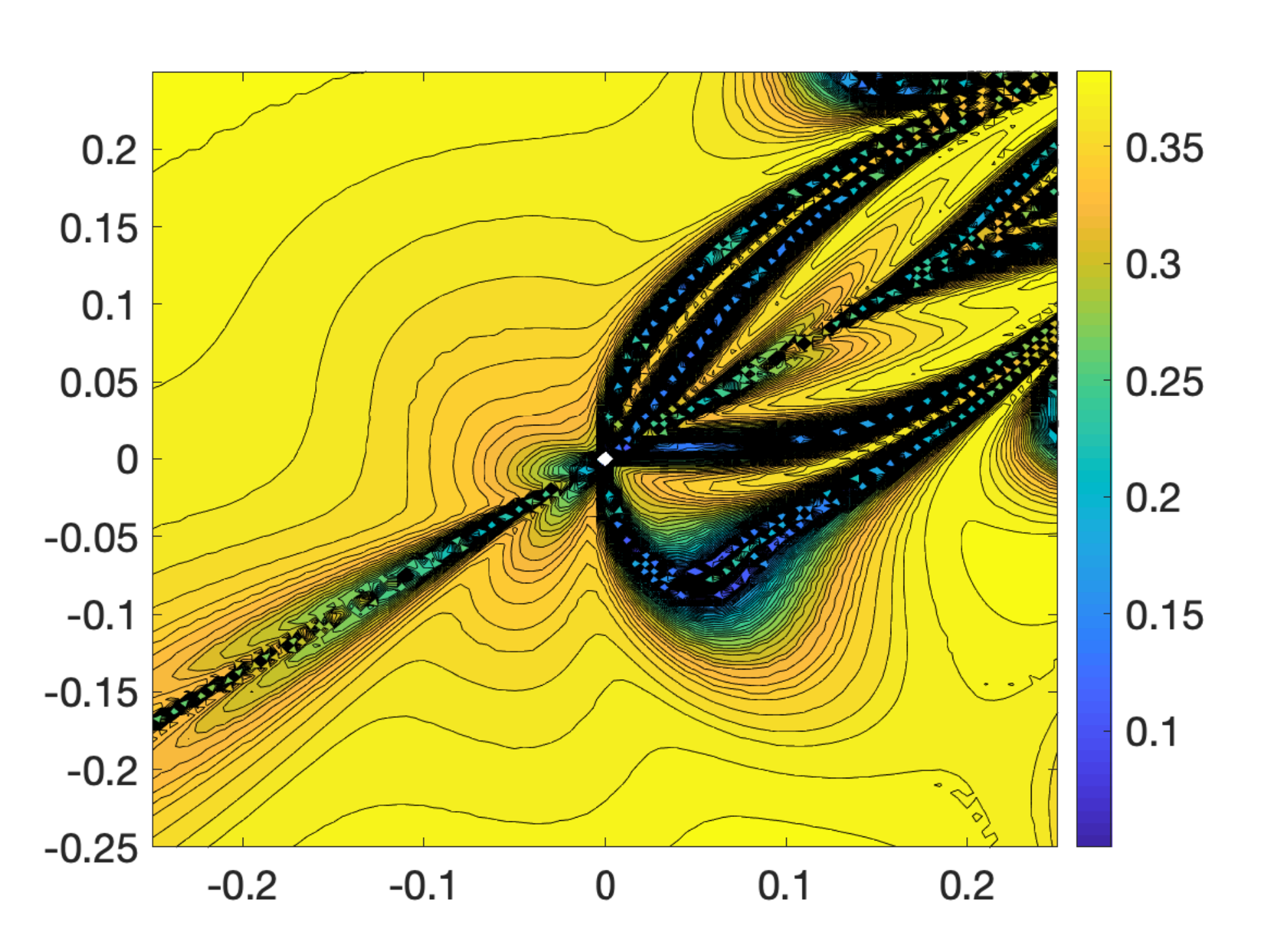}
\includegraphics[width=.7\textwidth]{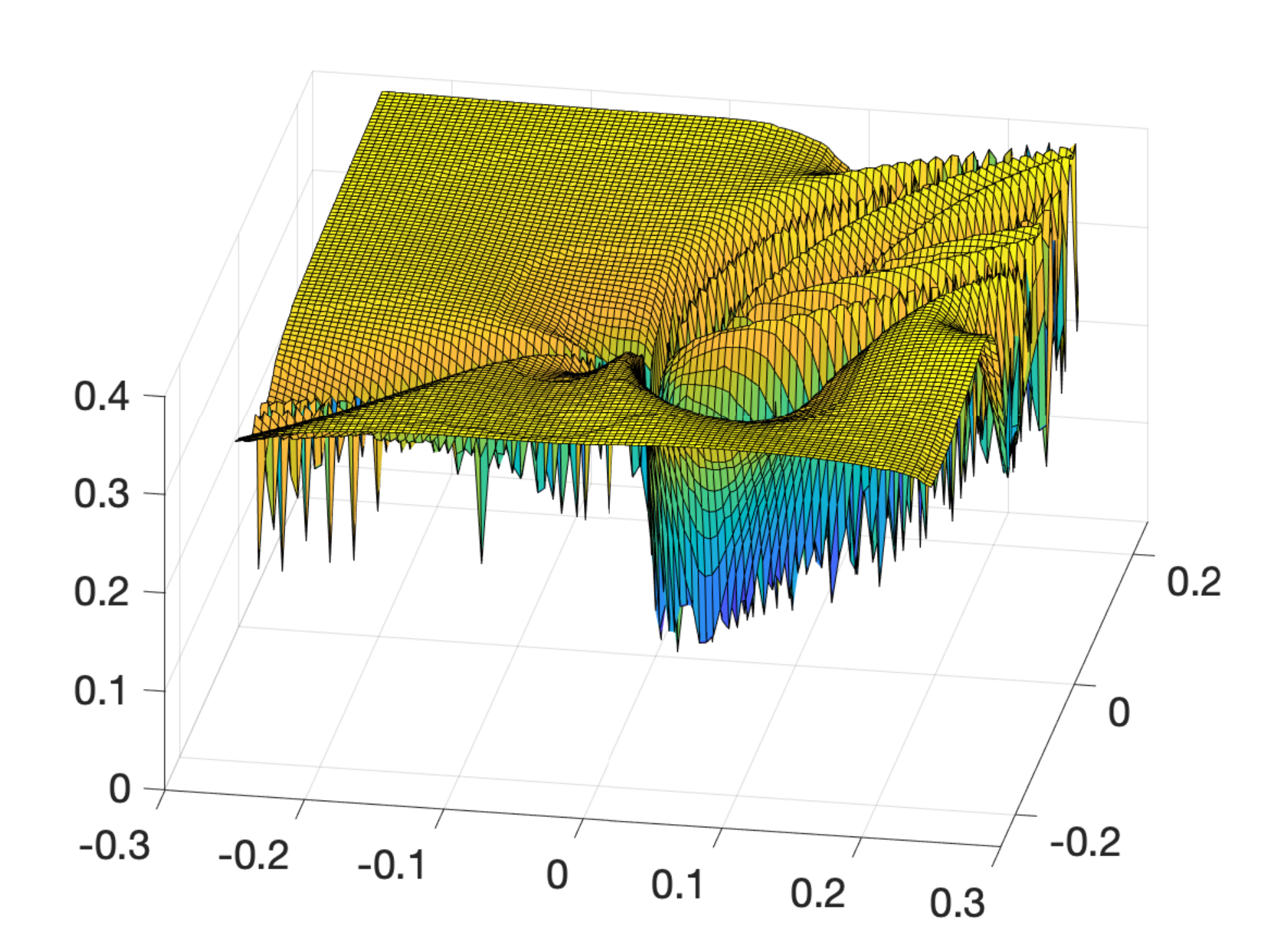}
\caption{Asymptotic convergence factor as a function of initial condition for \cref{prob:nonlinear2x2} (nonlinear).} \label{prob3-grid}
\end{figure}

\section{Conclusion}\label{sec:con}
%
\btxt{
In this paper, we have derived new theoretical results for the AA($m$) method, i.e., Anderson acceleration with window size $m$, applied to the case of linear fixed-point iterations $x_{k+1}=M x_{k}+b$, as a roadway towards improving our understanding of convergence acceleration by AA($m$), which has many open questions.

Writing AA($m$) as a Krylov method with polynomial residual update formulas, we have derived new $(m+2)$-term recurrence relations for the AA($m$) polynomials. This leads to several insights and further results that include a periodic memory effect for the AA($m$) residual polynomials, orthogonality relations, a lower bound on the AA(1) acceleration coefficient $\beta_k$, and explicit nonlinear recursions for the AA(1) residuals and residual polynomials that do not include the acceleration coefficient $\beta_k$. 
\rtxt{
Using these recurrence relations we have also derived new residual convergence bounds for AA(1) in the linear case, showing how the residual reduction in a given iteration varies strongly as a function of the residual reduction in the previous iteration and the angle between the two previous residual vectors. 
}
We have proved results on the invariance of the AA(1) asymptotic convergence factor under scaling of the initial condition, and on finite convergence for eigenvector initial conditions. Extensive numerical results have illustrated how these new theoretical results help to understand various aspects of the convergence behaviour of AA(1) in the linear case. 

One potential avenue for further research is to analyze finite-precision stability aspects of AA($m$) such as orthogonality loss and backward error analysis, like has been done extensively for, e.g., GMRES, see \cite{Paigebackward,liesen2013krylov,PaigeMGS}. This would depend on the least-squares solution method used. It would also be interesting to extend some of this analysis to AA($m$) for the nonlinear case. Various practical aspects of GMRES implementations could also be extended to AA($m$), for example, rescaling the initial guess to have a small initial residual \cite{Paigebackward,liesen2013krylov}.

Theoretical results on determining or bounding AA($m$) asymptotic convergence factors remain a difficult open problem.
The situation for windowed AA($m$) is somewhat similar to restarted GMRES($m$), where few theoretical results are known.
Even for non-restarted GMRES, practically usable asymptotic convergence results tend to be only available for well-defined
classes of well-behaved matrices, and the resulting bounds may strongly depend on the problem class. Similarly, it is likely
that practically useful asymptotic convergence results that may be derived for AA($m$) will also be problem-dependent.

For the case of AA($m$) with stationary coefficients $\beta_i$ in \cref{eq:AA-iteration}, \cite{desterck2020,wang2020} were able to compute the convergence improvement that results from the optimal stationary AA($m$) iteration, because the AA($m$) iteration function is differentiable in
the stationary case and the convergence factor can be computed as the spectral radius of the Jacobian of the AA($m$) iteration function evaluated at the fixed point.
However, for the non-stationary AA($m$) that is widely used in science and engineering applications, the AA($m$) iteration function
in (\ref{eq:Psi}) is not differentiable and it is not known how to determine the $r$-linear convergence factor $\rho_{ AA(m),x^*}$.
It is, hence, not known how to compute by how much $\rho_{ AA(m),x^*}$ improves upon $\rho_{q,x^*}$, even for the linear case with 2$\times$2 matrices $M=I-A$. In \Cref{sec:AA1-bounds} we have derived new results on AA(1) convergence bounds and the dependence of the residual reduction on the residual reduction in the previous iteration and the angle between consecutive previous residuals. Along with numerical results as in \Cref{fig:AA1_bounds_exposition}, this sheds interesting light on quasi-periodic AA(1) convergence patterns observed for symmetric systems. These patterns show some interesting similarities to convergence of restarted GMRES($m$) in this case, which is only partially understood. These results also indicate the difficulty in obtaining practically useful estimates of AA(1) asymptotic convergence factors from per-iteration convergence bounds. 

It is fair to say that Anderson acceleration has proved very effective in many application areas in science, engineering and machine learning, but our understanding of its convergence properties is still far from complete. For this reason, further extending our theoretical understanding of the convergence acceleration provided by the AA($m$) iteration, possibly building on some of the developments made and insights gained in this paper, remains an important topic of further research.}

\appendix

\btxt{
\section{Some results for AA($m$) with general initial guess}\label{app:AAm-random-guesses}
%
In this appendix we present results that lead to the proof of \cref{pro:polynomial-form-AAm-AAj} in \Cref{sec:AAm-krylov}.
We first derive a result, following from expression (\ref{eq:m+1-term-rk-linear-random-guess}), on writing the residual of the more general AA iteration (\ref{eq:general-guess-AAm}) with general initial guess $\{x_0,x_1,\ldots, x_m\}$ as a sum of $m+1$ vectors which are in $m+1$ Krylov spaces, $\{ \mathcal{K}_s(M,r_j)\}_{j=0}^m$ generated by the $m+1$ initial residuals $r_0,r_1,\ldots, r_m$. We also derive recurrence relations for the polynomials that arise in this expression. \cref{pro:polynomial-form-AAm-random-guess} can then easily be specialized to \cref{pro:polynomial-form-AAm-AAj}.

%
\begin{proposition}\label{pro:polynomial-form-AAm-random-guess}
 AA($m$) iteration \cref{eq:general-guess-AAm} with general initial guess $\{x_j\}_{j=0}^{m}$ applied to linear iteration \cref{eq:linear-fixed-point} is a multi-Krylov method. That is, the residual  can be expressed as
\begin{equation}\label{eq:rk-polynomial-form-M-AAm-random-guess}
    r_{k+1} = \sum_{j=0}^m p_{k-m+1,j}(M)\,r_j,\quad k\geq m,
\end{equation}
where the $p_{k-m+1,j}(\lambda)$  are polynomials of degree at most $k-m+1$  satisfying the following relations:
\begin{align}
  p_{1,j}(\lambda)&=-\beta^{(m)}_{m-j}\lambda, \quad j=0,\ldots,m-1; \qquad  p_{1,m}(\lambda)=\Big(1+\sum_{i=1}^{m}\beta_i^{(m)}\Big)\lambda; \label{AAm-AAj-eq1} \\
  p_{k-m+1,j}(\lambda)&=\lambda  \left(\Big(1+\sum_{i=1}^{m}\beta_i^{(k)}\Big)p_{k-m,j} - \sum_{i=1}^{m}\beta_i^{(k)} p_{k-m-i,j}\right),  \label{AAm-AAj-eq3}\\
  & \qquad \qquad \qquad \qquad \qquad \qquad k-m+1>1, j=0,\ldots, m; \nonumber
\end{align}
where for $i=1,\ldots,m,$ and   $j=0,\ldots, m$,
\begin{equation*}
  p_{1-i,j}(\lambda) =  \begin{cases}
         1                  & \text{if}\,\, i=j+1-m,\\
       0  & \text{otherwise}.
  \end{cases}
\end{equation*}
\end{proposition}
\begin{proof}
The results of \cref{AAm-AAj-eq1} are  obvious from \cref{eq:m+1-term-rk-linear-random-guess}. For \cref{AAm-AAj-eq3}, when $k\geq 2m+1$,  from  \cref{eq:m+1-term-rk-linear-random-guess},  $r_{k+1}$ is a linear combination of $\{r_{k+1-i}\}_{i=1}^{m+1}$ where the smallest subscript index in the residual is $k-m\geq m+1$. Thus, every term $r_{k+1-j}$ can be rewritten as a linear combination of $\{r_j\}_{j=0}^m$. Then, \cref{AAm-AAj-eq3} can be validated easily. As to $k< 2m+1$, since some terms in $\{r_{k+1-i}\}_{i=1}^{m+1}$ do not contain all $\{r_j\}_{j=0}^m$, we require  that  $p_{1-i,j}(\lambda)=0$ or 1 for \cref{AAm-AAj-eq3} to hold.
\end{proof}
\begin{remark}\label{rmk:multilevel-krylov-method}
 Expression \cref{eq:rk-polynomial-form-M-AAm-random-guess} indicates that the residual $r_{k+1}$ with $k\geq m$ of AA($m$) can be decomposed as a sum of $m+1$ vectors which are in $m+1$ Krylov spaces, $\{ \mathcal{K}_s(M,r_j)\}_{j=0}^m$ or $\{ \mathcal{K}_s(A,r_j)\}_{j=0}^m$, where $s=k-m+2$. Therefore, we can refer to  AA($m$) with general initial guess as a {\em{multi-Krylov space method}}. Note that, in the case of the usual AA($m$) iteration (\ref{eq:AA-iteration}) with one initial guess $x_0$, each $r_j\in \{r_j\}_{j=1}^m$ can, by \cref{pro:polynomial-form-AAm-random-guess}, be expressed as a polynomial in $M$ applied to $r_0$, so AA($m$) is a Krylov space method, that is, $r_{k+1}\in \mathcal{K}_s (M,r_0)$, as formalized in \cref{pro:polynomial-form-AAm-AAj} in \Cref{sec:AAm-krylov}.
\end{remark}

\begin{remark}
In \cref{pro:polynomial-form-AAm-random-guess}, we can, if desired, also rewrite the residual in terms of polynomials in the matrix $A$:
\begin{equation}\label{eq:rk-polynomial-form-A-AAm-random-guess}
    r_{k+1} = \sum_{j=0}^m\widetilde{ p}_{k-m+1,j}(A)\,r_j,\quad k\geq m,
  \end{equation}
where the $\widetilde{p}_{k-m+1,j}(\lambda)$  satisfy the  relations:
\begin{align*}
\widetilde{p}_{1,j}(\lambda)&=-\sum_{i=1}^{m}\beta_i^{(m)}(1-\lambda), \,\, j=0,\ldots,m-1; \\
  \widetilde{p}_{1,m}(\lambda)&=\Big(1+\sum_{i=1}^{m}\beta_i^{(m)}\Big)(1-\lambda); \\
  \widetilde{p}_{k-m+1,j}(\lambda)&=(1-\lambda)\left(\Big(1+\sum_{i=1}^{m}\beta_i^{(k)}\Big) \widetilde{p}_{k-m,j} - \sum_{i=1}^{m}\beta_i^{(k)} \widetilde{p}_{k-m-i,j}\right), \\
  & \qquad \qquad \qquad \qquad \qquad \qquad k-m+1>1, j=0,\ldots, m;
\end{align*}
where for $i=1,\ldots,m,  j=0,\ldots, m$,
\begin{equation*}
 \widetilde{p}_{1-i,j}(\lambda) =  \begin{cases}
         1                  & \text{if}\,\, i=j+1-m,\\
       0  & \text{otherwise}.
  \end{cases}
\end{equation*}
\end{remark}
%

}

\bibliographystyle{siamplain}
\bibliography{AArefII}

\begin{thebibliography}{10}

\bibitem{an2017anderson}
{\sc H.~An, X.~Jia, and H.~F. Walker}, {\em Anderson acceleration and
  application to the three-temperature energy equations}, Journal of
  Computational Physics, 347 (2017), pp.~1--19.

\bibitem{anderson1965iterative}
{\sc D.~G. Anderson}, {\em Iterative procedures for nonlinear integral
  equations}, Journal of the ACM (JACM), 12 (1965), pp.~547--560.

\bibitem{Baker_etal_2005}
{\sc A.~H. Baker, E.~R. Jessup, and T.~Manteuffel}, {\em {A Technique for
  Accelerating the Convergence of Restarted GMRES}}, SIAM Journal on Matrix
  Analysis and Applications, 26 (2005), pp.~962--984.

\bibitem{brune2015composing}
{\sc P.~R. Brune, M.~G. Knepley, B.~F. Smith, and X.~Tu}, {\em Composing
  scalable nonlinear algebraic solvers}, SIAM Review, 57 (2015), pp.~535--565.

\bibitem{sterck2012nonlinear}
{\sc H.~{De Sterck}}, {\em A nonlinear {GMRES} optimization algorithm for
  canonical tensor decomposition}, SIAM J. Scientific Computing, 34 (2012),
  pp.~A1351--A1379.

\bibitem{sterck2013steepest}
{\sc H.~De~Sterck}, {\em Steepest descent preconditioning for nonlinear {GMRES}
  optimization}, Numerical Linear Algebra with Applications, 20 (2013),
  pp.~453--471.

\bibitem{LinearacAA}
{\sc H.~De~Sterck and Y.~He}, {\em Linear asymptotic convergence of {A}nderson
  acceleration: fixed-point analysis}, arXiv:2109.14176,  (2021).

\bibitem{desterck2020}
{\sc H.~De~Sterck and Y.~He}, {\em On the asymptotic linear convergence speed
  of {Anderson} acceleration, {N}esterov acceleration and nonlinear {GMRES}},
  SIAM Journal on Scientific Computing,  (2021), pp.~S21--S46.

\bibitem{evans2020proof}
{\sc C.~Evans, S.~Pollock, L.~G. Rebholz, and M.~Xiao}, {\em A proof that
  {Anderson Acceleration} improves the convergence rate in linearly converging
  fixed-point methods (but not in those converging quadratically)}, SIAM
  Journal on Numerical Analysis, 58 (2020), pp.~788--810.

\bibitem{fang2009two}
{\sc H.-r. Fang and Y.~Saad}, {\em Two classes of multisecant methods for
  nonlinear acceleration}, Numerical Linear Algebra with Applications, 16
  (2009), pp.~197--221.

\bibitem{fu2020anderson}
{\sc A.~Fu, J.~Zhang, and S.~Boyd}, {\em Anderson accelerated douglas--rachford
  splitting}, SIAM Journal on Scientific Computing, 42 (2020),
  pp.~A3560--A3583.

\bibitem{golub2013matrix}
{\sc G.~H. Golub and C.~F. Van~Loan}, {\em Matrix computations}, JHU press,
  2013.

\bibitem{henderson2019damped}
{\sc N.~C. Henderson and R.~Varadhan}, {\em Damped {Anderson} acceleration with
  restarts and monotonicity control for accelerating {EM and EM-like}
  algorithms}, Journal of Computational and Graphical Statistics, 28 (2019),
  pp.~834--846.

\bibitem{higham1992backward}
{\sc D.~J. Higham and N.~J. Higham}, {\em Backward error and condition of
  structured linear systems}, SIAM Journal on Matrix Analysis and Applications,
  13 (1992), pp.~162--175.

\bibitem{ho2017accelerating}
{\sc N.~Ho, S.~D. Olson, and H.~F. Walker}, {\em Accelerating the {U}zawa
  algorithm}, SIAM Journal on Scientific Computing, 39 (2017), pp.~S461--S476.

\bibitem{kindermann2021optimal}
{\sc S.~Kindermann}, {\em Optimal-order convergence of {N}esterov acceleration
  for linear ill-posed problems}, Inverse Problems,  (2021).

\bibitem{liesen2013krylov}
{\sc J.~Liesen and Z.~Strakos}, {\em Krylov subspace methods: principles and
  analysis}, Oxford University Press, 2013.

\bibitem{lipnikov2013anderson}
{\sc K.~Lipnikov, D.~Svyatskiy, and Y.~Vassilevski}, {\em Anderson acceleration
  for nonlinear finite volume scheme for advection-diffusion problems}, SIAM
  Journal on Scientific Computing, 35 (2013), pp.~A1120--A1136.

\bibitem{liu2018parametrized}
{\sc C.~Liu and M.~Belkin}, {\em Parametrized accelerated methods free of
  condition number}, arXiv preprint arXiv:1802.10235,  (2018).

\bibitem{lockhart2022performance}
{\sc S.~Lockhart, D.~J. Gardner, C.~S. Woodward, S.~Thomas, and L.~N. Olson},
  {\em Performance of low synchronization orthogonalization methods in
  {A}nderson accelerated fixed point solvers}, in Proceedings of the 2022 SIAM
  Conference on Parallel Processing for Scientific Computing, SIAM, 2022,
  pp.~49--59.

\bibitem{lott2012accelerated}
{\sc P.~Lott, H.~Walker, C.~Woodward, and U.~Yang}, {\em An accelerated
  {P}icard method for nonlinear systems related to variably saturated flow},
  Advances in Water Resources, 38 (2012), pp.~92--101.

\bibitem{moler2004numerical}
{\sc C.~B. Moler}, {\em Numerical computing with MATLAB}, SIAM, 2004.

\bibitem{ni2009anderson}
{\sc P.~Ni}, {\em Anderson acceleration of fixed-point iteration with
  applications to electronic structure computations}, PhD thesis, Worcester
  Polytechnic Institute, 2009.

\bibitem{niu2020momentum}
{\sc C.~Niu and X.~Hu}, {\em Momentum accelerated multigrid methods}, arXiv
  preprint arXiv:2006.16986,  (2020).

\bibitem{oosterlee2000krylov}
{\sc C.~Oosterlee and T.~Washio}, {\em Krylov subspace acceleration of
  nonlinear multigrid with application to recirculating flows}, SIAM J.
  Scientific Computing, 21 (2000), pp.~1670--1690.

\bibitem{PaigeMGS}
{\sc C.~C. Paige, M.~Rozlo\v{z}n\'{\i}k, and Z.~Strako\v{s}}, {\em Modified
  {G}ram-{S}chmidt ({MGS}), least squares, and backward stability of
  {MGS}-{GMRES}}, SIAM J. Matrix Anal. Appl., 28 (2006), pp.~264--284,
  \url{https://doi.org/10.1137/050630416},
  \url{https://doi.org/10.1137/050630416}.

\bibitem{Paigebackward}
{\sc C.~C. Paige and Z.~Strako\v{s}}, {\em Residual and backward error bounds
  in minimum residual {K}rylov subspace methods}, SIAM J. Sci. Comput., 23
  (2002), pp.~1898--1923, \url{https://doi.org/10.1137/S1064827500381239},
  \url{https://doi.org/10.1137/S1064827500381239}.

\bibitem{pollock2021anderson}
{\sc S.~Pollock and L.~G. Rebholz}, {\em Anderson acceleration for contractive
  and noncontractive operators}, IMA Journal of Numerical Analysis, 41 (2021),
  pp.~2841--2872.

\bibitem{potra2013characterization}
{\sc F.~A. Potra and H.~Engler}, {\em A characterization of the behavior of the
  {A}nderson acceleration on linear problems}, linear Algebra and its
  Applications, 438 (2013), pp.~1002--1011.

\bibitem{saad1986gmres}
{\sc Y.~Saad and M.~H. Schultz}, {\em Gmres: A generalized minimal residual
  algorithm for solving nonsymmetric linear systems}, SIAM Journal on
  scientific and statistical computing, 7 (1986), pp.~856--869.

\bibitem{scieur2016regularized}
{\sc D.~Scieur, A.~d'Aspremont, and F.~Bach}, {\em Regularized nonlinear
  acceleration}, in Advances In Neural Information Processing Systems, 2016,
  pp.~712--720.

\bibitem{toth2015}
{\sc A.~Toth and C.~T. Kelley}, {\em Convergence analysis for {A}nderson
  acceleration}, SIAM J. Numer. Anal., 53 (2015), pp.~805--819.

\bibitem{walker2011anderson}
{\sc H.~F. Walker and P.~Ni}, {\em Anderson acceleration for fixed-point
  iterations}, SIAM Journal on Numerical Analysis, 49 (2011), pp.~1715--1735.

\bibitem{wang2020}
{\sc D.~Wang, Y.~He, and H.~De~Sterck}, {\em On the asymptotic linear
  convergence speed of {Anderson} acceleration applied to {ADMM}}, Journal of
  Scientific Computing, 88:38 (2021).

\bibitem{washio1997krylov}
{\sc T.~Washio and C.~W. Oosterlee}, {\em Krylov subspace acceleration for
  nonlinear multigrid schemes}, Electronic Transactions on Numerical Analysis,
  6 (1997), pp.~3--1.

\end{thebibliography}
\end{document}